\documentclass[11pt,oneside]{amsart}
      \usepackage{amssymb,enumerate,bbm}

      \theoremstyle{plain}
      \newtheorem{theorem}{Theorem}[section]
      \newtheorem{lemma}[theorem]{Lemma}
      \newtheorem{corollary}[theorem]{Corollary}
      
      \theoremstyle{definition}
      \newtheorem{definition}[theorem]{Definition}
      
      \theoremstyle{remark}
      \newtheorem{remark}[theorem]{Remark}
			
			\theoremstyle{proposition}
			\newtheorem{proposition}[theorem]{Proposition}

	\renewcommand{\theequation}{\arabic{section}.\arabic{equation}}
   
      \makeatletter
      \def\@setcopyright{}
      
      \def\serieslogo@{}
      \makeatother

\usepackage{graphicx}
\usepackage{caption}	
\usepackage{tikz}
\usetikzlibrary{calc}
\usetikzlibrary{through}
\usetikzlibrary{decorations.markings}

\usepackage{geometry}
\usepackage{marginnote}
\usepackage{verbatim}
\usepackage{color}

\newcommand{\R}{\mathbb{R}}
\newcommand{\C}{\mathbb{C}}

\renewcommand{\l}{\lambda}
\renewcommand{\d}{\delta}
\newcommand{\g}{\gamma}
\renewcommand{\a}{\alpha}
\newcommand{\e}{\varepsilon}

\newcommand{\lv}{\lvert}
\newcommand{\rv}{\rvert}
\newcommand{\lb}{\left(}
\newcommand{\rb}{\right)}

\newcommand{\GAi}{\textbf{(GA)$_1$}}
\newcommand{\GAii}{\textbf{(GA)$_2$}}
\newcommand{\map}[3] {#1:\, #2\, \rightarrow \, #3}
\renewcommand{\vec}[1]{\begin{pmatrix}#1\end{pmatrix}}
\newcommand{\sgn}{\operatorname{sgn}}
\newcommand{\Ai}{\operatorname{Ai}}
\newcommand{\Aik}{\operatorname{\mathbb{A}i}}
\renewcommand{\O}{\mathcal{O}}
\newcommand{\supp}{\operatorname{supp}}
\newcommand{\dopp}[1]{\mathbbmss{#1}}
\renewcommand{\Re}{\operatorname{Re}}
\renewcommand{\Im}{\operatorname{Im}}

\theoremstyle{plain}
\newtheorem{application}{Application}

\begin{document}

   \author{Thomas Kriecherbauer}
   \address{Inst. for Mathematics, Univ. Bayreuth, 95440 Bayreuth, Germany}
   \email{thomas.kriecherbauer@uni-bayreuth.de}

   \author{Kristina Schubert}
   \address{Inst. for Math. Stat., Univ. M\"unster, Einsteinstr. 62, 48149 M\"unster, Germany}
   \email{kristina.schubert@uni-muenster.de}
 
  \author{Katharina Sch\"uler}
   \address{Inst. for Mathematics, Univ. Bayreuth, 95440 Bayreuth, Germany}
   \email{katharina.schueler@uni-bayreuth.de}
   
   \author{Martin Venker}
   \address{Fac. of Mathematics, Univ. Bielefeld, P.O.Box 100131, 33501 Bielefeld, Germany}
   \email{mvenker@math.uni-bielefeld.de}

   \title[Global Asymptotics for the CD-Kernel of RMT]{Global Asymptotics for the Christoffel-Darboux Kernel of Random Matrix Theory}

\dedicatory{Dedicated to Leonid Pastur with admiration on the occasion of his 75th birthday}
   
   % abstract (optional)
   \begin{abstract}
The investigation of universality questions for local eigenvalue statistics continues to be a driving force in the theory of Random Matrices. For Matrix Models \cite{bookPasturS} the method of orthogonal polynomials can be used and the asymptotics of the Christoffel-Darboux kernel \cite{bookSzego} become the key for studying universality. In this paper the existing results on the CD-kernel will be extended in two directions. Firstly, in order to analyze the transition from the universal to the non-universal regime, we provide leading order asymptotics that are global rather than local. This allows e.g.~to describe the moderate deviations for the largest eigenvalues of unitary ensembles ($\beta = 2$), where such a transition occurs. Secondly, our asymptotics will be uniform under perturbations of the probability measure that defines the matrix ensemble. Such information is useful for the analysis of a different type of ensembles \cite{GoetzeVenker}, which is not known to be determinantal and for which the method of orthogonal polynomials cannot be used directly. The just described applications of our results are formulated in this paper but will be proved elsewhere. As a byproduct of our analysis we derive first order corrections for the $1$-point correlation functions of unitary ensembles in the bulk. Our proofs are based on the nonlinear steepest descent method \cite{DeiftZhou}.
They follow closely \cite{DKMVZ1} and incorporate improvements introduced in \cite{KuijlaarsVanlessen, Vanlessen}. The presentation is self-contained except for a number of general facts from Random Matrix theory and from the theory of singular integral operators.

  \end{abstract}

   % AMS subject classifications (used in AMS journals)

   % AMS keywords (used in AMS journals)

   % acknowledge support, etc
  \thanks{
The first, second and third author were partially supported by the SFB/TR 12. The fourth author was partially supported by the SFB 701. We gratefully acknowledge the support of the Deutsche Forschungsgemeinschaft provided through these grants.
  }

   % today's date, or fill in whatever date you prefer
   %\date{\today}

   \maketitle

\setcounter{equation}{0}
\section{Introduction}
  It is a somewhat surprising but well established fact that eigenvalues of randomly chosen matrices provide useful statistical
models in various areas of mathematics and physics ranging from number theory to quantum chaos. 
	A common theme in many of these applications of Random Matrix Theory is that local eigenvalue statistics yield good asymptotic descriptions of point processes on
the line where the random points are not distributed independently but display local repulsion of some form instead. 
	Examples for local statistics are distributions of extremal values and of nearest neighbor spacings. See \cite[Part III]{handbookRM} and references therein for a recent collection of such applications and \cite{KrieKrug} for an elementary exposition of a specific example.

In general, the
analysis of local statistics poses major technical difficulties for point processes with stochastic dependencies. However, in the
special case of eigenvalue distributions of Gaussian ensembles Gaudin and Mehta \cite{MehtaGaudin} discovered that certain statistical quantities could be expressed in terms of Hermite polynomials (cf.~\cite[Section 3.10]{bookAGZ}). The then already known large degree asymptotics for
these polynomials eventually led to an efficient description of local eigenvalue statistics for matrices of large size. Starting with the work of Pastur and Shcherbina \cite{PasturS97} this path of
analysis was later extended to a larger class of random matrices that are now known as matrix models or as invariant matrix ensembles. We do not present a broad view on these ensembles in this paper and we refer the reader to the recent books \cite{handbookRM, bookAGZ, DeiftGioev, bookForrester, bookPasturS} as well as to the classics \cite{bookMehta, Deift} instead.
	
In this paper we restrict our attention to unitarily invariant ensembles, which are commonly called unitary ensembles. Despite being slightly misleading we also use this abbreviation.  More precisely, we consider probability measures
$\hat{\mathbb{P}}_{N,V}$	on the set of $N \times N$ Hermitian matrices where $V: J\to \mathbb R$ denotes a function on $J=\{
x \in \mathbb R \,|\, L_- \leq x \leq L_+\}$, which is a finite or infinite interval  ($- \infty  \leq L_- < L_+ \leq \infty$). Precise
assumptions on $V$ will be given below (see assumption \GAi).
The measure  $\hat{\mathbb{P}}_{N,V}$ is of the form
	\begin{align}\label{measure_P}
	\mathrm d \hat{\mathbb{P}}_{N,V}(M)= C_{N,V} e^{- N \text{tr}(V(M))} \dopp1_{J}(M) \, \mathrm dM,
	\end{align}
where $C_{N,V}>0$ is a normalizing constant and 
\begin{align*}
\mathrm dM= \prod_{j=1}^N \mathrm dM_{jj} \prod_{1 \leq j <k \leq N} \mathrm dM_{jk}^R \mathrm dM_{jk}^I
\end{align*}
is called the Lebesgue measure on Hermitian matrices. It is defined as the product of Lebesgue measures on the real and imaginary  parts of the
entries $M_{jk}=M_{jk}^R +i M_{jk}^I$ of the upper triangular block and of Lebesgue measures on the real diagonal entries $M_{jj}$. The matrix $V(M)$ is defined by functional
calculus and $\text{tr}$ denotes the trace. By a slight abuse of notation, the symbol $\dopp1_{J}(M)$ takes on the value 1 if all eigenvalues of $M$ lie in $J$ and equals $0$ otherwise.

$\hat{\mathbb{P}}_{N,V}$ induces a probability measure $\mathbb{P}_{N,V}$ on the vector $\lambda=(\lambda_1,\ldots,
\lambda_N) \in \mathbb R^N$ of eigenvalues. Assuming that all orderings of the eigenvalues are equally likely, one obtains $\mathrm d
\mathbb{P}_{N,V}(\lambda)= P_{N,V}(\lambda) \, \mathrm d\lambda$ with density (see e.g.~\cite[Sec.~5.3]{Deift})
\begin{align}
\label{density_P}
P_{N,V}(\lambda)= \frac{1}{Z_{N,V}} \prod_{1 \leq j<k\leq N} (\lambda_k-\lambda_j)^2 \prod_{m=1}^N e^{- NV(\lambda_m)}
\dopp1_{J}(\lambda_m).
\end{align}
Of course, the constant $Z_{N,V}$ has to be chosen such that $\int_{\mathbb R^N} P_{N,V}(\lambda) \,\mathrm d\lambda=1$.
In order to analyze the distribution of the eigenvalues, one needs to understand the behavior of the $k$-point correlations
\begin{align}\label{def_correlation_function}
R_{N,V}^{(k)} (\lambda_1,\ldots, \lambda_k) := \frac{N!}{(N-k)!} \int_{\mathbb R^{N-k}} P_{N,V} (\lambda) \,\mathrm d\lambda_{k+1} \ldots
\mathrm d\lambda_N
\end{align}
for all $1 \leq k \leq N$. Observe that the above mentioned examples for local statistics, i.e.~the distribution of spacings as well as the distribution of the largest value $\lambda_{\text{max}} := \max \{ \lambda_1,\ldots, \lambda_N\}$,
can be expressed in terms of $k$-point correlations (cf.~\cite{handbookAnderson}). 
For densities $P_{N,V}$ of the form (\ref{density_P}) there are two key facts that allow to study  $k$-point correlations in great
detail. 

Firstly, $P_{N,V}$ defines a determinantal point process, i.e. 
\begin{align}
\label{det_point_process}
R_{N,V}^{(k)} (\lambda_1,\ldots, \lambda_k)=\det \left[ \left( K_{N,V}(\lambda_i,\lambda_j) \right)_{1\leq i,j\leq k} \right]
\end{align}
for some kernel $K_{N,V}$ that is independent of $k$ (see e.g.~\cite[(5.40)]{Deift}). 
Moreover, the kernel is of such type that gap
probabilities, i.e.~probabilities $p_I$ that no eigenvalue lies in $I$, are given by Fredholm determinants
\begin{align*}
p_I=\det \left( \dopp1- \mathcal K_{N,V}|_{L^2(I)}\right),
\end{align*}
where $\mathcal K_{N,V}$ denotes the integral operator associated with the kernel $K_{N,V}$ and $\dopp1$ denotes the identity on $L^2(I)$ (see \cite{TracyWidom1998} for a nice derivation). This implies, for example, 
\begin{align*}
\text{Probability}(\lambda_{\text{max}} \geq t)= \det \left( \dopp1- \mathcal K_{N,V}|_{L^2(t,\infty)}\right).
\end{align*}
The second key fact is that the kernel $K_{N,V}$ can be expressed in terms of orthogonal polynomials. More precisely, denote by
$p_j^{(N,V)}(x)=\gamma_j^{(N,V)} x^j  + \ldots $ the unique polynomial of degree $j$ with leading coefficient 
$\gamma_j^{(N,V)}>0$ satisfying 
\begin{align*}
\int_J p_i^{(N,V)}(x) \, p_j^{(N,V)}(x) \, e^{-NV(x)} \, \mathrm dx =\delta_{ij}.
\end{align*}
Then 
\begin{align}\label{defT_kernel}
K_{N,V}(x,y)=\sum_{j=0}^{N-1} p_j^{(N,V)}(x)p_j^{(N,V)}(y) e^{-\frac{N}{2}(V(x)+V(y))}
\end{align}
and by the Christoffel-Darboux formula \cite{bookSzego}
\begin{align}
\label{Chis_Darboux}
K_{N,V}(x,y)=\frac{\gamma_{N-1}^{(N,V)}}{\gamma_{N}^{(N,V)}} \,
\frac{p_N^{(N,V)}(x)p_{N-1}^{(N,V)}(y)-p_{N-1}^{(N,V)}(x)p_{N}^{(N,V)}(y)}{x-y} e^{-\frac{N}{2}(V(x)+V(y))}.
\end{align}
This formula also gives rise to the naming of $K_{N,V}$ as the Christoffel-Darboux kernel or CD-kernel in short.

A driving force in the development of Random Matrix Theory has been the {\em universality conjecture}. Recent accounts of the state of the art can be found in the monographs \cite{DeiftGioev, bookPasturS} and in the survey \cite{handbookKuijlaars}.
For unitary ensembles of type (\ref{measure_P}) it says, roughly speaking, that local eigenvalue statistics do not depend on $V$ as the
matrix size $N$ tends to infinity. Of course, one needs to distinguish between regions where eigenvalues accumulate (bulk), regions
where the appearance of eigenvalues is highly unlikely (void) and transitional regions (edge). 
One way to prove universality is to use relation (\ref{det_point_process}) between $k$-point correlations and the CD-kernel. 
The task then becomes to show that the appropriately rescaled kernel $K_{N,V}$ does not depend on $V$ in the limit $N \to
\infty$. For example, if $x \in J$ lies in the bulk of the spectrum, then $K_{N,V}$ is well described by the sine kernel. More
precisely, 
\begin{align}
\label{con_sine_kernel}
\frac{1}{N \beta} K_{N,V} \left(x+ \frac{s}{N\beta}, x+ \frac{t}{N\beta}\right) = \frac{\sin(\pi(s-t))}{\pi (s-t)} +
\mathcal{O}\left(N^{-1} \right),
\end{align}
where $\beta$ denotes some positive constant, depending on $x$ and $V$, such that $\frac{1}{N\beta}$ is the average spacing of
eigenvalues that lie in a small vicinity of $x$. 
In other words: The universal aspect of (\ref{con_sine_kernel}) is that the $V$-dependency of the leading order ($N \to \infty$) of
$K_{N,V}$ is all captured in a single rescaling factor $\beta$, which in addition defines a natural local length scale. 
One may not expect that the error bound $\mathcal{O}(N^{-1})$ in (\ref{con_sine_kernel}) is uniform for all $s$, $t \in
\mathbb R$ since $x+ \frac{s}{N\beta}$, $x+ \frac{t}{N\beta}$ are required to lie in a vicinity of the point $x$.
In fact, uniformity of the error bound in \eqref{con_sine_kernel} has so far been formulated for $s,t$ in arbitrary  but fixed bounded sets (see \cite{DKMVZ2} and \cite{KuijlaarsVanlessen, Vanlessen} for related unitary ensembles where orthogonal polynomials of Jacobi-type resp.~of Laguerre-type are used). 
This means that  (\ref{con_sine_kernel}) provides leading order asymptotics for $K_{N,V}(u,v)$ only if both $u$ and $v$ lie in some
set $U_N$ with shrinking length of order $\mathcal{O}( N^{-1})$.
For some applications such information is not sufficient (see e.g.~Application \ref{application_3}, \cite[Theorem 1]{KriecherbauerSchubert} and references therein) and better
asymptotics, as stated e.g.~in Theorem \ref{theorem_bulk} below, are essential. Bulk universality with weaker bounds on the rate of convergence than in \eqref{con_sine_kernel} but obtained under much weaker regularity assumptions can be found in \cite{PasturS97, PasturS08, McLaughlinMiller, LL08, LL09, L08, L09, L11, L12, L12a, Totik, Simon}. See also \cite{handbookKuijlaars} and references therein for a recent overview on universality results. 

It is well known that the validity of the sine kernel description ends at spectral edges.
There, the leading order description of $K_{N,V}$ becomes more complicated. In particular, one needs to distinguish between hard
edges and soft edges. We do not give a precise definition of this terminology, but the reader may consult Remark \ref{remark_mu} (b)
below or e.g.~\cite[Subsections 7.1.3 and 7.2.1]{bookForrester} for more information. In this paper we are only concerned with soft
edges where the Airy kernel (see e.g.~\cite[Theorem 5.3.3]{bookPasturS} and \cite[Section 10.4]{AS} for a definition and properties of the  Airy function Ai)
\begin{equation}
\label{Darst_Ai}
\Aik (s,t) :=\frac{\text{Ai}(s)\text{Ai}'(t)- \text{Ai}'(s)\text{Ai}(t)}{s-t} = \int_0^{\infty} \text{Ai}(s+r) \text{Ai}(t+r) \, \mathrm dr
\end{equation}
replaces the sine kernel (see e.g.~\cite{DeiftGioev2007, Vanlessen, McLaughlinMiller, LL11} and \cite[Section 13.1]{bookPasturS}). Let us have a closer look at the result \cite[Theorem 1.1]{DeiftGioev2007} which provides the strongest bounds on the rate of convergence. In \cite{DeiftGioev2007}
ensembles are considered that are slightly different from \eqref{measure_P}. Nevertheless, their result applies to ensembles of type (\ref{measure_P}) with
$J=\mathbb R$ and $V(x)=x^{2m}$ for some positive integer  $m$. It reads
\begin{align}
\label{result_K_edge}
\frac{1}{N^{\frac{2}{3}}\gamma} K_{N,V}\left( b+ \frac{s}{N^{\frac{2}{3}} \gamma}, b+ \frac{t}{N^{\frac{2}{3}} \gamma}\right)
= \Aik(s,t) + \mathcal{O}\left( N^{-\frac{2}{3}}e^{-C(s+t)} \right),
\end{align}
where both $b$ and $\gamma$ depend on $V$. Again, the universal aspect of (\ref{result_K_edge}) is that the $V$-dependency of the
leading asymptotics of $K_{N,V}$ only enters through these two numbers. While $\gamma$ has no natural interpretation, $b$ describes
the almost sure limit of the largest eigenvalue as $N \to \infty$. The bound $\mathcal{O}\left( N^{-2/3}e^{-C(s+t)} \right)$
as well as the positive constant $C$ appearing therein are both uniform for $s,t \in [L_0,\infty)$ with $L_0$ being arbitrary but fixed. 
This result appears to be of a more general type than the one stated above for the bulk, because $s$ and $t$ are not required to lie in a bounded
set. However, due to the rapid decay of the Airy kernel for $s$, $t \to \infty$ (see e.g.~(\ref{Asy_AiK})),
the error term is dominant for $s,t \geq (\log N)^{\alpha}$ if $\alpha > \frac{2}{3}$. Although (\ref{result_K_edge}) is sufficient
to prove the Tracy-Widom law for the rescaled largest eigenvalue $(\lambda_{\text{max}}-b)N^{2/3}\gamma$ in the limit $N \to
\infty$ (cf.~\cite{TracyWidom1994}, \cite[Corollary 1.4]{DeiftGioev2007}), relation (\ref{result_K_edge}) only provides estimates in the region 
$\left( b+ N^{-2/3}(\log N)^{\alpha}, \infty \right)$, $\alpha >
\frac{2}{3}$,  that are far too weak to obtain moderate or large deviation principles for $\lambda_{\text{max}}$. We will be able to
derive such principles from Theorem \ref{theorem_asymptotics} (see Application \ref{application_1}).

Another source of motivation to write this paper comes from a rather novel class of probability measures on $\R^N$, for which local bulk universality was proven recently in \cite{GoetzeVenker} (see also \cite{Venker} for related results). In a more general setup they appeared first in \cite{BPS}, where macroscopic correlations have been studied. In our case, these ensembles have densities proportional to
\begin{align}\label{density_venker}
\prod_{1\leq j<k\leq N} \varphi (x_j-x_k) \prod_{m=1}^N e^{- NQ(x_m)}.
\end{align}  
Apart from technical conditions it is assumed that $\varphi$ satisfies
\begin{align*}
\varphi(0)=0, \quad \varphi(t) >0 \quad\text{for } t\neq 0 \quad \text{and}\quad \lim_{t \to 0} \frac{\varphi(t)}{|t|^2}=C>0.
\end{align*}
We prefer to think that \eqref{density_venker} describes the
distribution of the positions $x_j$ of $N$ particles on the real line
rather than eigenvalues of some matrix ensemble. However, as the last
author learned from G. Borot, the ensemble
\eqref{density_venker} in the form \eqref{density_h} can also be
realized as eigenvalue distribution of a random Hermitian matrix
(\cite{Borot}, see also \cite{PasturLytova}).

The motivation to study such ensembles comes from the fact that
\eqref{density_venker} constitutes a repulsive particle system which
shares with \eqref{density_P} the repulsion strength between close particles. 
Proving local universality for these ensembles, i.e.~
showing that it has local limiting distributions not depending on $Q$
and $\varphi$ apart from the repulsion condition, might be a further
step to an understanding of origin and range of the ubiquity of local
random matrix distributions.

Generically, ensembles \eqref{density_venker}
do not carry a determinantal structure and the method of orthogonal polynomials, which led to the
introduction of the Christoffel-Darboux kernel, cannot be applied directly. Nevertheless, by a stochastic linearization procedure
any density of the form (\ref{density_venker}) with even $\varphi$ can be related to a family of ensembles of type (\ref{density_P}) with $V=V_N=Q+\frac{1}{N}f,$
indexed by $f$. Indeed, the functions $f$ are given as sample paths of a centered stationary Gaussian process on $\mathbb R$, whose
covariance function is determined by $\varphi$.  
Using universality results for $K_{N,Q+f/N}$, pointwise for every $f$, G\"otze and Venker were able to prove universality
for all $k$-point correlations in the bulk. Note that the rates of convergence for $K_{N,Q+f/N}$ will provide corresponding
rates for the $k$-point correlations only if they are uniform in $f$.  
In view of Applications \ref{application_2} and \ref{application_3}, we will
therefore derive bounds on the rates of convergence for the Christoffel-Darboux kernel $K_{N,V}$ that are independent of $V$ in some
neighborhood of $Q$. At this point we can also explain why we do not only consider functions $V$ that are defined on all of $\mathbb R$
but on some interval $J$ instead. The main reason is that in the analysis of  G\"otze and Venker a truncation is used so that
$V=Q+\frac{1}{N}f$ is only defined on a suitable compact interval that is chosen independently of $N$. Note that the
truncation is always performed in such a way that only soft edges (see paragraph below (\ref{MRS_number})) occur.

Before stating our results we summarize the main objectives of the present paper:\smallskip\\
$\bullet$
With Theorem \ref{theorem_kernel} we present a uniform leading order description of the Christoffel-Darboux kernel $K_{N,V}(x,y)$. Compared to the results in the literature the novel aspects are that the description is global in $(x, y)$ rather than local and that it applies not only to a fixed function $V$ but simultaneously to an open neighborhood of such functions. \\
$\bullet$
Having Applications \ref{application_1}, \ref{application_2}, and \ref{application_3} in mind, we further evaluate the results of Theorem \ref{theorem_kernel} in special cases. While Theorems \ref{theorem_bulk} and \ref{theorem_edge} extend the universality results (\ref{con_sine_kernel}) and (\ref{result_K_edge}) for unitary ensembles to their maximal domains of validity, Theorem \ref{theorem_asymptotics} contains first order corrections in the bulk as well as matching formulae at the
edges for the $1$-point correlation function that appear to be new, except for the Gaussian case \cite{Kamenev}.\\
$\bullet$
The proof of Theorem \ref{theorem_kernel} uses the Riemann-Hilbert formulation for orthogonal polynomials \cite{FIK}. The asymptotic analysis of the Riemann-Hilbert problem is essentially identical with \cite{DKMVZ1} and is based on the nonlinear steepest descent method introduced by Deift-Zhou \cite{DeiftZhou} and further developed in \cite{DVZ}. We have streamlined the exposition in view of our prime goal to analyze the global asymptotics of the Christoffel-Darboux kernel. In particular, we have incorporated the improvements that were introduced in \cite{KuijlaarsVanlessen} in the form presented in \cite{Vanlessen} (see (\ref{K_N:2}) and (\ref{eq_Y_+_1})).\\
$\bullet$
We restrict our attention to unitary Hermite-type ensembles, where only soft edges
occur, with varying weights  $e^{-NV}$ and convex $V$. Nevertheless, we hope that this paper may serve as a blueprint to extract
global asymptotics of the  Christoffel-Darboux kernel for other types of unitary ensembles (e.g.~Laguerre-type, Jacobi-type) or in degenerate cases by
adapting the analysis of this paper with the help of \cite{DaiKuijlaars, KriecherbauerMcLaughlin, KMVV, KuijlaarsVanlessen, KuijlaarsVanlessen03, Vanlessen03, Claeys, ClaeysKuijlaars06, ClaeysKuijlaars, CKV, ClaeysVanlessen, IKO, RiderZhou}.\smallskip

In order to specify our assumptions on $V$, we first need to introduce the Mhaskar-Rakhmanov-Saff numbers $a$, $b \in\mathbb R$ 
from the theory of orthogonal polynomials. In the case of convex $V$ and varying weights $e^{-NV}$ they are defined implicitly by
the conditions 
\begin{align}
\label{MRS_number}
\int_a^b \frac{V'(t)}{\sqrt{(b-t)(t-a)}} \,\mathrm dt =0, \quad \int_a^b \frac{t V'(t)}{\sqrt{(b-t)(t-a)}} \, \mathrm dt =2 \pi. 
\end{align} 
The significance of these relations will become clear in Section \ref{sec2}. For twice differentiable functions $V:
\mathbb R \to \mathbb R$ with $V'$ increasing and $\lim_{|x|\to \infty} V(x)= \infty$ one may show by elementary but not
entirely trivial arguments that there exist unique real numbers $a<b$ satisfying (\ref{MRS_number}). This also implies that for restrictions $\widetilde{V}=V{\big|_J}$ to any interval $J$, relations \eqref{MRS_number} can only hold
if $[a,b]\subset J$. In our context such restrictions come into play because of the truncation procedure that is used in the
analysis of ensembles of type \eqref{density_venker}, see \cite{KriecherbauerVenker}, \cite{SchubertVenker}. There the truncation is always performed in such a
way that $[a,b]$ is contained in the interior of $J$. As explained in Remark \ref{remark_mu} (b) below,
this is the reason why we only need to consider soft edges.

We are now ready to state our general assumptions on $V$ that are of two types: \vspace{8pt}

\noindent
\textbf{(GA)$_1$} \hspace{5pt} A function $V$ is said to satisfy \textbf{(GA)$_1$} if (1)--(4) hold  

%\begin{minipage}{0.08\linewidth}
 %\textbf{(GA)$_1$}
%\end{minipage}
%\begin{minipage}{0.85\linewidth}
 \begin{enumerate}
 \item $V:J\rightarrow \R$ is real analytic, $J=[L_-, L_+] \cap \R$ with $-\infty \le L_-<L_+ \le \infty$. \vspace{4pt}
\item $V'$ is strictly increasing (convexity assumption). \vspace{4pt}
\item $\lim_{\lv x\rv\to\infty}V(x)=\infty$. \vspace{4pt}
\item There exist $L_-<a_V<b_V<L_+$ such that \eqref{MRS_number} holds with $a=a_V$, $b=b_V$. 
\end{enumerate}
%\end{minipage}
\vspace{5pt}

Condition (1) allows to perform the nonlinear steepest descent method of Deift-Zhou
 \cite{DeiftZhou} in its simplest form (see \eqref{RHP_factor} and below). For an extension to the case of finite regularity we refer the reader to the work of McLaughlin and Miller \cite{McLaughlinMiller}, where functions $V$ are considered that have Lipschitz continuous second derivatives.
 Condition (2) ensures that the equilibrium measure (see Section \ref{sec2}) is of the most simple type as well, i.e.~its support consists of only one interval with square-root vanishing of the density at the endpoints (see e.g.~\cite{DKMVZ2, DKM, KuijlaarsMcLaughlin}  and references therein for a more general picture). 
It is clear that some condition
like (3) is needed to ensure integrability of $P_{N,V}$ given by \eqref{density_P} in the case of unbounded intervals $J$.
On first sight, (3) appears to be too weak for that. However, using condition (2), it is not hard to see that $V$ grows at least linearly
for $x\to\pm\infty$.

The second type of assumptions is adapted to the analysis of ensembles of the form \eqref{density_venker}.\vspace{8pt}

\noindent
\textbf{(GA)$_2$} \hspace{5pt} $Q$ is said to satisfy \textbf{(GA)$_2$} if \textbf{(GA)$_1$} holds for $Q$ with (2), (3) replaced by

%\begin{minipage}{0.08\linewidth}
% \textbf{(GA)$_2$}
%\end{minipage}
%\begin{minipage}{0.85\linewidth}
%\quad \ $Q$ satisfies \textbf{(GA)$_1$} with (2), (3) being replaced by
\begin{enumerate}
 \item[(2\,$^\prime$)] $\inf_{x\in J}Q''(x)>0$. \vspace{4pt}
\item[(3\,$^\prime$)] $J$ is compact, i.e.~$L_-$, $L_+$ are both finite.  
\end{enumerate} 
%\end{minipage}
\begin{remark}\label{remark_complex_entension}
 Given a function $Q$ that satisfies \GAii, there exists a complex open neighborhood $D$ of $J$ such that $Q$ can be extended
analytically to a bounded function on $D$ that we still denote by $Q$. In this sense, $Q$ belongs to the real Banach space  $X_D:=\{\map{f}{D}{\C}\,|\,f$ analytic and bounded, $f(D\cap\R)\subset\R\}$, equipped with the sup-norm
$\|f\|_\infty:=\sup_{z\in D}\lv f(z)\rv$.
\end{remark}

Given a function $V$ that satisfies \GAi, we now define a number of objects that are needed to state our results.
\begin{definition}\label{def_notation}
Assume that \GAi\ holds for $V$. The linear rescaling that maps $[-1,1]$ onto $[a_V,b_V]$ (cf.~\GAi (4)) is denoted by
\begin{align}
 \map{\l_V}{\R}{\R},\  \l_V(s):=\frac{b_V-a_V}{2}s+\frac{b_V+a_V}{2}.\label{def_lambda}
\end{align}
 Hence, $[-1,1]\subset \hat{J}:=\l_V^{-1}(J)$ (cf.~\GAi (1)). Moreover, set
\begin{align}
 &\map{h_V}{\hat{J}\times\hat{J}}{\R},\
h_V(t,x):=\int_0^1(V\circ\l_V)''(x+u(t-x))\, \mathrm du\label{def_h}\\
&\hspace{4.9cm}=\frac{(V\circ\l_V)'(t)-(V\circ\l_V)'(x)}{t-x},\nonumber\\
&\map{G_V}{\hat{J}}{\R},\ G_V(x):=\frac{1}{\pi}\int_{-1}^1\frac{h_V(t,x)}{\sqrt{1-t^2}}\, \mathrm dt,\label{def_G}\\
&\map{\rho_V}{\R}{\R},\ \rho_V(x):=\begin{cases}
                                   \frac{1}{2\pi}\sqrt{1-x^2}\, G_V(x)	&, \text{ if }\lv x\rv\leq 1 ,\\
				   0					&, \text{ else,} 
                                  \end{cases}\label{def_rho}\\
&\map{\xi_V}{\R}{\R},\ \xi_V(x):=2\pi\int_x^1\rho_V(t)\, \mathrm dt,\label{def_xi}\\
&\map{\eta_V}{\hat{J}}{\R},\ \eta_V(x):=\begin{cases}
					  \int_1^x\sqrt{t^2-1}\, G_V(t)\, \mathrm dt	&, \text{ if }x>1 ,\\
					  0				&, \text{ if }\lv x\rv\leq1 ,\\
					  \int_x^{-1}\sqrt{t^2-1}\, G_V(t)\, \mathrm dt	&, \text{ if }x<-1 ,
                                        \end{cases}\label{def_eta}\\
&\map{a}{\R\setminus[-1,1]}{\R},\ a(x):=\lb\frac{x-1}{x+1}\rb^{1/4},\label{def_a}\\
&\map{\hat{a}}{(-1,1)}{\R},\ \hat{a}(x):=\lb\frac{1-x}{1+x}\rb^{1/4}.\label{def_hat_a}
\end{align}
As argued in Remark \ref{remark_f_hat} below, there exists a real analytic function $\map{\hat{f}_V}{\mathcal{E}_V}{\R}$
with $\mathcal{E}_V:=[-1-\delta_V,-1+\delta_V]\cup[1-\delta_V,1+\delta_V]$ for some $0<\delta_V<1$ such that
\begin{align}
&(\lv x-1\rv\gamma_V^+\hat{f}_V(x))^{3/2}=\begin{cases}
					    \frac{3}{4}\xi_V(x)&, \text{ if }1-\delta_V\leq x\leq 1 ,\vspace{3pt}\\
					    \frac{3}{4}\eta_V(x)&, \text{ if }1\leq x\leq 1+\delta_V ,\vspace{3pt}
                                          \end{cases} \label{eq_Th_1}\\
&(\lv x+1\rv\gamma_V^-\hat{f}_V(x))^{3/2}=\begin{cases}
					    \frac{3}{4}(2\pi-\xi_V(x))&, \text{ if }-1\leq x\leq -1+\delta_V ,\vspace{3pt}\\
					    \frac{3}{4}\eta_V(x)&, \text{ if }-1-\delta_V\leq x\leq -1 .
                                          \end{cases} \label{eq_Th_2} \vspace{-6pt}
\end{align}
Here \begin{align}\label{def_gamma}
      \gamma^\pm_V:=2^{-1/3}G_V(\pm 1)^{2/3}
     \end{align}
is chosen in such a way that $\hat{f}_V(\pm1)=1$. Finally,
\begin{align}
 &\map{f_{N,V}}{\mathcal{E}_V}{\R},\ f_{N,V}(x):=\begin{cases}
						  N^{2/3}\gamma_V^+(x-1)\hat{f}_V(x)&, \text{ if } \lv x-1\rv\leq\delta_V ,\vspace{3pt}\\
						  -N^{2/3}\gamma_V^-(x+1)\hat{f}_V(x)&, \text{ if } \lv x+1\rv\leq\delta_V ,
                                               \end{cases} \label{eq_Th_3}\\
&\map{d_V}{\mathcal{E}_V}{\R},\ d_V(x):=\begin{cases}
						  ((x+1)\hat{f}_V(x))^{1/4}&, \text{ if } \lv x-1\rv\leq\delta_V ,\vspace{3pt}\\
						  ((1-x)\hat{f}_V(x))^{1/4}&, \text{ if } \lv x+1\rv\leq\delta_V .
                                               \end{cases} \label{eq_Th_4}
\end{align}
\end{definition}

Our theorem on the global asymptotics of the Christoffel-Darboux kernel reads

\begin{theorem}\noindent\label{theorem_kernel}
 \begin{enumerate}[(a)]
  \item Assume that the function $V$ satisfies \GAi. Let $\l_V$, $\delta_V$, $\hat{J}$, $\eta_V$, $a$, $\hat{a}$, $\xi_V$, $\gamma_V^\pm$, $f_{N,V}$, $d_V$ be
given as in Definition \ref{def_notation}. Then there exists a positive $\delta_0<\delta_V$ such that for all $\delta\in
(0,\delta_0]$ the following holds:
Define $\map{k}{\hat{J}}{\R^2}$ by
\begin{align*}
 k(x):=\begin{cases}
        \frac{1}{\sqrt{4\pi}}\sgn(x)^N e^{-\frac{N}{2}\eta_V(x)}\vec{a(x)\\ [a(x)]^{-1}}&, \text{ if }
x\in\hat{J}\setminus[-1-\delta,1+\delta],\vspace{3pt}\\
\frac{1}{\sqrt{\pi}}\vec{\hat{a}(x)\cos\lb\frac N2\xi_V(x)+\frac\pi4\rb\\ [\hat{a}(x)]^{-1}\cos\lb\frac{N}{2}\xi_V(x)-\frac\pi4\rb}&,
\text{ if } x\in (-1+\d,1-\d),\vspace{3pt}\\
\vec{-\Ai'(f_{N,V}(x))[N^{1/6}(\g_V^+)^{\frac{1}{4}} d_V(x)]^{-1}\\\Ai(f_{N,V}(x))[N^{1/6}(\g_V^+)^{\frac{1}{4}} d_V(x)]}&, \text{ if } x\in [1-\d,1+\d],\vspace{3pt}\\
(-1)^N \vec{\Ai(f_{N,V}(x))[N^{1/6}(\g_V^-)^{\frac{1}{4}} d_V(x)] \\ -\Ai'(f_{N,V}(x))[N^{1/6}(\g_V^-)^{\frac{1}{4}} d_V(x)]^{-1}}&, \text{ if } x\in [-1-\d, -1+\d].
\end{cases}
\end{align*}

Then,
\begin{align}
\frac{b_V-a_V}{2} K_{N,V}(\l_V(x),\l_V(y))=\frac{k_1(x)k_2(y)-k_2(x)k_1(y)}{x-y}+k(y)^T\mathcal{O}\lb N^{-1}\rb k(x).\label{eq_kernel_theorem}
\end{align}
Here $\mathcal{O}(N^{-1})$ denotes a real-valued $2\times2$ matrix whose entries are of order $N^{-1}$. The bound is uniform for $\delta$ in compact subsets of $(0,\d_0]$ and for $x,y$  in bounded subsets of $\hat{J}$.

\item Assume that \GAii\ holds for some function $Q$. Let $D$, $X_D$ be defined as in Remark
\ref{remark_complex_entension}. Then there exists an open neighborhood $\mathcal{U}$ of $Q$ in $(X_D,\|\cdot\|_\infty)$ such that $V\big|_J$
satisfies \GAi\ for all $V\in \mathcal{U}$. Hence claim (a) holds for $V\big|_J$. The error bound in \eqref{eq_kernel_theorem} is also uniform
in $V\in \mathcal{U}$ and $\d_0$ can be chosen to be independent of $V\in \mathcal{U}$.
 \end{enumerate}
\end{theorem}

\begin{remark}\noindent
 \begin{enumerate}[(a)]
\item The definition of $k$ in the bulk, i.e.~for $\lv x\rv<1-\d$ in Theorem \ref{theorem_kernel}, could also be given
by 
\begin{align*}
\begingroup
\renewcommand*{\arraystretch}{1.5}
 k(x)=\frac{1}{\sqrt{2\pi}(1-x^2)^{1/4}}\vec{1 & -1\\1 & 1}\vec{\cos\lb\frac N2\xi_V(x)-\frac12\arcsin(x)\rb\\ \sin\lb\frac
N2\xi_V(x)+\frac12\arcsin(x)\rb},
\endgroup
\end{align*}
which is more common in the literature. Using $\sin\lb\frac12\arccos(x)\rb=\sqrt{1-x}$, the identity of both formulae is easily
verified. We find the formula stated in Theorem \ref{theorem_kernel} more suitable for deriving sine-kernel asymptotics.
\item Of course, for $x=y\in\hat{J}$, the leading term of $K_{N,V}$ in \eqref{eq_kernel_theorem} reads
$k_1'(x)k_2(x)-k_1(x)k_2'(x)$.
 \end{enumerate}
\end{remark}

The evaluation of $K_{N,V}(x, y)$ on the diagonal $x=y$ is of particular interest because it equals the 1-point correlation 
function (see \eqref{det_point_process}) which agrees with the expected density of eigenvalues up to a factor $N$. The following theorem describes the large $N$ behavior of $K_{N,V}(x,x)$. Note that
particular care has been taken to extend bulk resp.~void asymptotics to the edge region. The first order correction in statement (i) was already presented in \cite{Kamenev, GaFoFr} for Gaussian ensembles $V(x) = x^2$, but without uniformity of the error term in $x$. In \cite{GaFoFr} first order corrections were obtained for the edge as well. Our method also allows to derive such corrections by expanding $R$ in Theorem \ref{theorem_R_+} to first order as indicated e.g.~in \cite[Theorem 7.10]{DKMVZ1}.

\begin{theorem}\label{theorem_asymptotics}
 Let the assumptions of Theorem \ref{theorem_kernel} (a) resp.~(b) be satisfied. 
\begin{comment} 
 Then the corresponding statements of Theorem \ref{theorem_kernel}, in particular the uniformity of error bounds (and of the choice of 
$\d_0$ in (b)), also hold for the following
\end{comment}
Then there exists a (small) positive constant $c$ such that the following representations hold for
$$D(x): = \frac{b_V-a_V}{2} K_{N,V}(\l_V(x),\l_V(x)) \, .$$ 
\begin{enumerate}[(i)]
 \item For $x\in [-1,1]$ with $|x| < 1 - c^{-1}N^{-2/3} :$
\begin{align*}
D(x)=\lb N\rho_V(x)-\frac{1}{2\pi(1-x^2)}\cos(N\xi_V(x))\rb\lb1+\mathcal{O}\lb \frac{1}{N^2(1-\lv
x\rv)^3}\rb\rb.
\end{align*}
\item For $x\in\hat{J}\setminus[-1,1]$ with $|x| > 1 + c^{-1}N^{-2/3} :$
\begin{align*}
D(x)=\frac{1}{4\pi}e^{-N\eta_V(x)}\left[\frac{1}{x^2-1}+\O\lb \frac{1}{N(\lv x\rv-1)^{5/2}}\rb+\O\lb \frac1N\rb\right].
\end{align*}
\item For $1-c<x<1+c N^{-2/5} :$
\begin{align}\label{eq_D3}
D(x)=N^{2/3}\g_V^+\Aik\lb\g_V^+N^{2/3}(x-1),\g_V^+N^{2/3}(x-1)\rb\lb1+r(x)\rb,\\
\text{ where } r(x)=\begin{cases}
			\O(1-x)+\O\lb N^{-2/3}\rb&, \text{ if } x\leq1,\\
			\O\lb N(x-1)^{5/2}\rb+\O\lb N^{-2/3}\rb&, \text{ if } x\geq1 .
                     \end{cases}\nonumber
\end{align}
\item In the case $-1- c N^{-2/5} <x<-1+ c$, statement (iii) holds with $x$ being replaced by $-x$ on the r.h.s. of equation \eqref{eq_D3} and with
$\g_V^+$ replaced by $\g_V^-$.
\end{enumerate}
In the situation of part (b)  the error bounds as well as the constant $c$ can be chosen independently of $V \in \mathcal{U}$ with $\mathcal{U}$ as given by Theorem \ref{theorem_kernel}.
\end{theorem} 
\begin{remark}
The reader may wonder why the error terms in statement (i) and (iii) can be written in multiplicative form. This is due to the fact that neither $D$ nor the factors on the right hand sides have any zeros. For $D$ resp.~for the Airy kernel this follows from \eqref{defT_kernel} resp.~from \eqref{Darst_Ai}. In the case of the first factor on the right hand side of (i) one needs to use the lower bound on $G_V$ provided by Lemma \ref{lemma_G_V} and Remark \ref{remark_G_V} together with \eqref{def_rho}.
\end{remark}
It will be shown in \cite{Schueler, EKS} that the information provided in Theorem \ref{theorem_asymptotics} (ii) suffices to
derive results on moderate and large deviations for the distribution of the largest eigenvalue $\l_\text{max}$.
Moreover, under some additional (mild) assumptions on $V$, one can prove \cite{Schueler, EKS} that the large deviations results persist in the superlarge regime.

\begin{application}\label{application_1}
Assume that $V$ satisfies \GAi\ with $L_+ < \infty$ and let $\hat{\mathbb{P}}_{N,V}$ be the associated unitarily invariant matrix ensemble $($see
\eqref{measure_P}$)$. The random variable $\a:=\l_V^{-1}(\l_\textup{max})$ denotes the linearly rescaled maximal eigenvalue
(cf.~Definition \ref{def_notation}). Moreover, we introduce the function $O_N$ that describes the distribution of
$\a$ on a local scale,
\begin{align*}
 O_N(s):=\hat{\mathbb{P}}_{N,V}\lb\a>1+\frac{s}{\g_V^+N^{2/3}}\rb \ \text{ $($cf.~\eqref{def_gamma}$)$}.
\end{align*}
Then we have in the regime of moderate deviations, i.e.~$q_N\leq s\leq p_N$ with $q_N\to\infty$ and $N^{-2/3}p_N\to0$ for
$N\to\infty$, that
\begin{align}\label{mod_dev1}
 \frac{\log O_N(s)}{s^{3/2}}=-\frac{4}{3}-\frac{\log\lb 16\pi s^{3/2}\rb}{s^{3/2}}+\O\lb\frac{s}{N^{2/3}}\rb +\O\lb\frac{1}{s^3}\rb.
\end{align}
Under the more restrictive assumption that $N^{-4/15}p_N \ \to 0$ this implies
\begin{align}\label{mod_dev2}
 O_N(s)=\frac{1}{16\pi}s^{-\frac{3}{2}}e^{-\frac{4}{3}s^{3/2}}\lb 1+\O\lb\frac{s^{5/2}}{N^{2/3}}\rb  +\O\lb\frac{1}{s^{3/2}}\rb\rb.
\end{align}
\begin{comment}
The error bounds in \eqref{mod_dev1} and \eqref{mod_dev2} do not depend on the choices for $(q_N)_N$ and $(p_N)_N$.
\end{comment}
\end{application}

\begin{remark}[see \cite{Schueler, EKS}]\label{remark_appl_1}
\begin{enumerate}[(a)]
\item
As mentioned above, it follows from the already known asymptotics of
the Christoffel-Darboux kernel that $O_N(s) \to 1-F_{\textup{TW},2}(s)$ for $s$ fixed as $N\to\infty$,
where $F_{\textup{TW},2}$ denotes the Tracy-Widom distribution with parameter $\beta=2$. 
The asymptotics of the Tracy-Widom distribution are given by
$$1-F_{\textup{TW},2}(s)=
\frac{1}{16\pi}s^{-\frac{3}{2}}e^{-\frac{4}{3}s^{3/2}} \lb 1+\O \lb s^{-\frac{3}{2}} \rb \rb \quad \text{as } s\to\infty \, .$$
Thus,  $F_{\textup{TW},2}(s)$ continues to provide the correct leading order description of $O_N(s)$ for all values of $s$ up to
order $o(N^{4/15})$.
\item
In case one is interested in the asymptotics of $O_N(s)$ for even larger values of $s$, one may use the
formula that is valid in the large deviations regime: Choose $(q_N)_N$ with $q_N\to\infty$ for $N\to\infty$ and $L>1$ such that
$[1+\frac{q_N}{N^{2/3}},L]$ is contained in the interior of $\hat{J}$. Then, in the situation of Application \ref{application_1}, for all $x\in[1+\frac{q_N}{N^{2/3}},L]$ (see \cite{Schueler, EKS})
\begin{align}\label{large_dev}
 \hat{\mathbb{P}}_{N,V}\lb\a>x\rb=\frac{e^{-N\eta_V(x)}}{4\pi N(x^2-1)\eta_V'(x)}\lb 1+\O\lb \frac{1}{N(x-1)^{3/2}}\rb\rb.
\end{align}
In fact, one may derive the statements contained in
Application \ref{application_1} from \eqref{large_dev}. 
Moreover, formula \eqref{large_dev} explains why \eqref{mod_dev2} generically only holds for $s = o(N^{4/15})$. Indeed, we learn from \eqref{eq_Th_1} that
$$
N\eta_V(x) = \frac{4}{3} (\g_V^+)^{3/2} N (x-1)^{3/2} \left[1 + c (x-1) + \O((x-1)^2) \right] \ .
$$
For $c \neq 0$ and $x = 1 + s/(\g_V^+N^{2/3})$ this implies $e^{-N\eta_V(x)} = e^{-\frac{4}{3}s^{3/2}}(1 + \textrm{ small})$ if and only if $s = o(N^{4/15})$.
\end{enumerate}
\end{remark}

There is a fundamental difference between the results in Theorems \ref{theorem_kernel}, \ref{theorem_asymptotics}, and
\eqref{large_dev} as compared to those stated in Application \ref{application_1}. In the latter case the dependency of the leading
term on $V$ is all included in the three numbers $a_V$, $b_V$ and $\gamma_V^+$. Therefore Application \ref{application_1}
is called a universality result, whereas all the other leading order formulae stated above depend on $V$  through all
of the information that is encoded in the function $G_V$ (and derived quantities such as $\rho_V$, $\eta_V$, $\xi_V$, $\hat{f}_V$, $f_{N,V}$, $d_V$).

As advertised earlier, we will now present universality results that are consequences of Theorems \ref{theorem_kernel} and
\ref{theorem_asymptotics} and that extend the previously known range of applicability. We start with the bulk:

\begin{theorem} \label{theorem_bulk}
Let the assumptions of Theorem \ref{theorem_kernel} (a) resp.~(b) be satisfied and let  $0<\delta_0 < \delta_V$ be given as in Theorem \ref{theorem_kernel}. 
For all $\delta \in (0,\delta_0]$ there exists a positive constant $c_{\delta}$ (independent of $V \in \mathcal{U}$ in the situation of (b) with $\mathcal{U}$ as given by Theorem \ref{theorem_kernel}) such that for all $x \in [-1+\delta, 1-\delta]$, and $|s|$, $|t| < c_{\delta} N$:
\begin{align*}
&\frac{b_V-a_V}{2 N \rho_V(x)} K_{N,V} \left( \lambda_V\left( x+ \frac{s}{N \rho_V(x)}\right),\lambda_V\left( x+ \frac{t}{N \rho_V(x)}
\right) \right)\\
& =  \frac{\sin(\pi (s-t))}{\pi (s-t)} + \mathcal{O}\left( \frac{1+|s|+|t|}{N}\right).
\end{align*}
The error bound is uniform in $x$ (and in $V \in \mathcal{U}$ in part (b)) for $\delta$ chosen from an arbitrary but fixed compact
subset of $(0,\delta_0]$.
\end{theorem}
The formulation of universality at the edge is more subtle. The reason is the rapid decay of the Airy-kernel as the arguments tend
to $+\infty$, whereas oscillations occur for negative arguments. 
In view of applications in Random Matrix Theory, we present our results on $K_{N,V}(x,y)$ for $b_V - \mathcal{O}(N^{-2/3})
\leq x,y \leq b_V + \mathcal{O}(N^{-2/5})$, i.e.~in a region which is not symmetric around $b_V$. 
\begin{theorem}\label{theorem_edge}
Let the assumptions of Theorem \ref{theorem_kernel} (a) resp.~(b) be satisfied and assume that $q<0<p$ are given.
Then 
\begin{align*}
&\frac{b_V-a_V}{2 N^{2/3}\gamma_V^+} K_{N,V} \left( \lambda_V \left( 1+ \frac{s}{N^{2/3}\gamma_V^+}\right), \lambda_V \left(
1+ \frac{t}{N^{2/3}\gamma_V^+} \right) \right)
\\
&= \begin{cases}
\Aik(s,t) + \mathcal{O}\lb N^{-2/3}\rb &, \text{ if } q \leq s,t \leq 2,
\\
\Aik(s,t) \left( 1+ \mathcal{O}\left( \frac{s^{5/2} + t^{5/2}}{N^{2/3}} \right) \right) 
&, \text{ if } 1 \leq s,t \leq  p N^{4/15},
\\
\Aik(s,t) + \Ai'(t) \mathcal{O}\left( \frac{t^{3/2}}{N^{2/3}} \right)
&, \text{ if } q \leq s \leq 1,\,  2\leq t \leq  p N^{4/15}.
\end{cases}
\end{align*}
The error bounds may depend on the choice of $q$ and $p$, but they are uniform in $V$ close to $Q$ in part (b).
\end{theorem}
Observe that in the case $s \le 1$ and $t \ge 2$ the asymptotics of the Airy function \cite[10.4.59, 10.4.61]{AS} imply that the Airy-kernel $\Aik(s,t)$ is of order $\Ai'(t) t^{-1}$, except near zeros of $\Aik$, and it is therefore of leading order for $t = o(N^{4/15})$.

 The second application is concerned with edge universality for models of the form \eqref{density_venker}. The following results
will appear in \cite{KriecherbauerVenker}. 

\begin{application}\label{application_2}
Let us define the density
\begin{align}\label{density_h}
 P_{N,Q}^h(x):=\frac{1}{Z_{N,Q}^h}\prod_{i<j}\lv x_i-x_j\rv^2 e^{-\sum_{i<j}h(x_i-x_j)-N\sum_{j=1}^N Q(x_j)}.
\end{align}
Note that $P_{N,Q}^h$ is of the form \eqref{density_venker} with $\varphi(t)=\lv t\rv^2e^{-h(t)}$. We assume that $Q$ is
even and satisfies \GAi~ and that $\alpha_Q:=\inf_{t\in\R} Q''(t)>0$, i.e.~that $Q$ is strictly convex. Furthermore, let
$R_{N,Q}^{h,k}$ denote the $k$-th correlation function of $P_{N,Q}^h$ which is  defined analogously to
\eqref{def_correlation_function}.
Then, given any even, real-analytic Schwartz function $h$ with Fourier transform of exponential decay, there is an $0<\alpha^h<\infty$ such that for all $Q$ as above with
$\alpha_Q\geq \a^h$, we have for some $b_Q^h$ and $\g_Q^h$
\begin{align}
\Big\lv \lb\frac{1}{N^{2/3}\g_Q^h}\rb^{\hspace{-2pt} k} \hspace{-4pt}   R_{N,Q}^{h,k}\lb
b_Q^h+\frac{t_1}{N^{2/3}\g_Q^h},\dots,b_Q^h+\frac{t_k}{N^{2/3}\g_Q^h}\rb-\det[\Aik(t_i,t_j)]_{1\leq i,j\leq k}\Big\rv
\to0,\label{appl_eq_Airy}
\end{align}
as $N\to\infty$, uniformly for $t_1,\dots,t_k$ from intervals of the form $[s, \mathcal{O}(N^{4/15})]$. 

Moreover, the rescaled largest particle converges in law towards the Tracy-Widom distribution with $\beta=2$, as $N\to\infty$:
\begin{align}
 \lim_{N\to\infty}P_{N,Q}^h\lb \g_Q^hN^{2/3}\lb x_{\max}-b_Q^h\rb\leq s\rb=F_{\textup{TW},2}(s)\label{appl_TW}
\end{align}
for any $s\in\R$. Here $F_{\textup{TW},2}$ is the distribution function of the Tracy-Widom distribution with $\beta=2$ already
mentioned in Remark \ref{remark_appl_1}.
\end{application}

To explain the origin of $b_Q^h$ and
$\g_Q^h$, we recall from \cite{GoetzeVenker} the result that under the conditions above there is a probability
measure
$\mu_Q^h$ on $\R$ which describes the global behavior of particles distributed according to \eqref{density_h}. 
This measure is the equilibrium measure in the sense of Section \ref{sec2},
minimizing the energy functional $I_{\widetilde{Q}}$ with the implicitly
defined external field $\widetilde{Q}(t):=Q(t)+\int h(t-s)\, \mathrm d\mu_Q^h(s)$.
It is given by $\mu_Q^h=\mu_{\widetilde{Q}}$ in the notation of \eqref{eq_dV} (see also \eqref{def_rho}) and
we may define $b_Q^h:=b_{\widetilde{Q}}$ resp.~$\g_Q^h:=\g_{\widetilde{Q}}^+$ in the sense of \GAi (4)
resp.~\eqref{def_gamma}. 

\begin{remark}
Recall that ensemble \eqref{density_h} is not known to be
determinantal, therefore convergence of the correlation functions does not follow directly from the convergence of a kernel as provided by Theorem \ref{theorem_edge}. 
The limiting
correlations, however, are of determinantal form and coincide in particular with the ones known from the more classical random
matrix ensembles.

The analysis of local universality for ensembles of this type has been started in \cite{GoetzeVenker},
where bulk universality was proved for $\beta=2$ and the macroscopic correlations have been identified. Bulk universality for
general $\beta>0$ has been addressed in \cite{Venker}.

In the proofs of \eqref{appl_eq_Airy} and \eqref{appl_TW}, truncations of $\widetilde{Q}$ to compact sets $J$ are
performed. In this truncated setting universality results and rates of convergence for ensembles of type \eqref{density_P} with external fields of the form
$V \equiv \widetilde{Q}+\frac{1}{N} f$ are essential. It is here that general assumptions \GAii\ are needed, eventually yielding uniformity in the functions $f$.

\end{remark}

 The third application deals with universality of the nearest-neighbor spacings for the ensembles
\eqref{density_h}.

Usually, universality is understood in the sense of convergence of the correlation functions to some universal limit. In general,
many local statistics of interest can be expressed in terms of correlation functions, but to deduce universality of these statistics
from universality of correlation functions, some care is needed. This is 
no mere
mathematical problem, 
%as
because
in a statistical
experiment, one does not observe correlations but 
e.g.~spacings instead. Then the  
relevant quantity 
%here 
is the counting measure of
the nearest-neighbor spacings
and it is this object for which limit theorems are needed.
\begin{application}\label{application_3}
 For the ensemble $P_{N,Q}^h$ $($see \eqref{density_h}$)$, let $a\in\supp(\mu_Q^h)^\circ$, where
$\mu_Q^h$ is the limiting measure for $P_{N,Q}^h$ as described in Application \ref{application_2}. Let $I_N(a)$ denote an interval with center $a$ such that $\lv I_N(a)\rv\to0$
and $N\lv I_N(a)\rv\to\infty$ for $N\to\infty$ hold. Furthermore, we consider the measure $P_{N,Q}^h$ on the Weyl chamber, i.e.~we have $x_1\leq\dots\leq x_N$. 
Let $\psi := \mathrm d \mu_Q^h / \mathrm dx$ denote the density of $\mu_Q^h$ and introduce the rescaled position $\tilde{x}_j:=N \psi(a) x_j$
of the $j$-th particle.
Then we define the counting
measure of spacings around $a$ as
\begin{align*}
 S(I_N(a),x):=\sum_{x_j,x_{j+1}\in I_N(a)}\delta_{\tilde{x}_{j+1}-\tilde{x}_j}.
\end{align*}
We will prove in \cite{SchubertVenker} that 
\begin{align}
 \lim_{N\to\infty}\mathbb{E}_{N,Q}^h\lb\sup_{s\in\R}\left\lv \int_0^s\frac{1}{N\lv
I_N(a)\rv \psi(a)}\, \mathrm dS(I_N(a),x)-F_2(s)\right\rv\rb=0,\label{spacings}
\end{align}
where $F_2$ is the distribution function of the limiting spacing distribution of the Gaussian Unitary Ensemble and
$\mathbb{E}_{N,Q}^h$ denotes expectation w.r.t. $P_{N,Q}^h$. 
\end{application}
Assertion \eqref{spacings} expresses the uniform convergence of the
distribution function of the empirical spacings to its universal limit. Here $N\lv
I_N(a)\rv\psi(a)$ asymptotically normalizes $S(I_N(a),x)$.

The convergence analog to \eqref{spacings}, initially shown in \cite{KatzSarnak} for circular unitary ensembles, will be proved in \cite{Schubert} for invariant ensembles ($\beta=1,2,4$) (see also
\cite{KriecherbauerSchubert} for a survey on spacings in invariant ensembles).
 Prior to \cite{Schubert}, for unitarily invariant
ensembles only convergence pointwise in $s$ has been shown (see \cite{DKMVZ2}, \cite{Deift}), reading
\begin{align}
 \lim_{N\to\infty}\mathbb{E}_{N,V}\int_0^s\frac{1}{N\lv
I_N(a)\rv \psi(a)}\, \mathrm dS(I_N(a),\l)=F_2(s)\label{spacings2}
\end{align}
for each $s\in\R$ and with 
corresponding definitions of $\mathbb{E}_{N,V}$ and $\psi$.
The advantage of uniform convergence as in \eqref{spacings}
over pointwise convergence as in \eqref{spacings2}, is that it allows for quantile comparison of some empirical spacing distribution
to the spacing distribution $F_2$.

The paper is organized as follows. In Section \ref{sec2} we collect all information about the equilibrium measure that is needed in our analysis of the Riemann-Hilbert problem which is then performed in Section \ref{sec3}. A somewhat involved construction, which allows to control the singularities that arise at the edges, is shifted to the Appendix. Section \ref{sec4} is devoted to the proofs of our main results,
Theorems \ref{theorem_kernel}, \ref{theorem_asymptotics}, \ref{theorem_bulk}, and \ref{theorem_edge}.

\setcounter{equation}{0}
\section{The equilibrium measure and its log-transform}\label{sec2}

The function $\rho_V$, that we have defined in \eqref{def_rho}, can be viewed as a rescaled version of the density of the equilibrium
measure of logarithmic potential theory with respect to the external field $V$. We do not use this fact directly in our arguments.
Nevertheless, let us briefly recall what it means. Assume that $V$ satisfies \GAi. It is well known, also under far less
restrictive assumptions on $V$ (see e.g.~\cite{BPS}), that the energy functional
\begin{align*}
 I_V(\mu):=\int_J\int_J\log\lv t-s\rv^{-1}\, \mathrm d\mu(s)\mathrm d\mu(t)+\int_JV(t)\, \mathrm d\mu(t),
\end{align*}
defined on the Borel measures $\mu$ on $J$ with $\int_J \mathrm d\mu=1$, has a unique minimizer $\mu_V$. In the situation of \GAi\ it can be
written in the form 
\begin{align}\label{eq_dV}
\mathrm d\mu_V(x)=\hat{\rho}_V(x)\, \mathrm dx, \text{ with }\,\hat{\rho}_V(x):=\frac{2}{b_V-a_V}\rho_V(\l_V^{-1}(x))
\end{align}
and with $\l_V$, $a_V$, $b_V$, $\rho_V$ as introduced in Definition \ref{def_notation}. Uniqueness of the minimizer follows from the positive
definiteness of the quadratic part of the functional $I_V$ (see e.g.~\cite[Lemma 6.41]{Deift}), so that $\mu_V$ is characterized as the
solution of the Euler-Lagrange equation:
\begin{align}\label{eq_euler_lagrange}
&\exists l\in\R: \qquad 2\int_J\log\left|x-t\right|^{-1}\, \mathrm d\mu(t)+V(x) 
\begin{cases}
 =l &,\text{ if } x\in \supp(\mu),\\
\geq l &,\text{ if }x\in J\setminus \supp(\mu)
\end{cases}
\end{align}
For the convenience of the reader we sketch the proof that $\mu_V$ as defined in \eqref{eq_dV}
satisfies \eqref{eq_euler_lagrange}. We reemphasize that it is only \eqref{eq_euler_lagrange}, in the form of
Corollary \ref{cor_g} below, which is used in the analysis of the subsequent sections. The fact that $\mu_V$ minimizes $I_V$ is not relevant in the context of this paper and no proof will be given.

Observe that $\hat{\rho}_V$ as given by \eqref{eq_dV}, see also Definition \ref{def_notation}, is H\"older continuous  and therefore
\begin{align}\label{hilbert_deriv}
&\frac{\mathrm d}{\mathrm dx}\int_J\log\left|x-t\right|^{-1}\, \mathrm d\mu_V(t)=\pi(H\hat{\rho}_V)(x), \ x\in\R,
\end{align}
with the Hilbert transform $H$ being defined by 
\begin{align*}
&Hf(x):=\frac{1}{\pi}\text{PV}\int_J\frac{f(t)}{t-x}\, \mathrm dt,\ x\in\R.
\end{align*}
In the following Lemma \ref{lemma_Hilbert}, together with Remark \ref{remark_mu}, we derive an explicit formula for
$2\pi(H\hat{\rho}_V)+V'$ on $J$, from which \eqref{eq_euler_lagrange} is immediate. We also prove $\int_J\mathrm d\mu_V=1$.
The method of proof is taken from \cite{DKM}; see also \cite[Section 6.7]{Deift} for some motivation of the auxiliary functions that appear in the arguments below. A different method to obtain explicit representations for the equilibrium measure was introduced in \cite{Pastur1996}, see also \cite[Section 11.2]{bookPasturS}.

\begin{lemma}\label{lemma_Hilbert}
Let $V$ satisfy assumption \GAi\ and let $\rho_V$, $\l_V$, $G_V$ be given as in Definition \ref{def_notation}.
Then
\begin{enumerate}[(a)]
 \item \begin{align*}
        \int^{1}_{-1}\rho_V(x)\, \mathrm dx=1 \qquad \text{ and }\qquad \rho_V(x)>0 \quad \text{ for all }x\in(-1,1).
       \end{align*}
\item \begin{align*}
       2\pi(H\rho_V)(x)+(V\circ\l_V)'(x)=
\begin{cases}
0 &, \text{ if } x\in[-1,1],\\
\sgn(x)G_V(x)\sqrt{x^2-1} &, \text{ if }x\in\hat{J}\setminus[-1,1].
\end{cases}
      \end{align*}
\end{enumerate}
\end{lemma}

\begin{remark}\label{remark_mu}
\begin{enumerate}[(a)]
\item
It follows from Lemma \ref{lemma_Hilbert} (a) and from the definition of $\rho_V$ that $\mu_V$ as given by
\eqref{eq_dV} is a Borel measure on $J$ with support $[a_V, b_V]$ and $\int_J\mathrm d\mu_V=1$. Moreover, 
\begin{align*}
2\pi(H\hat{\rho}_V)+V'=\frac{2}{b_V-a_V}\left[2\pi(H\rho_V)+(V\circ\l_V)'\right]\circ\l^{-1}_V
\end{align*}
vanishes on $[a_V, b_V]$ and is strictly positive resp.~negative to the right resp.~left of $[a_V, b_V]$ by statement (b) of the
above Lemma. Together with \eqref{hilbert_deriv} this implies that $\mu_V$ satisfies the Euler-Lagrange equation \eqref{eq_euler_lagrange}.

\item
Observe that the functional $I_V$ has a unique minimizer also for those $V$ that do not satisfy condition (4) of \GAi. In this case
the minimizer $\mu_V$ still has a density but it is unbounded in the vicinity of at least one endpoint of $J$. Such an endpoint is
called a hard edge. In all cases considered in the present paper, however, the density of $\mu_V$ vanishes at the endpoints which is the defining property of a soft edge.
\end{enumerate}
\end{remark}
\begin{proof}[Proof of Lemma \ref{lemma_Hilbert}]
In the proof we simplify the notation by using $\rho\equiv\rho_V$, $h\equiv h_V$, $G\equiv G_V$ (see
\eqref{def_rho}, \eqref{def_h}, \eqref{def_G}). In addition, we introduce the auxiliary function 
\begin{align}\label{def_W}
W:=V\circ\l_V
\end{align}
that is defined on $\hat{J}$ (see Definition \ref{def_notation}). For example, we have \eqref{def_h}
\begin{align*}
h(t,x)=\frac{W'(t)-W'(x)}{t-x}=\int^1_0W''(x+u(t-x))\, \mathrm du.
\end{align*}
It follows from the strict monotonicity of $V'$, which is inherited by $W'$, that $h$ is positive except possibly on the diagonal. Hence
$G>0$ on $\hat{J}$ and $\rho>0$ on $(-1,1)$. In order to prove the remaining claims, we introduce the unique function $q$ that is
analytic on $\C\setminus[-1,1]$ satisfying in addition
\begin{align}\label{def_q}
q(z)^2=(z^2-1) \quad \text{ and } \quad q(x)=\sqrt{x^2-1} \quad\text{for }x>1.
\end{align}
Moreover, restricting $q$ to the upper resp.~lower half-plane, we denote by $q_\pm$ the extension to the real line, i.e.
\begin{align*}
q_\pm(x):=\lim_{\e\searrow 0}q(x\pm i\e) =
\begin{cases}
\sqrt{x^2 - 1} &, \text{ if } x > 1\, ,\\
\pm i \sqrt{1 - x^2}&, \text{ if } x\in[-1,1]\, ,\\
- \sqrt{x^2 - 1} &, \text{ if } x < -1\, .
\end{cases}
\end{align*}
Furthermore, recall the definition of the Cauchy-transform with respect to $\R$:
\begin{align}\label{def_cauchy_transform}
(Cf)(z)=\frac{1}{2\pi i}\int_\R\frac{f(t)}{t-z}\, \mathrm dt \, , \quad z \in \C \setminus \R \, .
\end{align}
The key step in the proof is to relate $\rho$ to the auxiliary function
\begin{align*}
F:= q\cdot C\lb\frac{iW'}{\pi q_+}\mathbbmss{1}_{[-1,1]}\rb,
\end{align*}
extended analytically to $\C\setminus[-1,1]$, where $\dopp1_A$ denotes the characteristic function corresponding to $A$. Using
\begin{align*}
W'(t)=W'(x)+h(t,x)(t-x),
\end{align*}
we obtain for $z\in\C\setminus[-1,1]$ that
\begin{align*}
F(z)=\frac{q(z)}{2\pi i}\left[W'(x)\int^1_{-1}\frac{i}{\pi q_+(t)}\frac{1}{t-z}\, \mathrm dt+\int^1_{-1}\frac{ih(t,x)}{\pi
q_+(t)}\frac{t-x}{t-z}\, \mathrm dt\right].
\end{align*}
A straightforward residue calculation (with a vanishing residue at $\infty$) yields
\begin{align*}
\int^1_{-1}\frac{i}{\pi q_+(t)}\frac{1}{t-z}\, \mathrm dt=-\frac{1}{q(z)}.
\end{align*}
Choosing $x=\Re(z)$, we arrive by dominated convergence at 
\begin{align}\label{eq_F_pm}
F_\pm(x):=\lim_{\e\searrow 0}F(x\pm i\e)=-\frac{W'(x)}{2\pi i}+\frac{q_\pm(x)}{2\pi i}G(x)\ \text{ for } x\in\hat{J}. 
\end{align}
Recalling the definition of $\rho$ (see \eqref{def_rho}), we conclude 
\begin{align}\label{eq_rho}
\rho=\Re(F_+) \quad \text{ and } \quad2\rho=F_+-F_-  \quad \text{ on } \R,
\end{align}
where we have used in addition that $\rho$, $\Re(F_+)$, and $F_+-F_-$ vanish identically on $\R\setminus \left[ -1,1\right]\supset \R\setminus\hat{J}$. Now we have reached the only
point in the proof where condition (4) of \GAi\, comes into play. Translated to $W$ it reads
\begin{align*}
\int^1_{-1}\frac{W'(s)}{\sqrt{1-s^2}}\, \mathrm ds=0\quad \text{ and } \quad &\int^1_{-1}\frac{s \, W'(s)}{\sqrt{1-s^2}}\, \mathrm ds=2\pi.
\end{align*}
As a consequence we can determine the leading order behavior of $F(z)$ for $\left|z\right|\rightarrow\infty$.
It follows from 
\begin{align}\label{eq_cauchy_kernel}
\frac{1}{t-z}=-\frac{1}{z}-\frac{t}{z^2}-\frac{t^2}{z^3}+\frac{t^3}{z^3(t-z)}
\end{align}
that
\begin{align*}
C\lb\frac{iW'}{\pi
q_+}\dopp1_{[-1,1]}\rb(z)=-\frac{1}{2\pi^2i}\lb\frac{0}{z}+\frac{2\pi}{z^2}+\O\lb\left|z\right|^{-3}\rb\rb \ \text{ as } 
\left|z\right|\rightarrow\infty.
\end{align*}
Since $q(z)=z( 1+\O(|z|^{-2}))$, for $\left|z\right|\rightarrow\infty$, we have
\begin{align}\label{eq_F}
i\pi F(z)=-\frac{1}{z}+\O\lb\left|z\right|^{-2}\rb \ \text{ as } \left|z\right|\rightarrow\infty.
\end{align}
A first conclusion of \eqref{eq_F} is 
\begin{align*}
\frac{1}{2\pi i}\int_\R\frac{F_+(x)-F_-(x)}{x-z}\, \mathrm dx=F(z),\ z\in\C\setminus[-1,1],
\end{align*}
by a residue calculation (again with a vanishing residue at $\infty$). Recalling \eqref{eq_rho}, we have proved 
\begin{align}\label{eq_F2}
F=2C\rho.
\end{align}
Since
\begin{align*}
2\pi i(C\rho)(z)=-\frac{1}{z}\lb\int^1_{-1}\rho(x)\, \mathrm dx\rb+\O\lb\left|z\right|^{-2}\rb, \quad\left|z\right|\rightarrow\infty
\end{align*}
by \eqref{eq_cauchy_kernel}, claim (a) follows from \eqref{eq_F}.

Recall that the H\"older continuity of $\rho$ implies the existence of the pointwise limits 
\begin{align}\label{cauchyproperty}
(C_{\pm}\rho)(x) :=\lim_{\e\searrow
0}(C\rho)(x \pm i\e) =\frac{1}{2}\left[ \pm \rho(x) - i(H\rho)(x)\right] .
\end{align}
Thus, by \eqref{eq_F2} and \eqref{eq_F_pm},
\begin{align*}
H\rho=-2\Im(C_+\rho)=-\Im(F_+)=-\frac{W'}{2\pi}+\frac{\Re(q_+)}{2\pi}G
\end{align*}
on $\hat{J}$. As $\Re(q_+)$ vanishes on $[-1,1]$ and is equal to $\sgn(x)\sqrt{x^2-1}$ on $\R\setminus[-1,1]$, we have proved
statement (b).

\end{proof}

The results of Lemma \ref{lemma_Hilbert} provide information on the log-transform
\begin{align}\label{def_g}
\map{g}{\C\setminus(-\infty,1]}{\C},\ g(z):=\int^1_{-1}\log(z-t)\rho_V(t)\, \mathrm dt
\end{align}
of $\rho_V$ that will be essential for the analysis of the Riemann-Hilbert problem in the subsequent sections. This information is summarized in the following
corollary. 
\begin{corollary}\label{cor_g}
Let $V$ satisfy \GAi\ and let $\xi_V$, $\eta_V$, $\rho_V$, $W$, $g$ be given as in
Definition \ref{def_notation}, \eqref{def_W}, and \eqref{def_g}. Then $g$ is analytic and
there exists $l\in\R$ such that
\begin{align}\label{eq_cor_g1}
	g_+(x)-g_-(x)=i\xi_V(x) \ \text{ for } x\in\R,
\end{align} 
\begin{align}\label{eq_cor_g2}
	g_+(x)+g_-(x)=W(x)+l-\eta_V(x) \ \text{ for } x\in\hat{J},
\end{align} 
\begin{align}\label{eq_cor_g4}
	\xi _V(x)=2\pi\ \text{ for }x\leq -1 \quad \text{ and } \quad \xi _V(x)=0\ \text{ for } x\geq 1,
\end{align} 
\begin{align}\label{eq_cor_g3}
	g(z)=\log z+\O\lb\left|z\right|^{-1}\rb \ \text{ as }  \left|z\right|\rightarrow\infty.
\end{align}
Moreover, the function $e^g$ has an analytic extension onto $\C\backslash \left[ -1,1\right] $.

\end{corollary}
\begin{proof}
The limits $g_\pm$ of $g$ on $\R$ are given by 
\begin{align*}
g_\pm(x)=\int^1_{-1}\log\left|x-t\right|\rho_V(t)\, \mathrm dt\pm i\pi\int^1_x\rho_V(t)\, \mathrm dt,\ x\in\R.
\end{align*}
Relations \eqref{eq_cor_g1} and \eqref{eq_cor_g2} now follow from statement (b) of Lemma
\ref{lemma_Hilbert} and from
\begin{align*}
\frac{\mathrm d}{\mathrm dx}\lb\int^1_{-1}\log\left|x-t\right|\rho_V(t)\, \mathrm dt\rb=-\pi(H\rho_V)(x)\ \text{ for } x\in\R \quad (\text{cf.~}  \eqref{hilbert_deriv}).
\end{align*}
The relations in (\ref{eq_cor_g4}) follow from Lemma \ref{lemma_Hilbert} (a) and (\ref{def_xi}). The first relation of \eqref{eq_cor_g4} explains the analytic extendability of $e^g$ across $\lb -\infty ,-1\rb$.
The asymptotic formula \eqref{eq_cor_g3} is again a consequence of Lemma \ref{lemma_Hilbert} (a).

\end{proof}

We conclude this section by proving the first claim of statement (b) in our Theorem \ref{theorem_kernel}.

\begin{lemma}\label{lemma_q}
Assume that \GAii\ holds for some function $Q$. Let $D$, $X_D$ be defined as in Remark \ref{remark_complex_entension}. Then there exists an $\epsilon>0$ such that any $V\in B_{\epsilon}\lb Q\rb \subset X_D$, restricted to $J$, satisfies \GAi.
Moreover, the maps $V\mapsto a_V$, $V\mapsto b_V$, defined on $B_{\epsilon}\lb Q\rb$ by \GAi\,(4) are $C^1$ with bounded derivatives.
\end{lemma}

\begin{proof}
Observe first that condition (1) of \GAi\ holds because of the definition of $X_D$ and that \GAi\ (3) does not pose any condition due to the boundedness of $J$ (see \GAii\ (3\,$^\prime$)). Since $\Vert V''\big|_J\Vert _{\infty} \leq C \Vert V\Vert _{\infty}$ for some $C>0$ that only depends on the distance between $J$ and $\mathbb{C}\backslash D$, we can ensure \GAi\ (2) on $B_{\epsilon}\lb Q\rb$ for $\epsilon$ sufficiently small by using \GAii\ (2\,$^\prime$). In order to investigate \GAi\ (4) we introduce
\begin{align*}
&U:=\lbrace \lb a,b\rb\in\lb L_-,L_+\rb ^2 \,|\, a<b\rbrace ,\\
&\lambda_{a,b}\lb s\rb := \frac{b-a}{2}s+\frac{b+a}{2},\ s\in \mathbb{R},\ \text{ for }\lb a,b\rb \in U \quad (\text{cf.~}  \eqref{def_lambda}) ,\\
&\map{A}{X_D\times U}{\mathbb{R}^2},\ A\lb V,a,b\rb :=\vec{\int_{-1}^1 \frac{V'(\lambda_{a,b}(s))} {\sqrt{1-s^2}}\, \mathrm ds \vspace{5pt}\\
 (b-a)\int_{-1}^1 \frac{sV'(\lambda_{a,b}(s))}{\sqrt{1-s^2}}\, \mathrm ds-4\pi}.
\end{align*}
A short calculation shows that (\ref{MRS_number}) holds iff $A(V,a,b)=0$. Hence, by \GAii, we have $A(Q,a_Q,b_Q)=0$. Using again that derivatives of $V\big|_J$ are uniformly bounded by $\Vert V\Vert_{\infty}$, it is straightforward to verify that $A$ is a continuously differentiable map with bounded derivatives on bounded domains. By the implicit function theorem we only need to show that $[\partial _{(a,b)} A] (Q,a_Q,b_Q)$ has a non-vanishing determinant. This can be seen as follows:
Introducing the measure $\alpha$ on $[-1,1]$ with density $\mathrm d\alpha / \mathrm ds := Q''\lb\lambda_{a_Q,b_Q}\lb s\rb\rb / \sqrt{1-s^2}$, denoting its moments by $m_k:=\int_{-1}^1 s^k\, \mathrm d\alpha (s)$, and keeping in mind that $A_2(Q,a_Q,b_Q)=0$ one obtains
\begin{align*}
\lb \partial_{(a,b)} A\rb \lb Q,a_Q,b_Q\rb =
\begin{pmatrix}
\frac{1}{2} \lb m_0-m_1\rb & \frac{1}{2}\lb m_0+m_1\rb \vspace{5pt}\\
-\frac{4\pi}{b_Q-a_Q}+\frac{b_Q-a_Q}{2}\lb m_1-m_2\rb & \frac{4\pi}{b_Q-a_Q}+ \frac{b_Q-a_Q}{2} \lb m_1+m_2\rb
\end{pmatrix}
\end{align*}
with determinant $\frac{4\pi}{b_Q-a_Q}m_0+\frac{b_Q-a_Q}{2}\lb m_0 m_2-m_1^2\rb$. Since $m_1^2\leq m_0 m_2$ by the Cauchy-Schwarz inequality we have derived
\begin{align*}
\det \left[ \lb \partial_{(a,b)}A\rb \lb Q, a_Q,b_Q\rb\right] \geq \frac{4\pi}{b_Q-a_Q}m_0>0.
\end{align*}
\end{proof}

Finally, we note that the proof of Lemma \ref{lemma_q} requires less restrictive assumptions than \GAii. On the one hand, analyticity could be replaced by a finite regularity assumption (e.g.~$C^3$ would easily suffice). Analyticity will only be used in the next section in a crucial way. On the other hand, the strict positivity of $Q''$ as formulated in \GAii\ (2\,$^\prime$) could be replaced by the weaker \GAi\ (2) and $m_0>0$ used in the proof would still hold. However, in this situation not all $V\in B_{\epsilon}\lb Q\rb$ necessarily satisfy \GAi\ (2) no matter how small one chooses $\epsilon$. Then it is not a priori clear whether the corresponding functions $\rho_V$ are non-negative. In case there exists a point $x_0$ with $\rho_V(x_0)<0$, a gap in the support of the equilibrium measure $\mu_V$ opens and entirely different formulae for $\rho_V$, $G_V$, etc. are needed (see e.g.~\cite{DKM}).

\setcounter{equation}{0}
\section{The Riemann-Hilbert problem for the Christoffel-Darboux kernel}\label{sec3}

This section combined with the Appendix contains everything about Riemann-Hilbert problems that is used in this paper. We follow closely the path laid out in \cite{DKMVZ1} (see also \cite[Section 7]{Deift})  and start with the Fokas-Its-Kitaev characterization \cite{FIK} of orthogonal polynomials by Riemann-Hilbert problems (RHP). The asymptotic analysis is then performed using the Deift-Zhou nonlinear steepest descent method. In all of this section we assume that $V$ satisfies \GAi. As in the previous section we find it convenient to begin our analysis with the linear rescaling $\lambda _V$ (see (\ref{def_lambda})). As in (\ref{def_W}), let
\begin{align}\label{definition_W}
W:= V \circ \lambda _V,\ \map{W}{\hat{J}}{\mathbb{R}},\ \hat{J}=\lambda _V^{-1} \lb J\rb =\lbrace x \in \R \, |\,\hat{L}_-\leq x\leq \hat{L}_+\rbrace .
\end{align}
From the defining orthogonality relations for $p_j^{(N,V)}$ resp.~$p_j^{(N,W)}$ and from the definition of the corresponding Christoffel-Darboux kernels (cf.~formulae above \eqref{Chis_Darboux}), one obtains
\begin{align}
& p_j^{(N,W)}=\sqrt{\frac{b_V-a_V}{2}}p_j^{(N,V)}\circ \lambda _V,\ \mbox{and}\nonumber \\
& K_{N,W}\lb x,y\rb =\frac{b_V-a_V}{2}K_{N,V}\lb \lambda _V \lb x\rb ,\lambda _V \lb y\rb \rb ,\ \text{ for } x,y\in \hat{J}.  
\label{K_N}
\end{align}
The following theorem provides the Riemann-Hilbert formulation for the Christoffel-Darboux kernel $ K_{N,W}$ (see e.g.~\cite{Deift, DKMVZ1, KuijlaarsVanlessen, Vanlessen}).

\begin{theorem}\label{theorem_rhp}
Assume that $V$ satisfies \GAi\ and define $W$ as in (\ref{definition_W}). There exists a unique analytic map
\begin{align*}
\map{Y}{\mathbb{C}\backslash \hat{J}}{\mathbb{C}^{2\times 2}}
\end{align*}
satisfying (i)-(iii), where
\begin{enumerate}
\item[(i)] 
	The restrictions $Y\big|_{\mathbb{C}_{\pm}}$of $Y$ to the upper resp.~lower half-plane  can be extended continuously to maps $Y_{\pm}$ defined on $\overline{\mathbb{C}}_{\pm}\backslash \lbrace 		\hat{L}_{\pm}\rbrace$ satisfying
	\begin{align*}
	Y_+\lb x\rb = Y_-\lb x\rb 
	\begin{pmatrix}
	1 & e^{-NW\lb x\rb}\\ 0 & 1
	\end{pmatrix},\ \text{ for } x\in\lb \hat{L}_-,\hat{L}_+\rb .
	\end{align*}
\item[(ii)]
	$\displaystyle\lim_{|z|\to\infty}Y\lb z\rb 
	\begin{pmatrix}
	z^{-N} & 0 \\ 0 & z^N
	\end{pmatrix}
	=\begin{pmatrix}
	1 & 0 \\ 0 & 1
	\end{pmatrix}=:I$
\item[(iii)]
	$Y\lb z\rb =\begin{pmatrix}
	\mathcal{O}\lb 1\rb & \mathcal{O}\lb |\log |z-\hat{L}_{\pm}||\rb \\
	\mathcal{O}\lb 1\rb & \mathcal{O}\lb |\log |z-\hat{L}_{\pm}||\rb 
	\end{pmatrix},\ $ for $z\to \hat{L}_{\pm}$, in case $\hat{L}_{\pm}$ is finite.
\end{enumerate}
The unique solution is given by
\begin{align}\label{solution_rhp}
Y=\begin{pmatrix}
\left[\gamma _N^{(N,W)}\right]^{-1} p_N^{(N,W)} \phantom{XXX}& \left[\gamma _N^{(N,W)}\right]^{-1} C\lb p_N^{(N,W)}e^{-NW}\rb \vspace{8pt}\\
-2\pi i \gamma _{N-1}^{(N,W)}p_{N-1}^{(N,W)} \phantom{XXX}& -2\pi i \gamma _{N-1}^{(N,W)} C\lb p_{N-1}^{(N,W)}e^{-NW}\rb
\end{pmatrix},
\end{align}
where $C$ denotes the Cauchy transform (see Remark \ref{remark_cauchy} below). 

Moreover, $\det Y\lb z\rb =1$ for all $z\in\mathbb{C}\backslash \hat{J}$ and
\begin{align}
K_{N,W}\lb x,y \rb = e^{-\frac{N}{2}\lb W\lb x\rb + W\lb y\rb\rb}\frac{1}{2\pi i \lb x-y\rb} \begin{pmatrix}
0 & 1
\end{pmatrix} Y_+^{-1}\lb y \rb Y_+\lb x\rb \begin{pmatrix}
1 \\ 0
\end{pmatrix}
\label{K_N:2}
\end{align}
for $x,y\in\hat{J}$.
\end{theorem}

\begin{remark}\label{remark_cauchy}
In order to use the definition of the Cauchy transform as given in (\ref{def_cauchy_transform}),  we extend $e^{-NW}$ by 0 on $\R\backslash\hat{J}$.
\end{remark}

A detailed proof of the above theorem in the case $\hat{J}=\R$ can be found in \cite{Deift}. For the convenience of the reader we sketch the main ideas, emphasizing those aspects that are different in the case of finite $\hat{L}_{\pm}$.
\begin{proof}[Sketch of proof of Theorem \ref{theorem_rhp}]
Let us begin with the first column $\vec{f\\ g}:=\vec{Y_{11}\\ Y_{21}}$ of $Y$. From (i) and (iii) it follows that $f$ and $g$ can be extended to entire functions. By (ii) we learn from a version of Liouville's theorem that $f$ is a polynomial of degree $N$ with leading coefficient 1 and that $g$ is a polynomial of degree $<N$. Applying the basic relation $h=C_+h-C_-h$ (cf.~\eqref{cauchyproperty}) to $h=fe^{-NW}\dopp{1}_{\hat{J}}$ and $h=ge^{-NW}\dopp{1}_{\hat{J}}$, it is straightforward to verify that
\begin{align*}
Y=\begin{pmatrix}
f & C\lb fe^{-NW}\rb  \\ g & C\lb ge^{-NW}\rb 
\end{pmatrix}
\end{align*}
has all desired properties, except possibly relation (ii) for the second column, i.e. 
\begin{align}\label{Y_second_column}
\lim_{|z|\to\infty} Y\left(z\right)\vec{0 \\ z^N}=\vec{0\\ 1} .
\end{align}
Using the summation formula for geometric series, one derives
\begin{align*}
C\lb fe^{-NW}\rb = \sum_{j=0}^N \frac{i}{2\pi z^{j+1}} \int_{\hat{J}} f\lb x\rb x^j e^{-NW\lb x\rb }\, \mathrm dx+\frac{1}{z^{N+1}} C\lb x^{N+1}fe^{-NW}\rb.
\end{align*}
By the assumptions on $V$, transferred to $W$, $C\lb x^{N+1}fe^{-NW}\rb$ is bounded except for neighborhoods of $\hat{L}_{\pm}$, in case they are finite. Thus, the first component of (\ref{Y_second_column}) implies $\int_{\hat{J}} f\lb x\rb x^j e^{-NW\lb x\rb} \, \mathrm dx=0$ for each $0\leq j\leq N-1$. As already noted, $f$ is a polynomial of degree $N$ with leading coefficient $1$ and therefore $f=\left[\gamma _N^{(N,W)}\right]^{-1} p_N^{(N,W)}$. A similar reasoning leads to $g= \alpha p_{N-1}^{(N)}$, where $\alpha=-2\pi i \gamma_{N-1}^{(N,W)}$ follows from
\begin{align*}
1&=\lim_{|z|\to\infty} Y_{22}\lb z\rb z^N = \frac{i}{2\pi} \int _{\hat{J}}\alpha p_{N-1}^{(N,W)}\lb x\rb x^{N-1} e^{-NW\lb x\rb} \, \mathrm dx\\
&=\frac{i\alpha}{2\pi \gamma _{N-1}^{(N,W)}}\int_{\hat{J}}\lb p_{N-1}^{(N,W)}\lb x\rb \rb ^2 e^{-NW\lb x\rb }\, \mathrm dx = \frac{i\alpha}{2\pi \gamma _{N-1}^{(N,W)}}.
\end{align*}
So far we have argued that $Y$ as defined by (\ref{solution_rhp}) indeed solves the RHP (i)-(iii).

To prove uniqueness one may appeal to Liouville's theorem. Let $Z$ denote any solution of RHP (i)-(iii) and set $d\lb z\rb := \det \lb Z\lb z\rb\rb$ for $z\in\C\backslash\hat{J}$. Since $d_+\lb x\rb = d_-\lb x\rb$ on $( \hat{L}_-,\hat{L}_+)$ by (i), and $d\lb z \rb =\mathcal{O}( |\log |z-\hat{L}_{\pm}||)$ for $z$ near $\hat{L}_{\pm}$ by (iii), $d$ must have an entire extension. Condition (ii) yields $\lim_{|z|\to\infty}d\lb z\rb =1$ from which $d\lb z\rb =1$ follows for all $z\in\C\backslash\hat{J}$. Therefore $M:=YZ^{-1}$ defines an analytic function on $\C\backslash \hat{J}$ ($Y$ as given in (\ref{solution_rhp})) with $M_+=M_-$ on $(\hat{L}_-,\hat{L}_+)$, $M\lb z\rb=\mathcal{O}( |\log |z-\hat{L}_{\pm}||)$ near $\hat{L}_{\pm}$ and $\lim_{|z|\to\infty}M\lb z\rb =I$. Using Liouville's theorem again, we see that $M\lb z\rb=I$ and $Y\lb z\rb= Z\lb z\rb$ for all $z\in\C\backslash\hat{J}$. 

Finally, we need to verify the formula for the Christoffel-Darboux kernel $K_{N,W}$. From (\ref{solution_rhp}) and (\ref{Chis_Darboux}) we learn
\begin{align*}
K_{N,W}\lb x,y\rb = \frac{\lb Y_{11}\rb _+\lb x\rb \lb Y_{21}\rb _+\lb y\rb - \lb Y_{21}\rb _+\lb x\rb \lb Y_{11}\rb _+\lb y\rb}{-2\pi i \lb x-y\rb}
e^{-\frac{N}{2}\lb W\lb x\rb + W\lb y\rb\rb}
\end{align*}
for $x,y\in\hat{J}$ 
(recall, that $Y_{11}$ and $Y_{21}$ are polynomials and therefore $Y_{j1}=\lb Y_{j1}\rb _+$). 
As observed in the proof of uniqueness above, $\det Y_+\lb y\rb =1$ holds for all $y\in\hat{J}$ and therefore $\begin{pmatrix}
0&1
\end{pmatrix}Y_+^{-1}\lb y\rb = \begin{pmatrix}
\lb -Y_{21}\rb _+\lb y\rb & \lb Y_{11}\rb _+\lb y\rb 
\end{pmatrix}$, completing the proof.
\end{proof}

Let us now turn to the asymptotic analysis of the Riemann-Hilbert problem formulated in Theorem \ref{theorem_rhp}. Again, we follow closely \cite{Deift, DKMVZ1}.
The basic strategy is to perform explicit transformations that lead to an equivalent RHP with jump relations $R_+=R_- \lb I + \mathcal{O}\left( \frac{1}{N}\right)\rb$ and asymptotic condition $R(z) \to I$ for $|z| \to \infty$.
By basic functional analytic considerations one obtains $R(z)=I + \mathcal{O}\left( \frac{1}{N}\right)$, which is sufficient for proving our results.

In order to simplify notation, we suppress the dependency on $V$ when we discuss the above mentioned transformations, i.e.~$ G\equiv G_V$, $\rho \equiv \rho_V$, $\xi \equiv \xi_V$, $\eta \equiv \eta_V$ (see \eqref{def_G} to \eqref{def_eta})  as well as dependencies on $N$. Denoting the unique solution of the RHP in Theorem \ref{theorem_rhp} by $Y$, we define
\begin{equation}\label{def_T}
T(z) :=e^{- N \frac{l}{2} \sigma_3} Y(z) e^{- N\lb g(z) - \frac{l}{2}\rb \sigma_3}, \quad z\in \mathbb{C}\setminus \hat{J}.
\end{equation} 
Here $l$ is the Lagrange multiplier (see \eqref{eq_euler_lagrange}, \eqref{eq_cor_g2}),
$g$ the log-transform of $\rho$ \eqref{def_g},
and $\scriptsize \sigma_3 := \begin{pmatrix}
1 & 0 \\ 0 & -1 \end{pmatrix}$ denotes the third Pauli matrix, i.e.~$\scriptsize e^{a \sigma_3}= \begin{pmatrix}
e^a & 0 \\ 0 & e^{-a} \end{pmatrix}$ for $a \in \mathbb C$.
Recall also from Corollary \ref{cor_g} that $e^{Ng}$ can be viewed as an analytic function on $\mathbb C \setminus [-1,1] \supset  \mathbb C \setminus \hat{J}$. A straightforward calculation yields that $T$ also satisfies a RHP, similar to the one in Theorem  \ref{theorem_rhp}, but with (i) and (ii) replaced  by 
\begin{enumerate}
\item[$\text{(i)}_T$]
$T_+=T_- v_T\quad $on $\quad (\hat{L}_-, \hat{L}_+), \quad$ where  $v_T = \begin{pmatrix} e^{-N (g_+ - g_-)} & e^{N(g_+ + g_- -W-l)}
\\ 0 & e^{N(g_+ - g_-)} \end{pmatrix}$
\item[$\text{(ii)}_T$]
$\displaystyle \lim_{|z| \to\infty} T(z) =I$ \quad (use \eqref{eq_cor_g3}).
\end{enumerate}
\label{18}
Observe that $T$ already has the desired asymptotic behavior for $|z|\to \infty$. 
Note in addition that the results of Section \ref{sec2}, summarized in Corollary \ref{cor_g}, provide the following representation for the jump matrix $v_T$,
\begin{equation}
\label{def_v_T}
v_T= \begin{cases}
\begin{pmatrix}
1 & e^{-N \eta} \\0 &1
\end{pmatrix} & \text{on } \hat{J}\setminus [-1,1], \\ \\
\begin{pmatrix}
e^{-iN \xi} & 1 \\0 &e^{iN \xi}
\end{pmatrix} & \text{on }  [-1,1]. 
\end{cases}
\end{equation}
Since $\eta >0$ by definition, $v_T$ is of the form $I$+\textit{small} away from $[-1,1]$. On the critical interval $[-1,1]$ the jump matrix has rapidly oscillating terms on the diagonal. 

The nonlinear steepest descent method was designed to handle such jump conditions. 
In our situation this method is based on the factorization 
\begin{equation}\label{RHP_factor}
v_T= 
\begin{pmatrix}
e^{-iN \xi} & 1 \\ 0 & e^{iN\xi}
\end{pmatrix}
=
\begin{pmatrix}
1 &  0 \\ e^{iN \xi} & 1 
\end{pmatrix}
\begin{pmatrix}
0 &  1 \\ -1 & 0 
\end{pmatrix}
\begin{pmatrix}
1 &  0 \\ e^{-iN \xi} & 1 
\end{pmatrix}=: v_l v_0 v_u.
\end{equation}
As shown in Corollary \ref{Koro_RHP} below (see also Definition \ref{Def_xi_eta_complex} and Lemma \ref{lemma_G_V}), $\xi$ has an analytic continuation to a neighborhood of $(-1,1)$ with $(\text{Im}(\xi(z)))\cdot (\text{Im}(z)) < 0 $, i.e.  
\begin{equation*}
v_l= I + \textit{small} \quad \text{below } (-1,1), \quad \quad v_u= I + \textit{small} \quad \text{above } (-1,1).
\end{equation*}
The basic idea is as follows: Define $S$  as displayed in Figure \ref{fig:1}.
%
%
% FIGURE
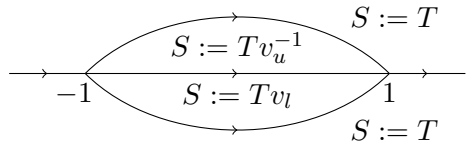
\begin{figure}[h]
\begin{tikzpicture}

% Punkte; Achtung: x-Koordinaten von B-A müssen gleich E-G sein
\coordinate
%[label=45:$A$]  % diese Zeile auskommentieren um Label V verschwinden zu lassen
(A) at (-7,0);
\coordinate
[label=268:$-1 \phantom{-}$] % diese Zeile auskommentieren um Label V verschwinden zu lassen
(B) at (-6, 0);
\coordinate
[label=270:$1$]   % diese Zeile auskommentieren um Label  verschwinden zu lassen
(C) at (-2,0);
\coordinate
%[label=0:$D$]  % diese Zeile auskommentieren um Label  verschwinden zu lassen
(D) at (-1,0);

% Ablenkungspunklte oben; Symmetrie recht <-> links durch: 0.25+0.75=1; Punkte weiter außen z.b. 0.25 ->0.1 und 0.75->0.9 (kommt geht im steileren Winkel in Endpubkte)
% Änderung der Höhe: Ersetze (0,1) z.B. durch (0,2)
\coordinate
%[label=0:$E$]  % diese Zeile auskommentieren um Label E verschwinden zu lassen
(E) at ($(B)+0.25*(C)-0.25*(B)+(0,1)$);
\coordinate
%[label=0:$F$]  % diese Zeile auskommentieren um Label F verschwinden zu lassen
(F) at ($(B)+0.75*(C)-0.75*(B)+(0,1)$);

%Ablenkungspunklte unten; Symmetrie recht <-> links durch: 0.25+0.75=1; 
% Änderung der Höhe: Ersetze (0,1) z.B. durch (0,2)
\coordinate
%[label=0:$G$]  % diese Zeile auskommentieren um Label G verschwinden zu lassen
(G) at ($(B)+0.25*(C)-0.25*(B)+(0,-1)$);
\coordinate 
%[label=0:$H$]  % diese Zeile auskommentieren um Label H verschwinden zu lassen
(H) at ($(B)+0.75*(C)-0.75*(B)+(0,-1)$);

%Linien einzeichnen incl Pfeilspitzen auf halber Strecke; keine Änderungen nötig
 \draw[decoration={markings, mark=at position 0.5 with {\arrow{>}}},
       postaction={decorate}](A) -- (B) ;
\draw[decoration={markings, mark=at position 0.5 with {\arrow{>}}},
       postaction={decorate}] (B)--(C)node[pos=.5,sloped,above] {$S:= Tv_u^{-1}$}node[pos=.5,sloped,below] {$S:= Tv_l$};
\draw[decoration={markings, mark=at position 0.5 with {\arrow{>}}},
       postaction={decorate}] (C)--(D);

\draw[decoration={markings, mark=at position 0.5 with {\arrow{>}}},
       postaction={decorate}] (B) .. controls (E) and (F) .. (C) node[pos=.8, above] { \quad \quad \qquad $ S:= T$};
\draw[decoration={markings, mark=at position 0.5 with {\arrow{>}}},
       postaction={decorate}] (B) .. controls (G) and (H) .. (C) node[pos=.8, below] { \quad \quad \qquad $ S:= T$};
\end{tikzpicture}
\caption{Definition of the matrix-valued function S} \label{fig:1}
\end{figure}
Observe that $S$ then has jumps across all of the oriented contour displayed in Figure \ref{fig:1}. The orientation is indicated by the arrows and it defines the $+$ side of the contour to lie on its left-hand side by standard convention. The jump relations read
\begin{align*}
S_+ &= T = (Tv_u^{-1}) v_u=S_- v_u \quad \text{on the upper lip} \\
S_+ &= Tv_l=S_-v_l \quad \text{on the lower lip} \\
S_+ &= T_+ v_u^{-1}=T_- v_T v_u^{-1} =T_- v_l v_0=S_- v_0 \quad \text{on }[-1,1],
\end{align*}
where we used (\ref{RHP_factor}) to derive the last relation.
Ignoring the jumps on the two lips and on $\hat{J} \setminus [-1,1]$, because there the corresponding jump matrices are close to $I$, we may hope that $S$ is close to the matrix $M$ that satisfies 
\begin{equation}
\label{RHP_M}
M_+=M_- v_0 \quad\text{on } [-1,1] \quad \text{ and } \quad M(z) \to I \quad\text{as } |z| \to \infty.
\end{equation}
This matrix is given by \cite[(7.66)-(7.72)]{Deift}
\begin{equation} 
\label{def_M}
M(z) := \frac{1}{2} 
\begin{pmatrix}
1 & 1 \\ i & -i
\end{pmatrix}
\begin{pmatrix}
a(z) & 0 \\ 0 & a^{-1}(z)
\end{pmatrix}
\begin{pmatrix}
1 & -i \\ 1 & i
\end{pmatrix}, \quad z \in \mathbb C\setminus [-1,1] ,
\end{equation}
where $a$ now denotes the analytic continuation $a(z) = (z-1)^{1/4}(z+1)^{-1/4}$ of the real function defined in (\ref{def_a}) to $\C\setminus [-1,1]$. 
As a consequence, $T$ should be close to $\mathbb{T}$ with $\mathbb{T}=Mv_u$ on the upper lens, $\mathbb{T}=Mv_l^{-1}$ on the lower lens, and $\mathbb{T}=M$ outside the lens shaped region of Figure \ref{fig:1}.
In order to prove that $\mathbb{T}$ is indeed a parametrix for $T$, i.e.~$\mathbb{T}$ is close to $T$ for large values of $N$, one introduces $R=T \, \mathbb{T}^{-1}$ and shows that the jump matrices corresponding to $R$ are sufficiently close to $I$ so that one can deduce $R=I+ \textit{small}$ everywhere  by functional analytic means. Then $T= (I+ \textit{small}) \, \mathbb{T}$ which turns out to be a useful formula for $T$ and via \eqref{def_T} also for $Y$. As we show now, the above described procedure works in principle, but requires a more detailed analysis near $\pm 1$ and near $\hat{L}_{\pm}$.

We begin by discussing the analytic extensions of some of the quantities of Definition \ref{def_notation}. One novel aspect of our paper is that we provide formulae and estimates that are uniform for functions $V$ from some open set. This is the reason, why we state the following auxiliary results in the setting of \GAii . The simpler case of a single function $V$, for which \GAi\ is assumed, is left for remarks thereafter.
\begin{lemma}\label{lemma_G_V}
Let $Q$ satisfy \GAii\  and denote $\hat{J}:= \lambda_Q^{-1} (J) \  (see \eqref{def_lambda})$. Then there exist $\sigma, d>0$ and an open neighborhood  $\mathcal{U}$ of $Q$ in $(X_D, \| \cdot \|_{\infty})$ (see Remark \ref{remark_complex_entension}) such that the analytic continuation of  $G_V$ \eqref{def_G} satisfies
\begin{align}\label{eq_new_G_V}
G_V(z) \in  \lbrace w \in \mathbb C \,|\, |w| \geq d, \textup{arg}(w) \leq \textstyle \frac{\pi}{8}\rbrace
\end{align}
for all $V \in \mathcal{U}$ and all 
$$ z \in \hat{J}_{\sigma}:= \{ z \in \mathbb C\, |\, \textup{dist}(z,\hat{J}) \leq \sigma\}.$$
\end{lemma}
\begin{proof}
Since $\hat{J}$ is compact and $Q'$ strictly increasing, it follows that $G_Q$ attains a positive minimum on $\hat{J}$, which we denote by $d_0$. 
By Lemma \ref{lemma_q}  we may choose some open neighborhood $\mathcal{U}_0$ of $Q$ and $\tau_0>0$ such that 
\begin{equation*}
\hat{J}_{2 \tau_0} \subset \lambda_V^{-1} (D) \quad \text{for all } V \in \mathcal{U}_0.
\end{equation*}
Thus definitions \eqref{def_h}, \eqref{def_G} of $h_V$ and $G_V$ can be extended to $\hat{J}_{2 \tau_0}$ simultaneously for all $ V \in \mathcal{U}_0$.
Denoting by $D'$ the convex hull of $\lambda_V(\hat{J}_{2 \tau_0}) \cup \lambda_Q(\hat{J}_{2 \tau_0})$, we have
\begin{align*}
\| (V \circ \lambda_V) - (Q \circ \lambda_Q) \|_{\hat{J}_{2 \tau_0}, \infty}
\leq \| V-Q \|_{D, \infty} + \|Q'  \|_{D',\infty} \cdot \|\lambda_V - \lambda_Q \|_{\hat{J}_{2 \tau_0}, \infty}.
\end{align*}
Using Lemma \ref{lemma_q} again, one may choose $\mathcal{U} \subset \mathcal{U}_0$
with 
\begin{equation*}
\| (V \circ \lambda_V)- (Q \circ \lambda_Q) \|_{\hat{J}_{2 \tau_0}, \infty}
< \frac{\tau_0^2 d_0}{8} \quad \text{ for all } V \in \mathcal{U}.
\end{equation*}
From the definition of $h_V$ and from the analyticity of $V \circ \lambda_V$ it follows that 
\begin{equation}\label{ineq_G}
\| G_Q - G_V \|_{\hat{J}_{\tau_0}, \infty} < \frac{d_0}{8}.
\end{equation}
For $ z \in \hat{J}_{\tau_0}$ there exists $x \in \hat{J}$ with $|x-z|\leq \tau_0$. Thus, 
\begin{equation*}
| G_V(z) - G_Q(x)| \leq \frac{d_0}{8} + \| G_Q' \|_{\hat{J}_{\tau_0}, \infty} \cdot \tau_0.
\end{equation*}
As $G_Q(x) \in [d_0, \infty)$ by the definition of $d_0$, the claim follows with 
$$\sigma:= \min\lb\tau_0, \,  \frac{d_0}{8} \|G_Q' \|^{-1}_{\hat{J}_{\tau_0}, \infty}\rb \quad \text{and} \quad d:= \frac{3}{4}d_0.$$
\end{proof}
\begin{remark}\label{remark_G_V}
In the case that some function $V$ satisfies \GAi\ with a bounded domain of definition $J$, the proof of Lemma \ref{lemma_G_V} implies that the claim of the lemma holds for this particular function $V$, i.e.~with $\mathcal{U}$ being replaced by $\{ V \}$. If $J$ is unbounded, the situation is different. Statement \eqref{eq_new_G_V} then holds for all $z \in \tilde{J}_{\sigma} = \{ z \in \mathbb C\, |\, \text{dist}(z, \tilde{J}) \leq \sigma\}$, where $\tilde{J}$ is an arbitrary but fixed bounded subset of $\hat{J}$. Observe that $d=d(\tilde{J})$ and $\sigma=\sigma(\tilde{J})$ may depend on the choice of $\tilde{J}$ if no additional assumptions on $V$ are being made. In order to be able to treat the cases of bounded and unbounded sets $J$ simultaneously, we define a standard compact subinterval $\tilde{J}_0$ of $\hat{J}$ by replacing any infinite endpoint $\pm \infty$ of $\hat{J}$ by $\pm 2$. Furthermore, set $\sigma := \sigma(\tilde{J}_0)$ and $d := d(\tilde{J}_0)$.
\end{remark}
We are now ready to define analytic extensions of $\xi_V$ and $\eta_V$.
\begin{definition}\label{Def_xi_eta_complex}
Denote by 
\begin{align}\label{def_q_tilde}
\tilde{q} (z) := (1-z)^{\frac{1}{2}}(1+z)^{\frac{1}{2}} 
\end{align}
the analytic extension of $t \mapsto \sqrt{1-t^2},\, t \in (-1,1)$, to $\mathbb C \setminus ((-\infty,-1] \cup [1,\infty))$ and recall \eqref{def_q} for the definition of the related function $q$. Set
\begin{align*}
\xi_V(z) &:= \int_z^1 \tilde{q}(t) G_V(t) \, \mathrm dt \quad \text{for }|\text{Re}(z)|<1, \\
\eta_V(z)&:= \begin{cases} 
\int_1^z q(t) G_V(t) \, \mathrm dt &, \text{ if Re}(z)>1,
\\ 
\int_{-1}^z q(t) G_V(t) \, \mathrm dt &,  \text{ if Re}(z)<-1.
\end{cases}
\end{align*}
Of course, the domains of $\xi_V$ and $\eta_V$ are further restricted by the domain of definition of $G_V$ as described in Lemma \ref{lemma_G_V} resp.~Remark \ref{remark_G_V} above. The reader should verify that $\xi_V$ and $\eta_V$ agree with the functions defined in \eqref{def_xi} resp.~\eqref{def_eta} on their common domains, justifying that the same symbols are used.
\end{definition}
\begin{corollary}\label{Koro_RHP}
Let $Q$ satisfy \GAii\ and let $\mathcal{U}$, $\sigma$, $d$, $\hat{J}=[ \hat{L}_-, \hat{L}_+]$ and $ \hat{J}_{\sigma}$ be given as in Lemma \ref{lemma_G_V}. Then for all $V \in \mathcal{U}$ the following holds:
\begin{enumerate}
\item[(a)] We have
\begin{align*}
\textup{Re}(\eta_V(z))& \geq \frac{2}{3}d |z-1|^{\frac{3}{2}} \quad \text{for }
z \in \hat{J}_{\sigma} \text{ with } |\textup{arg}(z-1)| \leq \frac{\pi}{16},
\\
\textup{Re}(\eta_V(z)) &\geq \frac{2}{3}d |z+1|^{\frac{3}{2}} \quad \text{for }
z \in \hat{J}_{\sigma} \text{ with } |\textup{arg}(-z-1)| \leq \frac{\pi}{16}.
\end{align*}
\item[(b)]For any compact $K \subset (0, \sigma]$ there exists $c(K)>0$ such that for all $\delta \in K:$
\begin{equation*}
\textup{Im}(\xi_V(z))
\begin{Bmatrix}
 \leq - c(K) |\textup{Im}(z)|  \\
\geq \phantom{-}c(K)|\textup{Im}(z)|  
\end{Bmatrix} \quad
\text{ for } \hspace{3pt} |\textup{Re}(z)| \leq 1-\delta \text{ and } 
\begin{Bmatrix}
 \phantom{-}0 \leq \textup{Im}(z) \leq \delta  \\
 - \delta \leq \textup{Im}(z) \leq 0
\end{Bmatrix}. 
\end{equation*}
\end{enumerate}
\end{corollary}
\begin{proof}
%\begin{enumerate}
%\item[(a)]
(a) \ We consider the case $\Re (z) >1$. Using the straight line parametrization $\gamma(t)=1+te^{i \, \text{arg}(z-1)}$, Lemma \ref{lemma_G_V} yields the following estimates on the integrand
$f(t):= q(\gamma(t)) \, G_V(\gamma(t))\,  \gamma'(t)$ in the definition of $\eta_V$
\begin{equation*}
|f(t)| \geq \sqrt{2t}\, d, \qquad  \qquad |\text{arg}(f(t))| \leq \frac{2 \pi}{16} + \frac{\pi}{8} = \frac{\pi}{4}
\end{equation*}
for all $0 \le t \le |z-1|$. From this the claim can be derived without effort. \vspace{10pt}

(b) \
%\item[(b)]
The proof is similar to (a). However, one should split the path of integration into two line segments, from $z$ to $x:= \Re(z)$ and from $x$ to 1. 
Since $\Im (\xi_V(x))=0$, one has
\begin{equation*}
\Im(\xi_V(z)) = \Re \left( \int_{\Im (z)}^0 f(t) \, \mathrm dt \right), \quad \text{with }
f(t)=\tilde{q}(x+it) G_V(x+it).
\end{equation*}
The claim now follows from Lemma   \ref{lemma_G_V} and from the estimates 
$|\text{arg } \tilde{q}(x+it)| \leq \frac{\pi}{8}$ and $|\tilde{q} (x+it)| \geq \sqrt{1-(1-\delta_K)^2}$  with $\delta_K:= \min(K)>0$. 
%\end{enumerate}
\end{proof}
\begin{remark}\label{remark_Koro_RHP}
If $V$ satisfies \GAi\ the following version of Corollary \ref{Koro_RHP} holds true and is used in the subsequent analysis. For statement (b) $\sigma$ is chosen as explained in Remark \ref{remark_G_V}. For statement (a) and unbounded sets $J$ we replace $\hat{J}$ by arbitrary bounded subsets $\tilde{J}$ of $\hat{J}$. In this situation $\sigma$ and $d$ depend on $\tilde{J}$ (cf.~Remark \ref{remark_G_V}).
\end{remark}
Our next step is to define the parametrix $\mathbb{T}$ explicitly. Recall from the discussion below \eqref{def_M} that different formulae are needed for different parts of $\C$. The relevant subdivision of the complex plane into regions $\text{I}$, $\text{II}_u$, $\text{II}_l$, $\text{III}^{\pm}$,  $\text{IV}^{\pm}$ is displayed in Figure \ref{fig:2}. Note that regions $\text{IV}^{\pm}$ only come into play if $\hat{L}_{\pm}$ is finite. 

\begin{figure}[h]
\begin{tikzpicture}
%% Punkte; Achtung: x-Koordinaten von B-A müssen gleich E-G sein
\coordinate
%[label=45:$V$]  % diese Zeile auskommentieren um Label V verschwinden zu lassen
(V) at (-10,0);
\coordinate
%[label=200:$A$] % diese Zeile auskommentieren um Label V verschwinden zu lassen
(A) at (-8, 0);
\coordinate
%[label=0:$B$]   % diese Zeile auskommentieren um Label  verschwinden zu lassen
(B) at (-6,0.8);
\coordinate
%[label=0:$C$]  % diese Zeile auskommentieren um Label  verschwinden zu lassen
 (C) at (-6,-0.8);
\coordinate
%[label=0:$E$]  % diese Zeile auskommentieren um Label  verschwinden zu lassen
 (E) at (-3,0.8);
\coordinate
%[label=0:$F$]  % diese Zeile auskommentieren um Label  verschwinden zu lassen
 (F) at (-3,-0.8);
\coordinate
%[label=0:$G$]  % diese Zeile auskommentieren um Label  verschwinden zu lassen
 (G) at (-1,0);
\coordinate
%[label=0:$H$]  % diese Zeile auskommentieren um Label  verschwinden zu lassen
 (H) at (1,0);
\coordinate
[label=270:{\tiny$\hat{L}_+$}, label=90: $\text{IV}^+$]   
(I) at (1.8,0);

\coordinate
[label=0:I]   
(ID1) at ($(G)!.4!(I) + (0,0.9)$);
\coordinate
[label=0:I]   
(ID2) at ($(G)!.4!(I) - (0,0.9)$);

%Mittelpunkte für Kreise berechnen; hier sollte Nichts verändert werden; Center1 bei -1; Center 2 bei 1
\coordinate (AB) at ($(A)!.5!(B)$);
\coordinate (AC) at ($(A)!.5!(C)$);
\coordinate (AAB) at ($(AB)!1!-90:(A)$);
\coordinate (AAC) at ($(AC)!2.8!90:(A)$);
\coordinate[label=270:{\tiny-$1$}, label=120: $\text{III}^-$ ]   (Center1) at (intersection of AB--AAB and AC--AAC);
 
\coordinate (EF) at ($(E)!.5!(F)$);
\coordinate (EG) at ($(E)!.5!(G)$);
\coordinate (EEF) at ($(EF)!1!-90:(E)$);
\coordinate (EEG) at ($(EG)!2.8!90:(E)$);
\coordinate[label=270:{\tiny $1$}, label=60: $\text{III}^+$] (Center2) at (intersection of EF--EEF and EG--EEG);

%Linien einzeichnen incl Pfeilspitzen auf halber Strecke; keine Änderungen nötig
 \draw[decoration={markings, mark=at position 0.5 with {\arrow{>}}},
       postaction={decorate}](V) -- (A)
			node[pos=.5,sloped,above] {$\Sigma_1$};
\draw[decoration={markings, mark=at position 0.5 with {\arrow{>}}},
       postaction={decorate}] (B)--(E)
			 node[pos=.8, above] { $\Sigma_2^u$};
\draw[decoration={markings, mark=at position 0.5 with {\arrow{>}}},
       postaction={decorate}] (C)--(F)
			node[pos=.8, below] { $\Sigma_2^l$};
\draw[decoration={markings, mark=at position 0.5 with {\arrow{>}}},
       postaction={decorate}] (G)--(H)
				node[pos=.5,sloped,above] {$\Sigma_1$};
\draw[dotted,decoration={markings, mark=at position 0.5 with {\arrow{>}}},
       postaction={decorate}] (A)--(G)
			node[pos=0.6] { $\Sigma_2^0$}
			node[pos=0.4, above] { $\text{II}_u$}
			node[pos=0.4, below] { $\text{II}_l$}
			;
\draw[dotted,decoration={markings, mark=at position 0.5 with {\arrow{>}}},
       postaction={decorate}] (H)--(I);

\coordinate (s) at ($(B)!.5!(Center1)$);
\coordinate (t) at ($(C)!.5!(Center1)$);
\coordinate (u) at ($(E)!.5!(Center2)$);
\coordinate (v) at ($(F)!.5!(Center2)$);

\draw[dotted]  (Center1) .. controls ($(s)+(-0.1,0.1)$) .. (B);
\draw[dotted] (Center1).. controls  ($(t)+(-0.1,-0.1)$) .. (C);
\draw[dotted] (E) .. controls ($(u)+(0.1,0.1)$) .. (Center2);
\draw[dotted] (Center2) .. controls ($(v)+(0.1,-0.1)$) ..  (F);

% Kreismittelpunkte einzeichnen
\draw[fill=black] (Center1) circle (1pt);
\draw[fill=black] (Center2) circle (1pt);
\draw[fill=black] (I) circle (1pt);

% Kreis um I
 \draw[->] let \p1 = ($(I)-(H)$) 
in
(H) arc (-180:90:{veclen(\x1,\y1)})
(H) arc (-180:-270:{veclen(\x1,\y1)})
node[above]{$\Sigma_4^+$};

% Kreis um Center1=-1 
\draw[->] let \p1 =  ($(Center1)-(A)$) 
in
(G) arc (0:90:{veclen(\x1,\y1)})
(G) arc (0:-270:{veclen(\x1,\y1)})
node[above]{$\Sigma_3^+$};

% Kreis um Center2=1
\draw[->] 
let \p1 = ($(Center2)-(G)$) in
(A) arc (-180:90:{veclen(\x1,\y1)})
(A) arc (-180:-270:{veclen(\x1,\y1)})
node[above]{$\Sigma_3^-$};

\end{tikzpicture}
\caption{Subdivision of $\C$ for the definition of the parametrix $\mathbb{T}$ in the case that $J$ is bounded above and unbounded below.}
\label{fig:2}
\end{figure}
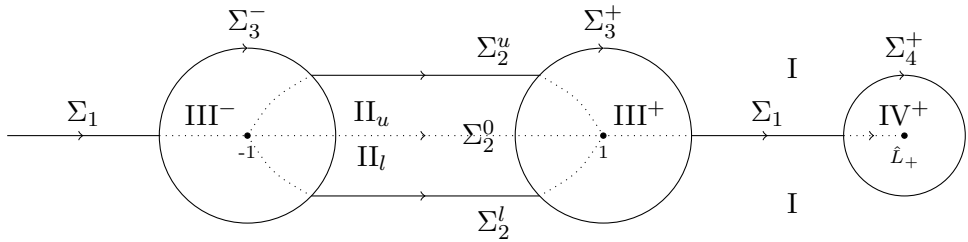
The construction of $\mathbb{T}$ depends on two parameters $\delta$ and $\varepsilon$ which denote the radii of the circles around $\pm 1$ resp.~around $\hat{L}_{\pm}$. The line segments $\Sigma_2^u$, $\Sigma_2^l $ are chosen to be parallel to the real axis and to begin at $-1+ \delta e^{\pi i/4}$ resp.~$-1+ \delta e^{-\pi i/4}$. Dotted lines are used for those parts of the boundaries where $T$ and $\mathbb{T}$ satisfy exactly the same jump relation implying that $R=T\, \mathbb{T}^{-1}$ as defined in \eqref{Def_R} has an analytic continuation across those lines. The definitions of the locations of the curved dotted lines within regions  $\text{III}^{\pm}$ and the definition of the parametrix $\mathbb{T}$ in these regions are somewhat involved. They are transferred to the Appendix (see \eqref{def_matbb_T}, \eqref{def_mathbb_T-}, and Figure \ref{fig:3}).

We begin the definition of $\mathbb{T}$ by fixing a common upper bound $\sigma_0$ for the radii $\delta$ and $\varepsilon$ of the circles that appear in the construction. Let $\sigma$ be given by Lemma \ref{lemma_G_V} resp.~as in Remark \ref{remark_G_V}
and let $0 < \hat{\sigma} < \sigma$ be defined through Lemma \ref{lemma_f_hat}. Set
\begin{equation}\label{def_sigma_0}
\sigma_0:= \min \lbrace \hat{\sigma}/10, \, (\hat{L}_+ -1)/6, \, (-1 - \hat{L}_- )/6 \rbrace.
\end{equation}
From now on we always assume $\delta$, $\varepsilon \in (0, \sigma_0]$. Moreover, we suppress the $V$-dependency in the notation, i.e.~$\xi \equiv \xi_V$, $\eta \equiv \eta_V$.
The definition of the parametrix $\mathbb{T}$ in the regions I, $\text{II}_u$, $\text{II}_l$
has already been motivated in the paragraph below \eqref{def_M}. Accordingly, we set
\begin{align}
\mathbb{T}(z) &:= M(z) \quad \text{ for } z \in \text{I} \label{def_T1} \, , \\
\mathbb{T}(z) &:= M(z) v_u(z) =M(z) \begin{pmatrix} 1 & 0 \\ e^{-iN \xi(z)} & 1 \end{pmatrix} \quad \text{ for  } z \in \text{II}_u \, , \label{def_T2}\\
\mathbb{T}(z) &:= M(z) v_l^{-1}(z) =M(z) \begin{pmatrix} 1 & 0 \\ e^{iN \xi(z)} & 1 \end{pmatrix}^{-1} \quad \text{ for  } z \in \text{II}_l \, .  \label{def_T3}
\end{align}
As mentioned above, the choice for the parametrix in $\text{III}^{\pm}$ is presented in the Appendix, see \eqref{def_matbb_T}, \eqref{def_mathbb_T-}.
For regions $\text{IV}^{\pm}$ we first define the function 
\begin{equation}
\label{def_b}
b(z):= \begin{cases}
\displaystyle \frac{1}{2 \pi i} \int_{\hat{L}_-}^{\hat{L}_- + 2 \sigma_0}  \frac{e^{-N \eta (t)}}{t-z} \, \mathrm dt &, \text{ if } z \in \text{IV}^-:= B_{\varepsilon}(\hat{L}_-)\setminus [\hat{L}_-, \hat{L}_- + \varepsilon)\, , \vspace{8pt}\\
\displaystyle \frac{1}{2 \pi i} \int_{\hat{L}_+ - 2 \sigma_0}^{\hat{L}_+}  \frac{e^{-N \eta (t)}}{t-z} \, \mathrm dt &, \text{ if } z \in \text{IV}^+:= B_{\varepsilon}(\hat{L}_+)\setminus (\hat{L}_+-\varepsilon, \hat{L}_+ ]\, .
\end{cases}
\end{equation}
It is a basic property of the Cauchy transform (cf.~\eqref{cauchyproperty}) that
\begin{equation}\label{prop_B}
b_+ -b_- =e^{-N \eta}
\end{equation}
on $(\hat{L}_-,\hat{L}_- + \varepsilon) \cup (\hat{L}_+ - \varepsilon,\hat{L}_+ )$.
Set 
\begin{equation}
\label{def_T4}
\mathbb{T}(z):= M(z) 
\begin{pmatrix}
1 & b(z) \\ 0 & 1
\end{pmatrix} \quad \text{for } z \in \text{IV}^+ \cup \text{IV}^-.
\end{equation}
As argued above, $T$ and $\mathbb{T}$ are compared by deriving a RHP for 
\begin{equation}
\label{Def_R}
R:= T \, \mathbb{T}^{-1}.
\end{equation}
The essential information is provided by the following 
\begin{lemma}\label{Lemma_R}
%\begin{enumerate}
%\item[(a)]
(a) \
Let $V$ satisfy  \GAi\ and let $R$ be defined as in \eqref{Def_R}, see also \eqref{def_sigma_0} to \eqref{def_T4}, \eqref{def_matbb_T}, and \eqref{def_mathbb_T-}. Then $R$ has an analytic extension to $\mathbb C \setminus \Sigma_R$, where
$$ \Sigma_R := \Sigma_1 \cup \Sigma_2^u \cup \Sigma_2^l \cup \Sigma_3^+ \cup \Sigma_3^- \cup \Sigma_4^+ \cup \Sigma_4^- \, , 
\quad \text{see Figure \ref{fig:2}}.$$
Moreover,
\begin{align} 
&R_+ = R_- v_R, \quad \text{with} \quad v_R=I + \Delta_R \text{ and }
\tag*{(i)$_R$} \label{error_R:1} \\ 
&\| \Delta_R\|_{L^1 (\Sigma_R)}+\| \Delta_R\|_{L^\infty (\Sigma_R)} = \mathcal{O} \left( N^{-1}\right) \label{error_R:2}, \\
&\lim_{|z| \to \infty} R(z)=I. \label{error_R:3}
\tag*{(ii)$_R$} 
\end{align} 
The error bound $\mathcal{O}(N^{-1})$ in \eqref{error_R:2} is uniform for $(\delta, \e)$ from an arbitrary but fixed compact subset of $(0,\sigma_0]^2$.
%\item[(b)]

(b) \
Assume that $Q$ satisfies \GAii . Then there exists an open neighborhood $\mathcal{U}$ of $Q$ in $(X_D, \| \cdot \|_{\infty})$ (see Remark \ref{remark_complex_entension}) such that statement (a) holds for all $V \in \mathcal{U}$. The error bound in \eqref{error_R:2} is in addition uniform in $V \in \mathcal{U}$.
% \end{enumerate}
 \end{lemma}
 \begin{proof}
 The claim that $R$ has an analytic continuation within III$^+ \, \cup$ III$^-$, as well as the estimate 
 $$\|\Delta_R \|_{L^{\infty}(\Sigma_3^+ \cup \, \Sigma_3^- )} =\mathcal{O}\left( N^{-1}\right),$$
 together with the stated uniformity of the error bound are proved in Lemma \ref{lemma_R_2}.
 The following representations for 
 \begin{align}\label{v_R}
 v_R=R_-^{-1} R_+ = \mathbb{T}_- T_-^{-1} T_+ \mathbb{T}_+^{-1}=\mathbb{T}_-v_T\mathbb{T}_+^{-1} 
 \end{align}
 are straightforward from \eqref{def_v_T} and from the explicit definition of $\mathbb{T}$ above.
 Observe that $v_T=I$ on those parts of $\Sigma_R$, where $T$ has no jumps. We obtain
 \begin{align}
\text{on }  \Sigma_1:\quad v_R&=M v_TM^{-1} =I+ M \begin{pmatrix} 0 & e^{-N\eta} \\0 &0 \end{pmatrix}M^{-1} \label{v_R:1} \, ,\\
 \text{on }    \Sigma_2^u:\quad v_R& =M v_uM^{-1} =I+ M \begin{pmatrix} 0 &0 \\ e^{-iN\xi} &0 \end{pmatrix}M^{-1} \label{v_R:2} \, ,
 \\
 \text{on }     \Sigma_2^0:\quad v_R&= M_-v_l^{-1} v_T v_u^{-1} M_+^{-1}=M_- v_0 M_+^{-1}=I   \text{ by }\eqref{RHP_factor}, \eqref{RHP_M}  \label{v_R:3} \, , \\
  \text{on }    \Sigma_2^l:\quad v_R&=M v_l M^{-1} =I +  M \begin{pmatrix} 0 &0 \\ e^{iN\xi} &0 \end{pmatrix}M^{-1} \label{v_R:4} \, , \\
   \text{on }   \Sigma_4^+ \cup  \Sigma_4^-: \quad v_R&= M  \begin{pmatrix} 1 &b \\ 0 &1 \end{pmatrix} M^{-1}=I+M 
			\begin{pmatrix} 0 &b \\ 0 &0 \end{pmatrix}
			M^{-1}   \label{v_R:5} \, ,\\ 
\phantom{XXX}  \text{on }  (\hat{L}_-,\hat{L}_- + \e ) \, \cup \; &   (\hat{L}_+ - \e,\hat{L}_+ ):\nonumber \\
          v_R&= M  \begin{pmatrix} 1 &b_- \\ 0 &1 \end{pmatrix}v_T  \begin{pmatrix} 1 &b_+ \\ 0 &1 \end{pmatrix}^{-1} M^{-1} \nonumber\\
          &=M  \begin{pmatrix} 1 & e^{- N \eta} - (b_+ - b_-) \\ 0 &1 \end{pmatrix} M^{-1}=I \quad \text{by } \eqref{prop_B}. \label{v_R:6}
 \end{align} 
 The analytic extendibility of $R$ across $\Sigma_2^0$ follows from \eqref{v_R:3}. Relation  \eqref{v_R:6} implies that $R$ has an analytic continuation on $B_{\e}(\hat{L}_{\pm}) \setminus \{\hat{L}_{\pm} \}$.
 Since $T$ and $\mathbb{T}$ are both bounded in the first column and are bounded by $\mathcal{O}(|\log|z-\hat{L}_{\pm}||)$ in the second column, we have $R(z)=\mathcal{O}(|\log|z-\hat{L}_{\pm}||)$ for $z \to \hat{L}_{\pm}$.
 The singularities at $\hat{L}_{\pm}$ must therefore be removable by the Riemann Continuation  Theorem. 
 
The $L^\infty$-part of estimate \eqref{error_R:2} is a consequence of the boundedness of $M$ on $\mathbb C \setminus (B_{\delta}(1)\cup B_{\delta}(-1) )$ (with a bound that depends on $\delta$), Corollary \ref{Koro_RHP}, and Remark \ref{remark_Koro_RHP}. 
For $s \in \Sigma_4^{\pm}$ close to $\hat{L}_{\pm} \mp \e$ one may resolve the difficulty due to the singularity of the Cauchy kernel by deforming the path of integration (in the definition of the function $b$) away from the real axis so that its distance to $s$ is at least, say $\frac{\e}{2}$. Here we use that the estimates of Corollary \ref{Koro_RHP} (a) also apply away from the real axis. Observe that it follows from the definition of $\sigma_0$ that the deformed contour lies in the range of applicability of Corollary \ref{Koro_RHP} (a).

The $L^1$-part of estimate \eqref{error_R:2} is implied by the $L^\infty$-estimate, except in the case that $\Sigma_1$ contains an unbounded component. This only occurs in the situation of \GAi\ with unbounded $J$. Suppose e.g.~that $L_+ = \infty$. It follows from the strict monotonicity of $W'$ that for $x \ge 2$ and $t \in [-1, 1]$ one has (see \eqref{def_h}, \eqref{def_G}, and \eqref{def_eta})
$$
h(t,x) \ge \frac{2}{3x}(W'(2)- W'(1)) \quad \text{and}  \quad \eta(x) \ge \eta(2) + \frac{W'(2)- W'(1)}{2} (x-2) \, ,
$$
yielding the desired $L^1$-bound.

 Claim \ref{error_R:3}  follows from $\text{(ii)}_T$ below \eqref{def_T},  from $\lim_{|z| \to \infty} M(z)=I$, and from the definition of $R$ in \eqref{Def_R}.
 \end{proof}
 We are now ready to formulate the main result of this section. As demonstrated in the subsequent section, it provides estimates on $R-I$ that suffice to prove all the theorems stated in the Introduction. The theorem follows from Lemma \ref{Lemma_R} and from some functional analytic arguments that we quote from \cite[Section 7.5]{Deift}.
 \begin{theorem}\label{theorem_R_+}
%\begin{enumerate}
%\item[(a)]
(a) \
Let $V$ satisfy \GAi\ and let $R$ be defined as in \eqref{Def_R}, see also \eqref{def_sigma_0}--\eqref{def_T4}, \eqref{def_matbb_T}, and \eqref{def_mathbb_T-}. Then for all $x$, $y  \in \hat{J},$
\begin{align}
R_+(x) -I& = \mathcal{O}\left(  N^{-1}\right), \quad R_+'(x)= \mathcal{O}\left(  N^{-1}\right), \text{ and }\label{R_plus:1}\\
&R_+(y)^{-1} R_+(x) =I+ \mathcal{O}\left(  \frac{|x-y|}{N}\right).\label{R_plus:2}
\end{align}
The error bounds are uniform for $(\delta,\e)$ in compact subsets of $(0,\sigma_0 / 2]^2$ and for $x$, $y$ in bounded subsets of $\hat{J}$.
%\item[(b)]

(b) \
Assume that $Q$ satisfies \GAii . Then there exists an open neighborhood $\mathcal{U}$ of $Q$ in $(X_D, \| \cdot \|_{\infty})$ (see Remark \ref{remark_complex_entension}) such that statement (a) holds for all $V \in \mathcal{U}$. The error bounds in \eqref{R_plus:1},  \eqref{R_plus:2} are in addition uniform in $V \in \mathcal{U}$.
%\end{enumerate}
\end{theorem}
\begin{proof}
The RHP for $R$ that is stated in Lemma \ref{Lemma_R} is equivalent to a singular integral equation (see Section 7.5 in \cite{Deift} for a detailed description). We apply \cite[Theorem 7.103]{Deift}  with $m \equiv R$, $b_- \equiv I$, $b_+ \equiv v_R$, $w_-\equiv 0$, $w=w_+\equiv \Delta_R$ (see Lemma  \ref{Lemma_R}).
The Cauchy operator $C$ has to be defined on $\Sigma_R$, i.e. 
$$ (Cf)(z)= \frac{1}{2 \pi i} \int_{\Sigma_R} \frac{f(t)}{t-z}\, \mathrm dt.$$
Theorem 7.103 in \cite{Deift} relates the RHP for $m$ to  the integral operator $C_w$ that reads 
\begin{equation*}
C_w: L^2 (\Sigma_R; \C^{2 \times 2}) \to  L^2 (\Sigma_R; \C^{2 \times 2}), \quad f \mapsto C_-(f\Delta_R)
\end{equation*}
in our situation. Since $C_-$ is a bounded operator on $ L^2 (\Sigma_R) \equiv L^2 (\Sigma_R; \C^{2 \times 2})$ and as 
$\|\Delta_R \|_{L^{\infty}(\Sigma_R)}=\mathcal{O}(N^{-1})$ by Lemma \ref{Lemma_R}, we have 
\begin{equation}
\label{C_w}
\|C_w\|_{\text{Op}} = \mathcal{O}\left( N^{-1} \right),
\end{equation}
where $\|\cdot\|_{\text{Op}}$ denotes the operator norm. Albeit the constant function $z \mapsto I$ might not lie in $L^2(\Sigma_R)$, $C_w(I)=C_-(\Delta_R)$ exists, because $\Delta_R \in L^2(\Sigma_R)$. 
Moreover, 
\begin{equation}
\label{Delta_R}
\| \Delta_R\|_{L^2} \leq \left( \| \Delta_R\|_{L^1} \| \Delta_R\|_{L^\infty} \right)^{1/2}
\end{equation}
together with \eqref{error_R:2} implies 
\begin{equation}
\label{Delta_R:2}
\|C_-(\Delta_R)\|_{L^2(\Sigma_R)} = \mathcal{O} \left( N^{-1} \right).
\end{equation}
The relevant singular integral equation  \cite[(7.104)]{Deift} then reads 
$(\dopp 1 - C_w)\mu =I$, where $\dopp 1$ denotes the identity on $L^2(\Sigma_R)$. In fact, $\mu$ is of the form $\mu=I+ \tilde{\mu}$, $\tilde{\mu} \in L^2(\Sigma_R)$, and the proper equation for $\tilde{\mu}$ is given by 
\begin{equation}
\label{Cw}
(\dopp 1 - C_w) \tilde{\mu}=C_w(I)=C_-(\Delta_R), \qquad \text{see \cite[(7.105)]{Deift}.}
\end{equation}
From \eqref{C_w} it follows that $\|C_w\|_{\text{Op}} \leq \frac{1}{2}$ for $N$ sufficiently large which implies the invertibility of $\dopp 1-C_w$.
Using also \eqref{Delta_R:2}, we conclude that  \eqref{Cw} has a unique solution $\tilde{\mu} \in L^2(\Sigma_R)$ with 
\begin{equation}
\label{mu_tilde}
\|\tilde{\mu}\|_{L^2(\Sigma_R)}=   \mathcal{O} \left( N^{-1} \right).
\end{equation}
Theorem 7.103 in \cite{Deift} provides a formula for $R$ in terms of the solution $\tilde{\mu}$ of integral equation \eqref{Cw}:
$$ R=I+C(\Delta_R + \tilde{\mu} \Delta_R).$$
Observe that \eqref{error_R:2}, \eqref{mu_tilde}, and  \eqref{Delta_R} yield
\begin{equation*}
\|\Delta_R + \tilde{\mu} \Delta_R \|_{L^1(\Sigma_R)}=   \mathcal{O} \left( N^{-1} \right).
\end{equation*}
By the definition of the Cauchy operator $C$ we have 
\begin{align*}
|R(z)-I|&\leq \frac{1}{2\pi} \frac{\|\Delta_R + \tilde{\mu} \Delta_R\|_{L^1(\Sigma_R)}}{\text{dist}(z, \Sigma_R)},
\\
|R'(z)|&\leq \frac{1}{2\pi} \frac{\|\Delta_R + \tilde{\mu} \Delta_R\|_{L^1(\Sigma_R)}}{\text{dist}(z, \Sigma_R)^2}
\end{align*}
and claim \eqref{R_plus:1} follows for $x \in \hat{J}$ that have at least distance, say $\frac{\delta}{10}$, from $\Sigma_R$.

We still need to consider those $x\in\hat{J}$ with $\text{dist}( x,\Sigma _R) <\frac{\delta}{10}$. For example, let $x=1-\frac{101}{100} \delta$ with distance $\frac{\delta}{100}$ from $\Sigma_3^+$. In this case we change the parameter $\delta$ in the construction, i.e.~the radius of the circles around $\pm 1$, to $\tilde{\delta}:=\frac{9}{10}\delta$. Note that by definition the corresponding parametrices $\mathbb{T}$ and $\tilde{\mathbb{T}}$ agree at the given point $x$, thus $\tilde{R}\lb x\rb:=T\lb x\rb \tilde{\mathbb{T}}^{-1}\lb x\rb = T\lb x\rb \mathbb{T}^{-1}\lb x\rb = R\lb x\rb$. 
We may apply Lemma \ref{Lemma_R} to $\tilde{R}$ to derive 
$$\Vert \Delta_{\tilde{R}} + \tilde{\mu}_{\tilde{R}} \Delta _{\tilde{R}} \Vert _{L^1 \lb \Sigma _{\tilde{R}}\rb } =\mathcal{O}\lb N^{-1}\rb $$ with the desired uniformity properties of the error bound. As $\text{dist}( x,\Sigma_{\tilde{R}})  \geq \frac{\delta}{10} $, we obtain the desired estimates on $|\tilde{R} -I |$ and on $|\tilde{R}'\lb x\rb |$. Since $R$ agrees with $\tilde{R}$ on a small neighborhood of $x$, we have $\tilde{R}'\lb x\rb = R'\lb x\rb$ and (\ref{R_plus:1}) is established for $x=1-\frac{101}{100} \delta$.
More generally, such a procedure of shrinking or enlarging the disks around $\pm 1$, $\hat{L}_{\pm}$ allows to prove (\ref{R_plus:1}) for all $x\in \hat{J}\backslash \Sigma_1$ (see Figure \ref{fig:2}) by ensuring that $\text{dist}( x,\Sigma _{\tilde{R}}) \geq \min\lb \varepsilon ,\delta\rb /10$. Observe that the restriction of $\delta$, $\e$ to the size $\sigma_0 /2$ in the statement of the theorem allows to enlarge circles without leaving the domain that is covered by Lemma \ref{Lemma_R}.  

In the case $x\in\Sigma_1$ we may assume that again the distance between $x$ and any of the circles is at least $\min\lb \epsilon ,\delta\rb /10 =: \kappa_0$ by shrinking the circles if necessary. Recall that $v_R$ is of the form (\ref{v_R:1}) on $\left[ x-\kappa_0 ,x+\kappa_0 \right]$. It is claimed in the theorem that the error bounds are uniform for $x$ in bounded sets $B \subset J$. This is relevant only in the case of unbounded sets $J$ in the situation of \GAi , otherwise we choose $B := J$. Denote further $\tilde{J} := \lambda_V^{-1}(B)$, i.e.~$\tilde{J}= \hat{J}$ for bounded sets $J$. According to Remark \ref{remark_Koro_RHP} the estimates of Corollary \ref{Koro_RHP} (a) will hold in all cases considered, if $z \in \hat{J}_{\sigma}$ is replaced by $z \in \tilde{J}_{\tilde{\sigma}}$ with $\tilde{\sigma} := \sigma (\tilde{J})$ (cf.~Remark \ref{remark_G_V}). Set $\kappa := \min \lb \kappa_0, \tilde{\sigma} \rb /2$. Thus, $v_R$ possesses an analytic continuation on a neighborhood of $\overline{B_{\kappa}\lb x\rb}$. We may deform the contour of jumps away from $x$ by defining
\begin{align*}
\tilde{R}\lb z\rb := \begin{cases}
R\lb z\rb v_R\lb z\rb  &, \text{ if } |z-x|<\kappa \text{ and } \Im (z) <0 ,\\
R\lb z\rb &, \mbox{ else}.
\end{cases}
\end{align*}
Then $\tilde{R}$ has an analytic continuation across $\lb x-\kappa ,x+\kappa\rb$ and the jump relation $\tilde{R}_+=\tilde{R}_- v_R$ is moved to the lower half-circle $\partial B_{\kappa}\lb x\rb \cap \left\lbrace z\in\mathbb{C}| \Im (z) <0\right\rbrace$ (cf.~\cite[Figure 7.8]{DKMVZ1}).
Observe further that any $z\in\overline{B_{\kappa}\lb x\rb}$ satisfies $|\arg \lb z-1\rb |<\frac{\pi}{16}$ if $x>0$, resp.~$|\arg \lb -1-z\rb | <\frac{\pi}{16}$ if $x<0$ by construction. Using the lower bounds for $\Re\lb \eta_V\lb z\rb\rb$ provided by statement (a) of Corollary \ref{Koro_RHP} in combination with Remark \ref{remark_Koro_RHP} concludes the proof of (\ref{R_plus:1}). Relation (\ref{R_plus:2}) follows from (\ref{R_plus:1}) and 
\begin{align*}
R_+\lb y\rb ^{-1} R_+\lb x\rb = I+R_+\lb y\rb ^{-1}\lb R_+\lb x\rb -R_+\lb y\rb\rb.
\end{align*}
\end{proof}

\setcounter{equation}{0}
\section{Proofs of Main Results}\label{sec4}

We begin with the proof of our basic Theorem \ref{theorem_kernel} about the Christoffel-Darboux kernel from which the other main results of this paper, Theorems \ref{theorem_asymptotics}, \ref{theorem_bulk}, and \ref{theorem_edge}, will be derived thereafter. 

We follow and further streamline the formalism introduced in \cite{Vanlessen}. In most of the section we again suppress the $V$-dependency of various quantities for the sake of readability.

\begin{proof}[Proof of Theorem \ref{theorem_kernel}:]

Recall relations \eqref{K_N}, \eqref{K_N:2} from the previous section, which imply
\begin{align}\label{K_N_proof}
\frac{b_V-a_V}{2}K_{N,V}\lb\lambda_V\lb x\rb,\lambda_V\lb y\rb\rb=
\frac{e^{-\frac{N}{2}\lb W\lb x\rb + W\lb y\rb\rb}}{2\pi i\lb x-y\rb}\begin{pmatrix}
0&1
\end{pmatrix}Y_+^{-1}\lb y\rb Y_+\lb x\rb \begin{pmatrix}
1\\ 0
\end{pmatrix}
\end{align}
for $x,y\in\hat{J}$  with $x \ne y$. Moreover, \eqref{def_T} and \eqref{Def_R} yield a formula for the solution $Y$ of the corresponding 
Riemann-Hilbert problem, which reads
\begin{align*}
Y_+=e^{N\frac{l}{2}\sigma_3}R_+\mathbb{T}_+ e^{N\lb g_+ -\frac{l}{2}\rb\sigma_3}\quad\text{on }\hat{J} ,
\end{align*}
where the parameters in the construction of $\mathbb{T}$ are chosen to be $\delta$ as given by the statement of Theorem \ref{theorem_kernel} with $\delta_0 := \sigma_0/2$ and $\e := \sigma_0/2$ (see \eqref{def_sigma_0}). We find it convenient to rewrite
\begin{align}
Y_+ &=e^{\frac{N}{2}W}e^{N\frac{l}{2}\sigma_3}R_+ A F_0 \quad \text{ with}\label{eq_Y_+}\\
A:&= \frac{1}{\sqrt{2}}e^{\frac{\pi i}{4}}\begin{pmatrix}
1&1 \\ i& -i
\end{pmatrix} \quad\text{ and }\quad F_0:= e^{-\frac{N}{2}W}A^{-1}\mathbb{T}_+ e^{N\lb g_+-\frac{l}{2}\rb \sigma_3} .\label{def_A_F}
\end{align}
Recall further that $\det Y=1$ everywhere (see Theorem \ref{theorem_rhp}) and that $\begin{pmatrix}
0 & 1
\end{pmatrix}Y_+^{-1}$ and formula \eqref{K_N_proof} for the Christoffel-Darboux kernel only depend on the entries of the first column of $Y_+$. This observation allows us to replace $F_0$ in \eqref{eq_Y_+} by any $F$ that has the same first column as $F_0$. Defining
\begin{align}\label{def_F}
F\lb x\rb :=\begin{cases}
A^{-1}\mathbb{T}_+\lb x\rb e^{\frac{N}{2}\lb 2 g_+\lb x\rb -l-W\lb x\rb\rb\sigma_3} &, \text{ if } x\in\hat{J}\setminus \lb \text{IV}^+\cup \text{IV}^- \rb , \\
A^{-1}\mathbb{T}_+\lb x\rb\begin{pmatrix}
1 & -b\lb x\rb \\ 0 & 1
\end{pmatrix}e^{\frac{N}{2}\lb 2 g_+\lb x\rb - l-W\lb x\rb\rb \sigma_3} &, \text{ if }  x\in \hat{J} \cap \lb \text{IV}^+\cup \text{IV}^-\rb ,
\end{cases}
\end{align}
(see Figure \ref{fig:2} and \eqref{def_b}) one obtains from \eqref{eq_Y_+}
\begin{align}\label{eq_Y_+_2}
Y_+\begin{pmatrix}
1\\0
\end{pmatrix}=e^{\frac{N}{2}W}e^{N\frac{l}{2}\sigma_3}R_+ AF\begin{pmatrix}
1\\0
\end{pmatrix} .
\end{align}
For the next computation it is essential that the determinants of $A$, $R_+$, $Y_+$, and $F$ are all equal to 1 on $\hat{J}$. For $A$ and $Y_+$ this follows directly from \eqref{def_A_F} resp.~from  Theorem \ref{theorem_rhp}. In the case of $F$ and $R_+$ one concludes from \eqref{def_F},  \eqref{Def_R}, \eqref{def_T}, and Theorem \ref{theorem_rhp} that it suffices to show that $\det \mathbb{T} = 1$. Except for the regions III$^{\pm}$ this is immediate from the definitions \eqref{def_M}, \eqref{def_T1}  to \eqref{def_T3},  and \eqref{def_T4}. In the disks III$^{\pm}$ relations \eqref{def_E_N}, \eqref{def_matbb_T}, \eqref{def_E_N-}, \eqref{def_mathbb_T-}, and \eqref{def_psi_beta} imply that one needs to verify that $\det \psi_0 = (2\pi)^{-1} e^{\pi i /6}$. To this end observe first that by definition \eqref{def_psi_0} $\psi_0$ is given as Wronskians for the linear second order differential equation $w''=zw$ (see \cite[10.4.1]{AS}) and its determinant is therefore constant on the upper and lower half-plane. The desired result then follows from the asymptotic formula \eqref{psi_0_asmp}.

Using the relation
\begin{align*}
X^{-1}=\frac{1}{\det X}\begin{pmatrix}
0 & -1 \\ 1 & 0
\end{pmatrix}X^T\begin{pmatrix}
0 & 1 \\ -1 & 0
\end{pmatrix}
\end{align*}
for invertible $2\times 2$ matrices $X$ twice, one obtains from \eqref{eq_Y_+_2}
\begin{align}
\begin{pmatrix}
0 & 1
\end{pmatrix} Y_+^{-1}&= \lb Y_+ \begin{pmatrix}
1 \\ 0
\end{pmatrix}\rb ^T \begin{pmatrix}
0 & 1 \\ -1 & 0
\end{pmatrix}\notag \\
&=e^{\frac{N}{2}W}\lb F\begin{pmatrix}
1 \\ 0
\end{pmatrix}\rb ^T \lb e^{N\frac{l}{2}\sigma_3}R_+ A \rb ^T \begin{pmatrix}
0 & 1 \\ -1 & 0
\end{pmatrix}\label{eq_Y_+_1} \\
&= e^{\frac{N}{2}W}\lb F \begin{pmatrix}
1\\0
\end{pmatrix}\rb ^T\begin{pmatrix}
0 & 1 \\ -1 & 0
\end{pmatrix}A^{-1}R_+^{-1}e^{-N\frac{l}{2}\sigma_3} . \notag 
\end{align}
Finally, we set
\begin{align}\label{def_k}
k :=\frac{1}{\sqrt{2 \pi}}e^{\frac{\pi i}{4}}F\begin{pmatrix}
1\\0
\end{pmatrix} \quad \text{on }\hat{J}.
\end{align}
Then  \eqref{K_N_proof}, \eqref{eq_Y_+_1},  \eqref{eq_Y_+_2}, and Theorem \ref{theorem_R_+} imply
\begin{align*}
& \frac{b_V-a_V}{2}K_{N,V}\lb \lambda_V\lb x\rb, \lambda_V \lb y \rb\rb \\
&=
\frac{1}{x-y}k\lb y\rb ^T \begin{pmatrix}
0 & -1 \\ 1 & 0
\end{pmatrix}A^{-1} R_+\lb y\rb ^{-1}R_+\lb x\rb A k\lb x\rb\\
&= \frac{k_1\lb x\rb k_2\lb y\rb -k_2\lb x\rb k_1 \lb y\rb}{x-y}+k\lb y\rb ^T \mathcal{O}\lb N^{-1}\rb k\lb x\rb.
\end{align*}
Moreover, the bounds on the $2\times 2$ matrix that are denoted by $\mathcal{O}\lb N^{-1}\rb$ in the above formula have all the uniformity properties stated in Theorem \ref{theorem_kernel} since they agree with those formulated in Theorem \ref{theorem_R_+}. 
Observe further,  that $R_+$ is differentiable by Theorem \ref{theorem_R_+}. This allows to extend the last relation to the diagonal $x=y$.

In order to complete the proof of Theorem \ref{theorem_kernel} one needs to verify that $k$ as defined by \eqref{def_k} and \eqref{def_F} is the same as the definition of $k$ in the statement of Theorem \ref{theorem_kernel}. This can be done by straightforward computations, using definitions \eqref{def_A_F}, \eqref{def_M}, \eqref{def_T1} to \eqref{def_b}, \eqref{def_T4}, \eqref{def_varphi}, \eqref{def_psi_0}, \eqref{def_psi_beta}, \eqref{def_E_N}, \eqref{def_f_N}, \eqref{def_matbb_T}, \eqref{def_f_N-}, \eqref{def_E_N-}, \eqref{def_mathbb_T-} together with \eqref{eq_varphi_eta_xi} that relates $\varphi$ to $\xi$ and $\eta$. Note also, that Corollary \ref{cor_g} implies 
\begin{align*}
e^{\frac{N}{2}\lb 2g_+\lb x\rb -l-W\lb x\rb\rb\sigma_3}=\begin{cases}
e^{i\frac{N}{2}\xi\lb x\rb\sigma_3} & \text{for } x\in\left[-1,1\right] ,\\
\sgn \lb x\rb ^N e^{-\frac{N}{2}\eta \lb x\rb\sigma_3} & \text{for }x\in\hat{J}\setminus\left[-1,1\right] .
\end{cases}
\end{align*}\end{proof}

The following proposition provides useful formulae for the leading order term of the Christoffel-Darboux kernel. Its proof is entirely elementary using relations such as $\frac{a}{b}-\frac{b}{a}=(\frac{1}{a}+\frac{1}{b})( a-b)$ and $2\cos (\alpha+\frac{\pi}{4})\cos(\beta -\frac{\pi}{4})=\cos (\alpha +\beta)+\sin( \beta -\alpha )$.

\begin{proposition}\label{proposition_kernel}
Let $k$ be defined as in Theorem \ref{theorem_kernel}. Then
\begin{align*}
\frac{k_1(x)k_2(y)-k_2(x)k_1(y)}{x-y}
\end{align*}
equals
\begin{align*}
\mbox{(a) }\, &\frac{\sgn (xy)^N}{4\pi}e^{-\frac{N}{2}(\eta (x)+\eta (y))}\lb \frac{1}{a(x)} + \frac{1}{a(y)}\rb \frac{a(x)-a(y)}{x-y} \\
&\mbox{ for } x,y\in\hat{J}\setminus [-1-\delta,1+\delta] ,\\
\mbox{(b) }\, &\frac{1}{2\pi}\lb\frac{\hat{a}(x)}{\hat{a}(y)}+\frac{\hat{a}(y)}{\hat{a}(x)}\rb \frac{\sin\lb N\pi\int_y^x\rho(s)\, \mathrm ds\rb}{x-y}\\
& + \frac{1}{2\pi} \cos\lb\frac{N}{2}\lb \xi (x)+\xi (y)\rb\rb\lb\frac{1}{\hat{a}(x)}+\frac{1}{\hat{a}(y)}\rb \frac{\hat{a}(x)-\hat{a}(y)}{x-y}\\
&\mbox{ for }\, x,y\in(-1+\delta,1-\delta) ,\\
\mbox{(c) }\, &\Aik \lb f_N(x),f_N(y)\rb \frac{f_N(x)-f_N(y)}{x-y}\\
& +\lb \frac{\Ai (f_N(x))\Ai'(f_N(y))}{d(y)}+\frac{\Ai(f_N(y))\Ai'(f_N(x))}{d(x)}\rb
\frac{d(x)-d(y)}{x-y}\\
&\mbox{ for }\, x,y\in[1-\delta,1+\delta] ,\\
\mbox{(d) }\, & (-1) \mbox{ times the formula in (c)}\\
&\mbox{ for } \, x,y\in [-1-\delta,-1+\delta] .
\end{align*}
\end{proposition}

In the remaining part of the present section we evaluate the formulae for the Christoffel-Darboux kernel given by Theorem \ref{theorem_kernel} and by Propostion \ref{proposition_kernel} in special cases.
\begin{proof}[Proof of Theorem \ref{theorem_bulk}:]

For $x\in[-1+\delta,1-\delta]$ and real $s,t$ denote
\begin{align*}
u:=x+\frac{s}{N\rho (x)}, \qquad v:=x+\frac{t}{N\rho(x)}.
\end{align*}
Choosing $d$ as defined in Lemma \ref{lemma_G_V} resp.~as at the end of Remark \ref{remark_G_V} one learns from \eqref{def_rho} that  $\rho(x) \ge \frac{d}{2\pi} \sqrt{1-(1-\delta)^2} $. Thus $u$, $v \in (-1 + \frac{\delta}{2}, 1 - \frac{\delta}{2})$ for $|s|$, $|t| < c_{\delta} N$ with $c_{\delta}:= \frac{\delta d}{4\pi} \sqrt{1-(1-\delta)^2} $.
We apply Theorem \ref{theorem_kernel} with $\delta$ replaced by $\frac{\delta}{2}$ and Proposition \ref{proposition_kernel} with $(x, y)$ replaced by $(u, v)$. Since the second summand in Proposition \ref{proposition_kernel} (b) is uniformly bounded for $u$, $v\in (-1+\frac{\delta}{2},1-\frac{\delta}{2})$, one may conclude from
\begin{align*}
	\frac{\hat{a}(u)}{\hat{a}(v)}+\frac{\hat{a}(v)}{\hat{a}(u)}=2+\mathcal{O}\lb |u-v|^2\rb
\end{align*}
that
\begin{align*}
\frac{b_V-a_V}{2N\rho_V(x)}K_{N,V}\lb\lambda_V(u),\lambda_V(v)\rb =
\frac{\sin \lb N\pi \int_v^u \rho(r)\, \mathrm dr\rb}{\pi (s-t)}+\mathcal{O}\lb N^{-1}\rb.
\end{align*}
The claim now follows from
\begin{align*}
N\pi\int_v^u \rho (r) \, \mathrm dr = \pi (s-t)+\mathcal{O}\lb |s-t|\cdot \frac{|s|+|t|}{N}\rb
\end{align*}
and the crude estimate $|\sin(\alpha+\beta)-\sin(\alpha)|\leq |\beta|$.
\end{proof}

\begin{proof}[Proof of Theorem \ref{theorem_asymptotics}:]

We apply Theorem \ref{theorem_kernel} and Proposition \ref{proposition_kernel} with $\d = \d_0$.
Statements (i) and (ii) for values of $x\in\hat{J}$ that have distance greater than $\delta_0$ from $\pm 1$ follow from formulae (a) and (b) of Proposition \ref{proposition_kernel}, from $|k(x)|=\mathcal{O} ( e^{- N\eta(x)/2} ) $ for $|x|\geq 1+\delta_0$, from $|k(x)|=\mathcal{O}(1)$ for $|x|\leq 1-\delta_0$, and from
\begin{align*}
2 \frac{a'(x)}{a(x)}=\frac{1}{x^2-1}\quad \text{for } |x|>1, \qquad
2\frac{\hat{a}'(x)}{\hat{a}(x)}=\frac{1}{x^2-1}\quad \text{for } |x|<1 .
\end{align*}
Note that statement (iii) is not relevant for such values of $x$, i.e.~for $x \in \hat{J}$ satisfying $||x|-1| > \d_0$, as we choose the constant $c \le \d_0$. In fact, starting with $c=\delta_0$ the value of $c$ may be decreased a number of times as we proceed in the proof so that the same constant $c$ can be used simultaneously in all statements of Theorem \ref{theorem_asymptotics}.

We turn to the case $x\in[1-\delta_0,1+\delta_0]$. From Theorem \ref{theorem_kernel} and Proposition \ref{proposition_kernel} (c) it is clear that
\begin{align}\label{ABC}
&\frac{b_V-a_V}{2}K_{N,V}\lb \lambda_V(x),\lambda_V(x)\rb \\
=\,\, & \Aik \lb f_N(x),f_N(x)\rb f_N'(x)+2\Ai (f_N(x))\Ai'(f_N(x))\frac{d'(x)}{d(x)}+k(x)^T\mathcal{O}\lb N^{-1}\rb k(x)\notag\\
=:\,\, & A(x)+B(x)+C(x).\notag
\end{align}
For $x\in (1,1+\delta_0]$ the following asymptotic results on the Airy function and the Airy kernel are relevant. They are immediate from \cite[10.4.59, 10.4.61]{AS} and from $\Aik (\zeta,\zeta)=\Ai'(\zeta)^2-\zeta\Ai(\zeta)^2$.
\begin{comment}
satisfying $(x-1)^{-1}=o(N^{2/3})$ one has $s:=f_N(x)\to \infty$, as $N\to\infty$ (see Lemma \ref{lemma_f_hat} (iii)). Thus the following asymptotic results, which are immediate from \cite{xxx} \marginnote{Zitat} and $\Aik (s,s)=\Ai'(s)^2-s\Ai(s)^2,\, s\in\mathbb{R}$, are relevant.
\end{comment}
For $\zeta \geq 1$ and $u:=\frac{2}{3}\zeta^{3/2}$ one has
\begin{align}
&\Ai(\zeta)^2=\frac{1}{4\pi\sqrt{\zeta}}e^{-2u}\lb 1+\mathcal{O}\lb u^{-1}\rb\rb,\quad
\Ai'(\zeta)^2=\frac{\sqrt{\zeta}}{4\pi}e^{-2u}\lb 1+\mathcal{O}\lb u^{-1}\rb\rb,\label{airy^2}\\
& \Ai(\zeta)\Ai'(\zeta)=-\frac{1}{4\pi}e^{-2u}\lb 1+\mathcal{O}\lb u^{-1}\rb\rb,\label{airy_prime}\\
& \Aik (\zeta,\zeta)=\frac{1}{8\pi \zeta}e^{-2u}\lb 1+\mathcal{O}\lb u^{-1}\rb\rb\label{airy_kernel} .
\end{align}
Using in addition $f_N(x)=\O(N^{2/3} (x-1))$ and $f_N^{-1}(x)=\O(N^{-2/3} (x-1)^{-1})$ (see Lemmas \ref{lemma_f_hat} (iii) and \ref{lemma_G_V}), we obtain for $N^{2/3}(x-1)$ sufficiently large that the left hand side of \eqref{ABC} is of the form
\begin{align}
 f_N\Aik(f_N,f_N)\left[\frac{f_N'}{f_N}-4\frac{d'}{d}\lb1+\O\lb\frac{1}{N (x-1)^{3/2}}\rb\rb+\O\lb\frac{1}{N (x-1)^{1/2}}\rb\right],\label{airy_ABC}
\end{align}
where the three summands correspond to A, B and C in \eqref{ABC}. Observe further that $\frac23 (f_N(x))^{3/2}=\frac12N\eta(x)$ and $f_N=N^{2/3}\g^+(da)^4$ (see \eqref{def_a}, \eqref{eq_Th_1}, \eqref{eq_Th_3}, \eqref{eq_Th_4}), implying
\begin{align}
 \frac{f_N'(x)}{f_N(x)}=4\frac{d'(x)}{d(x)}+4\frac{a'(x)}{a(x)}=4\frac{d'(x)}{d(x)}+\frac{2}{x^2-1}\, .\label{f_and_a}
\end{align}
Now, statement (ii) of Theorem \ref{theorem_asymptotics} follows for $x\in (1+ c^{-1} N^{-2/3}, 1+\d_0]$ from \eqref{airy_ABC}, from $d'/d=\O(1)$, and from \eqref{airy_kernel}. The just presented arguments also work in the case $x \in [-1-\d_0,-1 - c^{-1} N^{-2/3} )$. Except for replacing $x-1$ by $\lv x\rv-1$ in the $\mathcal{O}$-terms, only \eqref{f_and_a} needs to be modified by
\begin{align}
  \frac{f_N'}{f_N}=4\frac{d'}{d}-4\frac{a'}{a}\label{f_and_-d}\quad \text{(see \eqref{eq_Th_3} and \eqref{eq_Th_4})}.
\end{align}
This neutralizes the change of sign stated in part (d) of Proposition \ref{proposition_kernel}.

Next, we consider $x\in [1-\d_0,1 - c^{-1} N^{-2/3})$. Setting $u:=\frac23 (-\zeta)^{3/2}$, the following asymptotics for $\zeta \leq-1$, derived from \cite[10.4.60, 10.4.62]{AS}, are useful.
\begin{align}
 &\Ai(\zeta)^2=\O((-\zeta)^{-1/2}), \quad \Ai'(\zeta)^2=\O((-\zeta)^{1/2}),\label{airy_squared}\\
 &\Ai(\zeta)\Ai'(\zeta)=-\frac{1}{2\pi}\lb \cos(2u)+\O\lb u^{-1}\rb\rb,\label{airy_times_airy_prime}\\
 &\Aik(\zeta,\zeta)=\frac{1}{\pi}(-\zeta)^{1/2}\lb 1-\frac{1}{6u}\cos(2u)+\O\lb u^{-2}\rb\rb.\label{airy_diagonal}
\end{align}
Together with Corollary \ref{cor_f}, one obtains in the notation of \eqref{ABC} and with $u =\frac{2}{3}(-f_N)^{3/2}$:
\begin{align*}
 A(x)+C(x)&=\Aik(f_N(x),f_N(x)) f_N'(x) \lb1+\O\lb \frac{1}{N^2 (1-x)}\rb\rb\\
 &=\frac{f_N'(x)}{f_N(x)}\lb -\frac{3}{2\pi}u+\frac{1}{4\pi}\cos(2u)+\O\lb u^{-1}\rb\rb\lb1+\O\lb\frac{1}{N^2 (1-x)}\rb\rb,\\
 B(x)&=-\frac{1}{\pi}\cos(2u)\frac{d'(x)}{d(x)}+\O\lb u^{-1}\rb.
\end{align*}
In addition to \eqref{f_and_a}, we observe from \eqref{eq_Th_1} and \eqref{def_xi} that
\begin{align}\label{u_and_xi}
 u=\frac{2}{3}(-f_N)^{3/2}=\frac{1}{2}N\xi,\ \text{ thus } \frac{f_N'}{f_N}=\frac{2\xi'}{3\xi} \text{ with }\xi'=-2\pi\rho.
\end{align}
In summary, $A+B+C$ is of the form
\begin{align*}
 N\rho(x)\lb1+\O\lb\frac{1}{N^2 (1-x)}\rb\rb+\frac{1}{2\pi(x^2-1)}\cos(N\xi(x))+\O\lb\frac{1}{N (1-x)^{5/2}}\rb.
\end{align*}
Since $\rho^{-1}(x)=\O((1-x)^{-1/2})$ by Lemma \ref{lemma_G_V}, statement (i) of Theorem \ref{theorem_asymptotics} follows near 1 by choosing the value of $c$ so small such that, say, $(2\pi (1-x^2))^{-1} < \frac{1}{2} N \rho(x)$ on $[1-\d_0,1 - c^{-1} N^{-2/3})$.

For $x\in(-1 + c^{-1} N^{-2/3} , -1+\d_0]$, one proceeds in exactly the same way, replacing in Proposition \ref{proposition_kernel} statement (c) by (d), \eqref{f_and_a} by \eqref{f_and_-d}, and \eqref{u_and_xi} by
\begin{align*}
 u=\frac23(-f_N)^{3/2}=\frac{1}{2}N(2\pi-\xi),\ \text{ thus } \frac{f_N'}{f_N}=-\frac{2\xi'}{3(2\pi-\xi)}.
 \end{align*}
 
In order to prove statement (iii) of Theorem \ref{theorem_asymptotics}, let us denote
\begin{align*}
 s:=N^{2/3}\g^+(x-1), \qquad \zeta:=f_N(x)=s\hat{f}(x)
\end{align*}
for $x\in[1-\d_0,1+\d_0]$. Since $\Aik(t,t)=\int_t^\infty\Ai(r)^2\, \mathrm dr$, we have
\begin{align}
 \frac{\Aik(\zeta, \zeta)}{\Aik(s,s)}=1-\lb\int_0^1 \frac{\Ai(s+r(\zeta-s))^2}{\Aik(s,s)}\, \mathrm dr\rb(\zeta-s).\label{airy_frac}
\end{align}
It follows from Lemma \ref{lemma_f_hat}(iii) that
\begin{align}
 \zeta-s=s(\hat{f}(x)-1)=s\O\lb\lv x-1\rv\rb=\O\lb N^{2/3}(x-1)^2\rb.\label{s_s_prime}
\end{align}
Using \eqref{airy^2}, \eqref{airy_kernel}, \eqref{airy_squared}, \eqref{airy_diagonal}, and $\Aik(t,t)^{-1}=\O(1)$ for $-1\leq t\leq1$, one obtains
\begin{align}\label{airy_frac2}
 \frac{\Ai(s+r(\zeta-s))^2}{\Aik(s,s)}=
 \begin{cases}
                                        \O\lb s^{1/2}\rb		&,\ \text{if }1\leq s\leq N^{4/15}\\
                                        \O\lb 1\rb				&,\ \text{if }-1\leq s\leq1\\
                                        \O\lb \lv s\rv^{-1}\rb	&,\ \text{if }s\leq-1
                                       \end{cases}
\end{align}
uniformly for $r\in[0,1]$ and $x\in[1-\d_0,1+\d_0]$. Applying in addition \eqref{airy_prime} and \eqref{airy_times_airy_prime}, we find
$B(x)+C(x)=N^{2/3}\g^+\Aik(\zeta, \zeta)\lb r_B(x)+\O\lb N^{-1}\rb\rb$ with
\begin{align}\label{r_B}
 r_B(x)=\begin{cases}
         \O\lb N^{-2/3}\rb+\O(\lv x-1\rv)& \ \text{for }x\in[1,1+\d_0] ,\\
         \O\lb N^{-2/3}\rb & \ \text{for }x\in[1-\d_0,1] .
        \end{cases}
\end{align}
Finally, claim (iii) of Theorem \ref{theorem_asymptotics} follows from \eqref{airy_frac}, \eqref{s_s_prime}, \eqref{airy_frac2}, \eqref{r_B} and
\begin{align*}
 A+B+C=N^{2/3}\g^+\Aik(s,s)\cdot\frac{\Aik(\zeta,\zeta)}{\Aik(s,s)}\lb \hat{f}+(x-1)\hat{f}'+r_B+\O\lb N^{-1}\rb\rb.
\end{align*}
\end{proof}

Before turning to the proof of Theorem \ref{theorem_edge} we state an asymptotic result for a special type of integrals, which can be verified using integration by parts. 
\begin{proposition}
\label{asymptotics_integral}
Let $\alpha, \beta \in \mathbb R$. Then for  $s,t \geq 1:$
\begin{align*}
&\int_0^{\infty} (s+r)^{\alpha} (t+r)^{\beta} e^{- \frac{2}{3} [(s+r)^{3/2} +(t+r)^{3/2}  ]} \, \mathrm dr&
\\
& = \frac{ s^{\alpha} \, t^{\beta}}{s^{1/2} + t^{1/2}} \, e^{- \frac{2}{3} (s^{3/2} +t^{3/2}  )}
\left( 1+ \mathcal{O}\lb s^{- \frac{3}{2}}\rb+\mathcal{O}\lb t^{- \frac{3}{2}}\rb \right).
\end{align*}
\end{proposition}

\begin{proof}[Proof of Theorem \ref{theorem_edge}:]
For $q \leq s, t \leq p N^{4/15}$ denote:
\begin{align*}
x&:= 1+ \frac{s}{ N^{2/3} \gamma^+}, & \hat{s}&:= f_N(x)=s \hat{f}(x), \\
y&:=  1+ \frac{t}{ N^{2/3} \gamma^+}, & \hat{t}&:= f_N(y)=t \hat{f}(y).
\end{align*}
It follows from Theorem \ref{theorem_kernel} and Proposition \ref{proposition_kernel} (c) that the Christoffel-Darboux kernel can be written in the form 
\begin{align*}
\frac{b_V - a_V}{2N^{\frac{2}{3}}\gamma^+ } K_{N,V} \left(  \lambda_V (x) , \lambda_V(y)  \right) = 
\hat{A} +  \hat{B} + \hat{C}
\end{align*}  
with \vspace{-7pt} 
\begin{align*}
\hat{A} & := \Aik (\hat{s},\hat{t}) \left( \hat{f}(x) + (y-1) \frac{\hat{f}(x) - \hat{f}(y)}{x-y} \right)\\
\hat{B} & := \frac{1}{N^{2/3} \gamma^+} \left( \frac{\Ai (\hat{s} )\Ai' (\hat{t}) }{ d(y)} + \frac{\Ai' (\hat{s} )\Ai (\hat{t}) }{ d(x)}\right)\frac{d(x) - d(y)}{x-y} \vspace{7pt}
\\ 
\hat{C} & := k(y)^T \mathcal{O} \lb N^{-5/3}\rb k(x).
\end{align*}
In the case $q \leq s,t \leq 2$ one may use the boundedness of the derivatives of $\Aik$, $\hat{f}$, and $d$ together with the bounds  $d^{-1}(x)$, $d^{-1}(y) =\mathcal{O}(1)$, and $k(x)$, $k(y) = \mathcal{O}(N^{1/6})$ to derive 
\begin{align*}
\hat{A} &  = \Aik (s,t) \lb 1+ \mathcal{O}(|x-1|) +  \mathcal{O}(|y-1|) \rb = \Aik (s,t) + \mathcal{O}\lb N^{- 2/3}\rb ,
\\
\hat{B} &  =  \mathcal{O}\lb N^{- 2/3}\rb , \qquad \text{and} \qquad
\hat{C}   =  \mathcal{O}\lb N^{- 4/3}\rb.
\end{align*}
This yields the desired result. 

Next, we assume $1 \leq s,t \leq p N^{4/15}$. Recall from (\ref{Darst_Ai}) that 
\begin{align*}
\Aik (u,v) = \int_0^{\infty} \Ai (u+r) \Ai (v+r) \, \mathrm dr.
\end{align*}
The asymptotics \cite[10.4.59, 10.4.61]{AS} of $\Ai$ resp.~$\Ai'$ and Proposition \ref{asymptotics_integral} lead to 
\begin{equation}
\label{Asy_AiK}
\Aik (\hat{s}, \hat{t}) = (4\pi)^{-1} (\hat{s} \hat{t})^{-\frac{1}{4} } (\hat{s}^{\frac{1}{2} } + \hat{t}^{\frac{1}{2} })^{-1}  e^{-\frac{2}{3} (\hat{s}^{3/2} + \hat{t}^{3/2})} \left( 1+ \mathcal{O}\lb\hat{s}^{-\frac{3}{2}}\rb + \mathcal{O}\lb\hat{t}^{-\frac{3}{2}}\rb\right).
\end{equation}
Using the asymptotics of the Airy function again, we obtain 
\begin{align*}
\hat{B} &= \Aik (\hat{s}, \hat{t}) \cdot \mathcal{O} \lb N^{-2/3} \lb\hat{s}^{\, 1/2} +\hat{t}^{\,1/2}\rb ^2 \rb
=\Aik (\hat{s}, \hat{t}) \cdot \mathcal{O} \left( N^{-1} (s+t) \right) \\
\hat{C} &= \Aik (\hat{s}, \hat{t})  \cdot \mathcal{O} \lb N^{-2} \lb s^{3/2}+ t^{3/2}\rb  
 + N^{-5/3} (s+t) + N^{-4/3} \lb s^{1/2} + t^{1/2}\rb\rb
 \\
 &= \Aik (\hat{s}, \hat{t}) \cdot \mathcal{O} \lb N^{-1}\rb
 \end{align*}
 and consequently 
 \begin{equation}
 \label{A_B_C}
 \hat{A} + \hat{B} + \hat{C} = \Aik (\hat{s}, \hat{t}) \left( 1+ \mathcal{O}\left( \frac{s+t}{N^{2/3}} \right) \right).
 \end{equation}
 Differentiating the integral representation in \eqref{Darst_Ai} and applying Proposition \ref{asymptotics_integral} yields the bound on the derivative
$$D \Aik (\tilde{s}, \tilde{t})= \Aik (s,t) \cdot \mathcal{O}\lb s^{1/2} + t^{1/2} \rb , $$
uniformly for $\tilde{s}, \tilde{t}$ between $s$ and $\hat{s}$ resp.~between $t$ and $\hat{t}$.
Since $s-\hat{s} = \mathcal{O} (s^2N^{-2/3})$ and $t-\hat{t} = \mathcal{O} (t^2N^{-2/3})$, this implies 
$$ \Aik (\hat{s}, \hat{t}) = \Aik (s,t) \left( 1+ \mathcal{O} \left(  \frac{s^{5/2} + t^{5/2}}{N^{2/3}}\right) \right).$$
In view of \eqref{A_B_C} the claim follows. 

Finally, we turn to the case $ q \leq s \leq 1$ and $2 \leq t \leq p N^{4/15}$. Elementary calculations together with \cite[10.4.59, 10.4.61]{AS} give 
\begin{align*}
\Ai (\hat{t})&= \Ai(t) \left(  1+ \mathcal{O} \left( \frac{t^{5/2}}{N^{2/3}} \right)\right), & \Ai'(\hat{t}) &=  \Ai'(t) \left(  1+ \mathcal{O} \left( \frac{t^{5/2}}{N^{2/3}} \right)\right), \\
\Ai (\hat{s})&= \Ai(s) + \mathcal{O} \lb N^{-2/3}\rb , & \Ai'(\hat{s}) &=  \Ai'(s) +   \mathcal{O} \lb N^{-2/3}\rb ,
\\
\hat{s} - \hat{t}&= (s-t) \left(1+ \mathcal{O}\left(\frac{t}{N^{2/3}}\right)\right), &\frac{1}{s-t}&= \mathcal{O} \left( t^{-1}\right).
\end{align*}
These relations suffice to derive 
\begin{align*}
\hat{A} &=\hspace{-3pt} \left( \Aik (s,t)  + \Ai'(t) \cdot \mathcal{O}\hspace{-3pt}\left( \frac{t^{3/2}}{N^{2/3}} \right) \! \right) 
\hspace{-3pt} \left( 1+  \mathcal{O}\hspace{-3pt}\left( \frac{t}{N^{2/3}} \right) \! \right) \hspace{-3pt}
=  \Aik (s,t)  + \Ai'(t) \cdot \mathcal{O}\hspace{-3pt}\left( \frac{t^{3/2}}{N^{2/3}} \right),
\\
\hat{B} &= \Ai'(t) \mathcal{O}\lb N^{-2/3}\rb , \qquad \text{and} \qquad
\hat{C} = \Ai'(t) \mathcal{O}\lb N^{-4/3}\rb ,
\end{align*}
which completes the proof.
\end{proof}

\renewcommand{\thesection}{\Alph{section}}
\renewcommand{\theequation}{\Alph{section}.\arabic{equation}}
\setcounter{section}{0}
\setcounter{equation}{0}
\section{Appendix: Construction of the parametrix near $\pm 1$}\label{appendix}
Since the definition of the parametrix $\mathbb{T}$ near $\pm 1$ is somewhat involved, we first motivate the construction. The validation of the non-obvious claims made along the way begins with Lemma \ref{lemma_f_hat}. Our presentation is similar to \cite[Section 7.6]{Deift} and \cite[Subsection 6.4.6]{handbookKuijlaars}.

Assume that $V$ satisfies \GAi\ and let $q$, $\tilde{q}$, $G\equiv G_V$, $\xi\equiv \xi_V$, $\eta\equiv \eta_V$, $\rho\equiv \rho_V$ be given as in  \eqref{def_q}, \eqref{def_q_tilde}, and \eqref{def_G} to \eqref{def_eta}. We begin the construction by introducing another auxiliary function. Let $\kappa>0$ be chosen such that $G$ has an analytic extension to $B_{\kappa}\lb -1\rb \cup B_{\kappa}\lb 1\rb$. Set
\begin{align}\label{def_varphi}
\varphi\lb z\rb := \begin{cases} 
\frac{1}{2}\int _1^z q\lb \zeta\rb G\lb \zeta\rb \,\mathrm d\zeta &, \text{ if } z\in B_{\kappa}\lb 1\rb \backslash \lb 1-\kappa, 1\right] ,
\vspace{5pt}\\
\frac{1}{2} \int_{-1}^z q\lb \zeta\rb G\lb\zeta\rb\,\mathrm  d\zeta &, \text{ if } z\in B_{\kappa}\lb -1\rb \backslash \left[-1,-1+\kappa\rb .
\end{cases}
\end{align}
The map $\varphi$ is analytic. According to Definition \ref{Def_xi_eta_complex}, we have
\begin{align}\label{eq_varphi_eta_xi}
\varphi\lb z\rb =\frac{1}{2} \eta \lb z\rb \quad \mbox{ and } \quad \varphi \lb z\rb =\begin{cases}
\mp \frac{i}{2}\xi\lb z\rb & \mbox{near $\phantom{-}1$, } \Im (z) \gtrless 0 ,  \vspace{5pt}\\
\mp \frac{i}{2}\lb\xi\lb z\rb -2\pi\rb & \mbox{near $-1$, } \Im (z) \gtrless 0 ,
\end{cases}
\end{align}
on the corresponding common domains of definition. In order to see this use $q\lb z\rb = \pm i \tilde{q}\lb z\rb$ for $\Im (z) \gtrless 0$, $\int_{-1}^z \tilde{q} G =\int_{-1}^1 2\pi \rho + \int_1^z \tilde{q}G $, and Lemma \ref{lemma_Hilbert} (a).

The reason for defining $\varphi$ is the relation
\begin{align}\label{eq_varphi_v_T}
e^{N\varphi_- \sigma_3} v_T e^{-N\varphi_+\sigma_3}=\begin{pmatrix}
1 &1 \\0 & 1
\end{pmatrix} \quad \mbox{ on } \lb -1-\kappa ,-1+\kappa\rb \cup \lb 1-\kappa ,1+\kappa\rb ,
\end{align}
which is immediate from \eqref{eq_varphi_eta_xi} and \eqref{def_v_T}. Recall from Lemma \ref{Lemma_R} that it is one of the desired properties of the parametrix $\mathbb{T}$ that $R=T \, \mathbb{T}^{-1}$ has no jumps in regions III$^{\pm}$. By (\ref{v_R}) this implies $\mathbb{T}_+ =\mathbb{T}_- v_T$ on III$^{\pm}\cap\mathbb{R}$.

Writing $\mathbb{T}$ in the form $\mathbb{T}=\hat{\mathbb{T}}e^{N\varphi\sigma_3}$, we obtain for $\hat{\mathbb{T}}$ the constant jump
\begin{align}\label{jump_hat_T}
\hat{\mathbb{T}}_+ =\hat{\mathbb{T}}_-
\begin{pmatrix}
1 & 1 \\ 0 & 1
\end{pmatrix} \quad\mbox{ on } \lb -1-\kappa ,-1+\kappa \rb \cup \lb 1-\kappa ,1+\kappa \rb .
\end{align}
The next ingredient in the construction is the observation that for $\omega := e^{2\pi i/3}$ the matrix-valued function
\begin{align}\label{def_psi_0}
\psi_0\lb\zeta\rb :=\begin{cases}
\begin{pmatrix}
\text{Ai}(\zeta ) & \text{Ai} (\omega ^2 \zeta ) \\ \text{Ai}'(\zeta ) & \omega ^2 \text{Ai}'(\omega ^2 \zeta )
\end{pmatrix} e^{-\frac{\pi i}{6}\sigma_3} &, \text{ if } \zeta\in \lbrace z\in\mathbb{C}\,|\, \Im z>0\rbrace ,\\
\begin{pmatrix}
\text{Ai} (\zeta ) & -\omega ^2 \text{Ai} (\omega\zeta ) \\ \text{Ai}' (\zeta ) & -\text{Ai}' (\omega\zeta )
\end{pmatrix} e^{-\frac{\pi i}{6}\sigma_3} &, \text{ if } \zeta\in \lbrace z\in\mathbb{C}\,|\, \Im z<0\rbrace ,
\end{cases}
\end{align}
satisfies the same jump condition as $\hat{\mathbb{T}}$,
\begin{align}\label{jump_psi_0}
\lb \psi_0\rb _+ = \lb \psi_0 \rb _-\begin{pmatrix}
1&1 \\ 0 & 1
\end{pmatrix} \quad\mbox{ on } \mathbb{R} ,
\end{align}
due to the relation $\text{Ai}(\zeta) + \omega\text{Ai}(\omega\zeta) + \omega ^2 \text{Ai}(\omega ^2\zeta)=0$ for all $\zeta\in\mathbb{R}$ (see \cite[10.4.7]{AS}).

Note that there are plenty of maps $\map{A}{\mathbb{C}\backslash\mathbb{R}}{\mathbb{C}^{2\times 2}}$ with jump relation (\ref{jump_psi_0}). What leads to the choice of $\psi_0$ is the additional condition $v_R=I+\mathcal{O}( N^{-1})$ on $\Sigma_3^{\pm}$ from Lemma \ref{Lemma_R} that is essential for proving Theorem \ref{theorem_R_+}. In order to obtain the required jump relation for $v_R$ across $\Sigma_3^{\pm}$ we need more freedom in the construction. Observe first that for any entire function $f$ mapping the upper resp.~lower half-plane into itself the composition $\psi_0\circ f$ also satisfies (\ref{jump_psi_0}), because the jump matrix is constant. The crucial property of $\psi_0$ is its asymptotic behavior (cf. \eqref{Psi_asy} below)
\begin{align}\label{psi_0_asmp}
\psi_0\lb\zeta\rb e^{\frac{2}{3}\zeta ^{3/2}\sigma_3} \sim \frac{e^{\frac{\pi i}{12} }}{2\sqrt{\pi}}
\begin{pmatrix}
\zeta^{-\frac{1}{4}} & 0 \\ 0 & \zeta ^{\frac{1}{4}}
\end{pmatrix}
\begin{pmatrix}
1 & 1 \\ -1 & 1
\end{pmatrix}
e^{-\frac{\pi i }{4}\sigma_3}
\end{align}
as $\zeta\to\infty$. We may use this as follows to define the parametrix $\mathbb{T}$ on, say $B_{\kappa}\lb 1\rb$. Suppose there exists an analytic function $\map{f_N}{B_{\kappa}\lb 1\rb}{\mathbb{C}}$ satisfying $\frac{2}{3} f_N\lb z\rb ^{3/2}=N\varphi\lb z\rb$. Define $\mathbb{T}_0:= \lb \psi_0 \circ f_N\rb e^{N\varphi \sigma_3}$. By the above discussion we have
\begin{align}\label{jump_T_0}
\lb \mathbb{T}_0\rb _+ = \lb \mathbb{T}_0 \rb _- v_T \quad \mbox{ on } \text{III}^{+}\cap \mathbb{R}.
\end{align}
By (\ref{v_R}) the jump matrix $v_R$ will be equal to $\mathbb{T}_0 M^{-1}$ (see (\ref{def_M})) on the arc $A:=\Sigma_3^+ \cap \lbrace z\in\mathbb{C}\, |\, |\arg \lb z-1 \rb |<\frac{3\pi}{4}\rbrace$ if we choose $\mathbb{T}=\mathbb{T}_0$ in III$^+$, because $\mathbb{T}_+=M$ on $A$. 
As $N\to\infty$, $f_N\lb z\rb$ also tends to infinity and (\ref{psi_0_asmp}) implies
\begin{align*}
\mathbb{T}_0\lb s\rb M^{-1}\lb s\rb \sim \frac{1}{2\sqrt{\pi}} e^{-\frac{\pi i}{6}} 
\begin{pmatrix}
0 & a(s)f_N^{-1/4}(s) \\ -f_N^{1/4}(s) a^{-1}(s) & 0
\end{pmatrix}
\begin{pmatrix}
1 & -i \\ 1 & i
\end{pmatrix}
\end{align*}
for $s\in A$. These are not the desired asymptotics for $v_R\sim I$ on $A$. However, defining
\begin{align}\label{def_E_N}
E_N:= \sqrt{\pi}e^{\frac{\pi i}{6}}
\begin{pmatrix}
1 & -1 \\ -i & -i
\end{pmatrix}
\begin{pmatrix}
f_N^{1/4}a^{-1} & 0 \\ 0 & f_N^{-1/4} a
\end{pmatrix} \quad \mbox{and} \quad \mathbb{T}:= E_N \mathbb{T}_0 \quad \mbox{in III}^+ ,
\end{align}
one obtains $v_R=\mathbb{T}_-\mathbb{T}_+^{-1}\sim I$ on $A$. Moreover, it turns out that $E_N$ has an analytic extension in III$^+$ so that jump relation (\ref{jump_T_0}) still holds for $\mathbb{T}$. Finally, we need to consider the jumps of $v_R$ on $\Sigma_3^+ \backslash A$. Here, $\mathbb{T}_+=Mv_u$ resp.~$\mathbb{T}_+=Mv_l^{-1}$. Recall that the definition of $\mathbb{T}$ in (\ref{def_T2}), (\ref{def_T3}) is based on the factorization of $v_T$ as given in (\ref{RHP_factor}). Applying the same procedure inside the disk III$^+$ with the factorization
\begin{align}
\begin{pmatrix}
1&1\\ 0&1
\end{pmatrix}=\begin{pmatrix}
1&0 \\ 1 & 1
\end{pmatrix}\begin{pmatrix}
0 & 1 \\ -1 & 0
\end{pmatrix}\begin{pmatrix}
1 & 0 \\ 1 & 1
\end{pmatrix}
\end{align}
for the jump matrix of $\hat{\mathbb{T}}$ (see (\ref{jump_hat_T})), one is led to a modification in the definition of $\psi_0$. It consists of replacing $\psi_0$ by 
\begin{align}\label{def_psi_beta}
\psi_{\beta}\lb z\rb = \begin{cases}
\psi_0\lb z\rb  & \mbox{for } z\in\Omega_1\cup \Omega_4 ,\\
\psi_0\lb z\rb \begin{pmatrix}
1&0\\ 1&1
\end{pmatrix}^{-1} & \mbox{for } z\in\Omega_2 ,\\
\psi_0\lb z\rb \begin{pmatrix}
1&0 \\ 1&1
\end{pmatrix}& \mbox{for } z\in\Omega_3 ,
\end{cases}
\end{align}
where the unbounded sectors $\Omega_j$ depend on the angle $\beta$ as displayed in Figure \ref{fig:3} below. 
The value for $\beta$ will be specified in \eqref{def_Beta}. The following Lemma allows to define the function $f_N$ satisfying $\frac{2}{3} f_N^{3/2}=N\varphi$ and to verify all the properties that were used in the construction described above.

\begin{lemma}\label{lemma_f_hat}
%\begin{enumerate}
%\item[(a)]
(a)\
 Let $V$ satisfy \GAi\ and let $\varphi$, $\gamma^{\pm}\equiv\gamma_V^{\pm}$ be defined as in (\ref{def_varphi}), (\ref{def_gamma}). Then there exist $\hat{\sigma}>0$ and an analytic function $\map{\hat{f}}{B_{\hat{\sigma}}\lb -1\rb \cup B_{\hat{\sigma}}\lb 1\rb}{\mathbb{C}}$ such that (i) to (iii) hold.
\begin{enumerate}
\item[(i)] 
$\left[\phantom{-} \gamma ^+\lb z-1\rb \hat{f}\lb z\rb \right] ^{3/2} =\frac{3}{2}\varphi\lb z\rb\quad$ 
for $z\in B_{\hat{\sigma}}\lb 1\rb \backslash \left( 1-\hat{\sigma},1\right]$, \vspace{3pt}\\
$\left[ -\gamma^-\lb z+1\rb\hat{f}\lb z\rb\right] ^{3/2}=\frac{3}{2}\varphi\lb z\rb\quad$ for 
$z\in B_{\hat{\sigma}}\lb -1\rb\backslash \left[-1,-1+\hat{\sigma}\rb$, \vspace{3pt}
\item[(ii)] $\hat{f}\lb s\rb \in\mathbb{R}\quad$ for all $\quad s\in\mathbb{R}\cap\lb B_{\hat{\sigma}}\lb 1\rb \cup B_{\hat{\sigma}}\lb -1\rb\rb$, \vspace{3pt}
\item[(iii)] $\hat{f}\lb \pm 1\rb =1\quad$ \mbox{and }$\quad\hat{f}\lb B_{\hat{\sigma}}\lb\pm 1\rb\rb\subset B_{1/5}\lb 1\rb$.
\end{enumerate}
%\item[(b)]

(b)\
Assume that $Q$ satisfies \GAii . Then there exists an open neighborhood $\mathcal{U}$ of $Q$ in $\lb X_D,\Vert\cdot \Vert_{\infty}\rb$ (see Remark \ref{remark_complex_entension}) such that statement (a) holds for all $V\in\mathcal{U}$ with $\hat{\sigma}$ independent of $V$.
%\end{enumerate}
\end{lemma}

\begin{proof}
Let $\sigma$ be defined as in (the end of) Remark \ref{remark_G_V} resp.~as in Lemma \ref{lemma_G_V} in the situation of (a) resp.~(b). This implies in particular that $G \equiv G_V$ possesses an analytic extension on $B_{\sigma}(\pm 1)$. Note that in the case of (b) Lemma \ref{lemma_G_V} also provides a neighborhood $\mathcal{U}$ of $Q$ such that the same value of $\sigma$ can be used for all $V \in \mathcal{U}$ simultaneously. We introduce the auxiliary function
\begin{align*}
h\lb z\rb := \frac{1}{2} G\lb z\rb\lb 1\pm z\rb^{1/2} ,
\end{align*}
which is analytic on $B_{\sigma}\lb\pm 1\rb$. By definition \eqref{def_varphi} we have
\begin{align*}
\varphi \lb z\rb =\pm\int_{\pm 1}^z\lb\pm s-1\rb ^{1/2} h\lb s\rb \,\mathrm ds\quad \text{for } z\in B_{\sigma}\lb\pm 1\rb\setminus\left[ -1,1\right] .
\end{align*}
We write $h\lb z\rb =h\lb\pm 1\rb + \tilde{h}\lb z\rb\lb z\mp 1\rb$ with $\tilde{h}$ analytic on $B_{\sigma}\lb\pm 1\rb$. Reducing $\sigma$ and $\mathcal{U}$, if necessary, we may ensure that $G$ and $G'$ are both bounded on $B_{\sigma}\lb \pm 1 \rb$ (uniform in $V$ in the situation of (b), cf.~arguments leading to \eqref{ineq_G}). Thus, such a uniform bound also exists for $\tilde{h} \equiv \tilde{h}_V$, i.e.
\begin{align}\label{esti_h_tilde}
\exists C>0\, \forall z\in B_{\sigma}\lb\pm 1\rb \, \lb \forall V\in\mathcal{U}\rb : \, |\tilde{h}_V\lb z\rb |\leq C.
\end{align}
A short calculation gives
\begin{align*}
\frac{3}{2}\varphi \lb z\rb &= h \lb\pm 1\rb\lb\pm z-1\rb ^{3/2} \lb 1+r\lb z\rb\rb \qquad\text{with}\\
r\lb z\rb &:=\frac{3}{2h\lb\pm 1\rb\lb\pm z-1\rb ^{3/2}}\lb \int_{\pm 1}^z \lb\pm s-1\rb ^{3/2}\tilde{h}\lb s\rb \,\mathrm ds\rb
\end{align*}
for $z\in B_{\sigma}\lb\pm 1\rb \setminus \left[ -1,1\right]$. Observe that $r$ has an analytic extension to all of $B_{\sigma}\lb\pm 1\rb$. Moreover, the upper bound \eqref{esti_h_tilde} on $\tilde{h}$ together with the lower bound on $h\lb \pm 1 \rb$ provided by Lemma \ref{lemma_G_V} resp.~by Remark \ref{remark_G_V} imply
\begin{align}\label{esti_r}
\exists C_0>0\, \forall z\in B_{\sigma}\lb\pm 1\rb \, \lb \forall V\in\mathcal{U}\rb :\, | r_V\lb z\rb |\leq C_0|z\mp 1| .
\end{align}
Choosing $\hat{\sigma}:=\min (\sigma,\, \lb 5C_0\rb ^{-1})$ it follows from \eqref{esti_r} that
\begin{align*}
\map{\hat{f}}{B_{\hat{\sigma}}\lb 1\rb \cup B_{\hat{\sigma}}\lb -1\rb}{\mathbb{C}},\quad \hat{f}\lb z\rb := \lb 1+r\lb z\rb\rb ^{2/3}
\end{align*}
defines an analytic function that satisfies claim (iii). 
To see property (ii), one first verifies from the construction that $\hat{f}$ is real on $\lb -1-\hat{\sigma},-1\right]\cup\left[1,1+\hat{\sigma}\rb$.
The claim is then a consequence of the identity principle. The relations in (i) are immediate from the construction and from 
$\gamma^{\pm}=\lb h\lb\pm 1\rb\rb^{2/3}$ (see \eqref{def_gamma}).
\end{proof}

\begin{remark}\label{remark_f_hat}
In view of Definition \ref{Def_xi_eta_complex} and relation \eqref{eq_varphi_eta_xi} it follows from statement (i) of Lemma \ref{lemma_f_hat} that the restriction of $\hat{f}$ to the reals has all the properties that are claimed in Definition \ref{def_notation} with, say, $\delta_V = \hat{\sigma}/2$.
\end{remark}

\begin{corollary}\label{cor_f}
Let the assumptions of Lemma \ref{lemma_f_hat} (a) resp.~(b) be satisfied. Define $\map{f}{B_{\hat{\sigma}}\lb -1\rb\cup B_{\hat{\sigma}}\lb 1\rb}{\mathbb{C}}$ by $f\lb z\rb :=\lb z\mp 1\rb \hat{f}\lb z\rb$ for $z\in B_{\hat{\sigma}}\lb \pm 1\rb$, with $\hat{\sigma}$, $\hat{f}$ as provided by Lemma \ref{lemma_f_hat}. Then $f'\lb B_{\hat{\sigma}/10}\lb\pm 1\rb\rb\subset B_{1/3}\lb 1\rb$.
\end{corollary}

\begin{proof}
From Lemma \ref{lemma_f_hat} (iii) it follows for all $z\in B_{\hat{\sigma}/10}\lb\pm 1\rb$ that $|\hat{f}\lb z\rb -1|<\frac{1}{5}$ and that $|\hat{f}'\lb z\rb |\leq \frac{6}{5}( \frac{9}{10}\hat{\sigma}) ^{-1}$, thus $|f'\lb z\rb -1|<\frac{1}{5}+\frac{6}{45}=\frac{1}{3}$.
\end{proof}

We are now ready to define the parametrix $\mathbb{T}$ on III$^{\pm}$. We begin with the disk III$^+=B_{\delta}\lb 1\rb$ for $0<\delta\le\sigma_0\le\frac{\hat{\sigma}}{10}$ (see \eqref{def_sigma_0} for a definition of $\sigma_0$ and Lemma \ref{lemma_f_hat} for $\hat{\sigma}$). Let $\hat{f}$, $f$ be given by Lemma \ref{lemma_f_hat} resp.~ by Corollary \ref{cor_f} and set in accordance with Definition \ref{def_notation}
\begin{align}\label{def_f_N}
f_N\lb z\rb := N^{2/3}\gamma^+ f\lb z\rb =N^{2/3}\gamma^+\lb z-1\rb\hat{f}\lb z\rb, \quad z\in B_{\delta}\lb 1\rb .
\end{align}
Observe that $E_N$  of \eqref{def_E_N} is now defined on the domain $B_{\delta}\lb 1\rb\setminus\lb 1-\delta,1\right]$. Moreover, $E_N$ can be extended analytically to $B_{\delta}\lb 1\rb$ since the factor $\lb z-1\rb^{1/4}$ cancels and $\hat{f}^{\pm 1/4}$ is analytic by statement (iii) of Lemma \ref{lemma_f_hat}. Note in addition, that Corollary \ref{cor_f} implies that $\map{f}{B_{\delta}\lb 1\rb}{f\lb B_{\delta}\lb 1\rb\rb}$ is biholomorphic and statement (iii) of Lemma \ref{lemma_f_hat} yields
\begin{align}\label{def_Beta}
\beta_{\delta}^+ := \arg \lb f\lb 1+\delta e^{\frac{3\pi i}{4}}\rb\rb \in\lb \frac{5\pi}{8},\frac{7\pi}{8}\rb.
\end{align}
The significance of the point $1+\delta e^{3\pi i /4}$ is that $\Sigma_2^u$ and $\Sigma_3^+$ of Figure \ref{fig:2} intersect here. By Lemma \ref{lemma_f_hat} (ii), we have $\arg ( f( 1+\delta e^{-3\pi i/4})) = -\beta_{\delta}^+$. Splitting up $f\lb B_{\delta}\lb 1\rb\rb$ as displayed in Figure \ref{fig:3} with $\beta=\beta_{\delta}^+$, the regions III$_j^+$ are defined by the relations $f ( \text{III}_j^+ ) =\Omega_j\cap f\lb B_{\delta}\lb 1\rb\rb$ for $j=1,2,3,4$. 
Properties (ii) and (iii) of Lemma \ref{lemma_f_hat} imply that III$_1^+$, III$_2^+$ are located in the upper half-plane, whereas III$_3^+$, III$_4^+$
lie below $\mathbb{R}$. Moreover, the joint boundary of III$_1^+$ and III$_2^+$ resp.~III$_3^+$ and III$_4^+$ are curves that connect 1 with $1+\delta e^{3\pi i/4}$ (resp.~$1+\delta e^{-3\pi i/4}$). They can also be found in Figure \ref{fig:2} as dotted lines. 

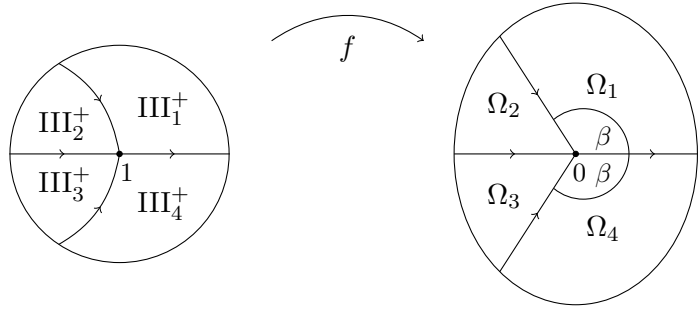
\begin{figure}[h]
\begin{tikzpicture}

% Koordinaten
\coordinate[label=270:{\small $ \phantom{..}1$}]   (Center1) at (-5,0);
\coordinate   (A) at (-5.8,1.2);
\coordinate (B) at (-5.8,-1.2);

% Kreismittelpunkt
\draw[fill=black] (Center1) circle (1pt);

%Kreis
\node (kreis) at (Center1) [draw, circle through=(A)] {};

 % Beschriftung Kreis
\coordinate[label=0:{ III$_2^+$}]   (III2) at (-6.2,0.4);
\coordinate[label=0:{ III$_3^+$}]   (III3) at (-6.2,-0.4);
\coordinate[label=0:{ III$_1^+$}]   (III4) at (-4.9 ,0.6);
\coordinate[label=0:{ III$_4^+$}]   (III5) at (-4.9,-0.66);

%geschwungene Linien in ersten Kreis
\draw[decoration={markings, mark=at position 0.5 with {\arrow{>}}},
       postaction={decorate}] 
			(B) to[out=30, in=-100] (Center1);
		
\draw[decoration={markings, mark=at position 0.5 with {\arrow{>}}},
       postaction={decorate}] 
			(A) to[out=-30, in=100] (Center1);
			
% gerade Linien im ersten Kreis	 			
\begin{comment}
%%%%%%%%%%%%%%%%%%%%%%% Schwierigkeiten - Anfang %%%%%%%%%%%%%%%%
\draw[decoration={markings, mark=at position 0.5 with {\arrow{<}}},
       postaction={decorate}]   
let \p1 = ($(A)-(Center1)$) in
(Center1)--++(180:({veclen(\x1,\y1)});

 \draw[decoration={markings, mark=at position 0.5 with {\arrow{>}}},
       postaction={decorate}]   
let \p1 = ($(A)-(Center1)$) in
(Center1)--++(0:({veclen(\x1,\y1)});
%%%%%%%%%%%%%%%%%%%%%%%% Schwierigkeiten - Ende %%%%%%%%%%%%%%%%%%%
\end{comment}

\draw[decoration={markings, mark=at position 0.5 with {\arrow{>}}},
       postaction={decorate}] (-6.45,0)--(Center1);
\draw[decoration={markings, mark=at position 0.5 with {\arrow{>}}},
       postaction={decorate}] (Center1)--(-3.55,0);

% Abbildungspfeil (von F nach G)
\coordinate (F) at (-3,1.5);
\coordinate (G) at (-1,1.5);
\draw[->] (F) to[out=35, in=145] (G);
\coordinate (FG1) at ($(F)!.5!(G)$);
\coordinate[label=-90:{  $f$}] (FG2) at ($(FG1)+(0,0.2)$);

%Ellipse

%Mittelpunkt
\coordinate[label=-90:{\small $ \phantom{.} 0$}]   (Center2) at (1,0);
\draw (Center2) ellipse (1.6 and 2);
 \draw[fill=black] (1,0) circle (1pt);

% Punkte auf Ellipse 
\coordinate (t) at (0,1.5621);
\coordinate (s) at (0,-1.5621);

% Linien in Ellipse
\draw[decoration={markings, mark=at position 0.5 with {\arrow{>}}},
       postaction={decorate}] (t)--(Center2);
\draw[decoration={markings, mark=at position 0.5 with {\arrow{>}}},
       postaction={decorate}] (s)--(Center2);
\draw[decoration={markings, mark=at position 0.5 with {\arrow{>}}},
       postaction={decorate}] (-0.6,0)--(Center2);
\draw[decoration={markings, mark=at position 0.65 with {\arrow{>}}},
       postaction={decorate}] (Center2)--(2.6,0);

% Winkel mit Beschriftung
\draw(1.7,0) arc (0:130:0.6);
\coordinate[label=0:{\small $\beta$}]   (C) at (1.12,0.2);

\draw(1.7,0) arc (0:-130:0.6);
\coordinate[label=0:{\small $\beta$}]   (C) at (1.12,-0.28);

%% Beschriftung Ellipse
\coordinate[label=0:$\Omega_2$]   (O1) at (-0.3,0.65);
\coordinate[label=0:$\Omega_3$]   (O2) at (-0.3,-0.56);
\coordinate[label=0:$\Omega_1$]   (O3) at (1,0.9);
\coordinate[label=0:$\Omega_4$]   (O4) at (1,-0.99);
\end{tikzpicture}
\caption{Definition of subregions III$^+_j$ via $f$ as given in Corollary \ref{cor_f}}
\label{fig:3}
\end{figure}

As motivated above (see in particular the discussion between \eqref{psi_0_asmp} and \eqref{def_psi_beta}) the parametrix is given by 
\begin{align}\label{def_matbb_T}
\mathbb{T}\lb z\rb :=\begin{cases}
E_N\lb z\rb \psi_{\beta_{\delta}^+}\lb f_N\lb z\rb\rb e^{N\varphi\lb z\rb \sigma_3} &,\text{ if } z\in \text{III}_1^+\cup \text{III}_4^+ ,\\
E_N\lb z\rb \psi_{\beta_{\delta}^+}\lb f_N\lb z\rb\rb e^{N\varphi\lb z\rb \sigma_3}\begin{pmatrix}
1 & 0 \\ e^{-iN\xi\lb z\rb} &1
\end{pmatrix} &, \text{ if } z\in \text{III}_2^+ ,\\
E_N\lb z\rb \psi_{\beta_{\delta}^+}\lb f_N\lb z\rb\rb e^{N\varphi\lb z\rb \sigma_3}\begin{pmatrix}
1 & 0\\ e^{iN\xi\lb z \rb  }&1
\end{pmatrix}^{-1} &,\text{ if } z\in\text{III}_3^+ .
\end{cases}
\end{align}
The definition of $\mathbb{T}$ on III$^-$ can be summarized as follows (cf.~Definition \ref{def_notation}):
\begin{align}
f_N\lb z\rb &:= -N^{2/3}\gamma^- f\lb z\rb =-N^{2/3} \gamma^-\lb z+1\rb \hat{f}\lb z\rb, \ z\in B_{\delta}\lb -1\rb ,
\label{def_f_N-} \vspace{4pt}\\
E_N &:= \sqrt{\pi}e^{\frac{\pi i}{6}}\begin{pmatrix}
1&1 \\ i & -i
\end{pmatrix}
\begin{pmatrix}
f_N^{1/4}a & 0\\ 0 & f_N^{-1/4}a^{-1}
\end{pmatrix},
\, \text{ extended to } B_{\delta}\lb -1\rb , \label{def_E_N-} \vspace{4pt}\\
\beta_{\delta}^- &:= \arg \lb -f\lb -1+\delta e^{-\frac{\pi i}{4}}\rb\rb \in\lb \frac{5\pi}{8},\frac{7\pi}{8}\rb .
\end{align}
Splitting $-f\lb B_{\delta}\lb -1\rb\rb$ into wedges $\Omega_j$ with $\beta=\beta_{\delta}^-$, we define III$_j^-$ via $-f ( \text{III}_j^- ) =\Omega_j \cap \left[-f\lb B_{\delta}\lb -1\rb\rb\right]$. Observe that the regions above $\mathbb{R}$ are now III$_3^-$ and III$_4^-$. Then
\begin{align} \label{def_mathbb_T-}
\mathbb{T}\lb z\rb :=\begin{cases}
E_N\lb z\rb\sigma_3\psi_{\beta_{\delta}^-}\lb f_N\lb z\rb\rb \sigma_3 e^{N\varphi\lb z\rb \sigma_3}&,\text{ if } z\in \text{III}_1^- \cup \text{III}_4^- ,  \\
E_N\lb z\rb\sigma_3\psi_{\beta_{\delta}^-}\lb f_N\lb z\rb\rb \sigma_3 e^{N\varphi\lb z\rb \sigma_3}\begin{pmatrix}
1 & 0 \\ e^{iN\xi\lb z\rb} &1
\end{pmatrix}^{-1}   &,\text{ if } z\in \text{III}_2^- ,\\
E_N\lb z\rb\sigma_3\psi_{\beta_{\delta}^-}\lb f_N\lb z\rb\rb \sigma_3 e^{N\varphi\lb z\rb \sigma_3}\begin{pmatrix}
1 & 0\\ e^{-iN\xi\lb z\rb} & 1
\end{pmatrix}&,\text{ if } z\in \text{III}_3^- .
\end{cases}
\end{align}
Finally, we state the main result of the Appendix. It is crucial for proving Lemma~\ref{Lemma_R}.
\begin{lemma}\label{lemma_R_2}
Let the assumptions of Lemma \ref{Lemma_R} (a) resp.~(b) be satisfied and let $R$ be given by \eqref{Def_R}, see also \eqref{def_sigma_0} to \eqref{def_T4}, \eqref{def_matbb_T}, and \eqref{def_mathbb_T-}. Then $R$ has an analytic continuation on III$^{\phantom{.}+}\, \cup\,$III$^{\phantom{.}-}$. For $\Delta_R \equiv v_R-I$ we have
\begin{align*}
\Vert \Delta_R\Vert _{L^{\infty}\lb \Sigma_3^+\cup \Sigma_3^-\rb}=\mathcal{O}\lb N^{-1}\rb .
\end{align*}
The error bound is uniform for $\lb \delta,\varepsilon\rb$ from an arbitrary but fixed compact subset of $\lb 0,\sigma_0\right]^2$. In the situation of Lemma \ref{Lemma_R} (b) the error bound is also uniform in $V\in\mathcal{U}$ for some open neighborhood $\mathcal{U}$ of $Q$.
\end{lemma}

\begin{proof}
Using \eqref{v_R}, \eqref{def_varphi}, \eqref{eq_varphi_eta_xi}, and \eqref{jump_psi_0} it follows from the definition of $\mathbb{T}$ that all jump matrices $v_R$ equal $I$ within III$^{\pm}$. Hence, $R$ has an analytic continuation in these discs. Moreover, one easily derives
\begin{align*}
v_R=E_N [\psi_{\beta_{\delta}^+}\circ f_N] e^{N\varphi \sigma_3}M^{-1}\quad\text{on }\Sigma_3^+.
\end{align*}
Inserting formulae \cite[10.4.7, 10.4.59, 10.4.61]{AS} into definitions \eqref{def_psi_0}, \eqref{def_psi_beta} gives
\begin{align}\label{Psi_asy}
\psi_{\beta}\lb\zeta\rb e^{\frac{2}{3}\zeta ^{3/2}\sigma_3}=\frac{e^{\frac{\pi i}{12}}}{2\sqrt{\pi}}\begin{pmatrix}
\zeta^{-\frac{1}{4}} &0\\0& \zeta^{\frac{1}{4}}
\end{pmatrix}\begin{pmatrix}
1&1\\-1&1
\end{pmatrix}e^{-\frac{\pi i}{4}\sigma_3}\left[ I+\mathcal{O}\lb |\zeta|^{-1}\rb\right]  \text{ for } |\zeta|\to\infty
\end{align}
with a uniform bound $\mathcal{O}(|\zeta|^{-1})$ for $\beta$ in compact subsets of $(\frac{\pi}{3},\pi)$. Setting $\zeta=f_N\lb z\rb$ and recalling the choice $\beta_{\delta}^+\in [\frac{5\pi}{8},\frac{7\pi}{8}]$ for $\beta$, we obtain from the boundedness of $M,M^{-1}$ on $\Sigma_3^+$ (uniform in $\delta$ for $\delta$ in compact subsets of $\lb 0,\sigma_0\right]$) that
\begin{align*}
\Vert \Delta_R\Vert_{L^{\infty}\lb\Sigma_3^+\rb}=\mathcal{O}\lb \frac{1}{N}\left\| \lb \gamma^+ f\rb ^{-3/2}\right\| _{L^{\infty}\lb\Sigma_3^+\rb}\rb
=\mathcal{O}\lb\frac{1}{N}\lb \frac{4}{5}\gamma^+\delta\rb ^{-3/2}\rb,
\end{align*}
where we have also used Lemma \ref{lemma_f_hat} (iii). In the situation of Lemma \ref{Lemma_R} (b) we appeal to 
Lemma \ref{lemma_G_V} to justify that $\gamma^+$ has a lower bound uniform in $V\in\mathcal{U}$. In summary,
the bound $\Vert \Delta_R\Vert_{L^{\infty}( \Sigma_3^+)}=\mathcal{O}( N^{-1})$ has been established with the desired uniformity properties. A similar argument
proves $\Vert \Delta_R\Vert _{L^{\infty}( \Sigma_3^-)}=\mathcal{O}( N^{-1})$.

\end{proof}

\section*{Acknowledgements}
The first author is grateful to Leonid Pastur for good advice and for sharing his deep insights into the field of Mathematical Physics on a number of occasions. The fourth author would like to thank Leonid Pastur for several helpful suggestions while he was preparing his PhD thesis. We are grateful to the referee and to Gernot Akemann for useful remarks.

 % BIBILOGRAPHY
\bibliographystyle{plain}
\bibliography{arxiv_update}
\end{document}